\documentclass[article]{elsarticle}
\usepackage{amsmath,amssymb,amsthm,enumerate,amsfonts,setspace}
\usepackage{geometry}
\usepackage {multirow}
\usepackage{tabularx}
\usepackage{makecell}
\usepackage{graphicx, subfigure}
\usepackage[ruled]{algorithm2e}
\usepackage{float} 

\usepackage{xcolor}
\usepackage{bm}
\usepackage[normalem]{ulem} 
\newcommand\hl{\bgroup\markoverwith
	{\textcolor[rgb]{0.8,0.8,0.8}{\rule[-.5ex]{1pt}{2.5ex}}}\ULon}
\usepackage[font={bf},textfont=md,labelfont={footnotesize},labelformat={default}]{caption}
\usepackage{caption}
\usepackage[colorlinks,bookmarksopen,bookmarksnumbered,citecolor=black, linkcolor=black, urlcolor=black]{hyperref}
 
\newcommand{\nm}[1]{\left\lVert {#1} \right\rVert}

\newcommand{\dual}[1]{\left\langle {#1} \right\rangle}

\journal{Journal of \LaTeX\ Temlates}

\newtheorem{definition}{\textbf{Definition}}[section]
\newtheorem{theorem}{\textbf{Theorem}}[section]
\newtheorem{assumption}{\textbf{Assumption}}[section]
\newtheorem{lemma}{\textbf{Lemma}}[section]
\newtheorem{remark}{\textbf{Remark}}[section]

\newtheorem{proposition}{\textbf{Proposition}}[section]
\newtheorem{example}{\textbf{Example}}[section]

\begin{document}

\begin{frontmatter}

\title{First-order Methods for Unconstrained Vector Optimization Problems: A Unified Majorization-Minimization Perspective}

\author[mymainaddress,mysecondaryaddress]{Jian Chen}
\author[mysecondaryaddress]{Jinjie Liu}
\author[mysecondaryaddress]{Liping Tang}
\author[mysecondaryaddress]{Xinmin Yang\corref{mycorrespondingauthor}}
\cortext[mycorrespondingauthor]{Corresponding author.\\ \indent \indent  Email addresses: \href{mailto:chenjian_math@163.com}{chenjian\_math@163.com} (Jian Chen),
	\href{mailto:jinjie.liu@cqnu.edu.cn}{jinjie.liu@cqnu.edu.cn} (Jinjie Liu), \href{mailto:tanglipings@163.com}{tanglipings@163.com} (Liping Tang), \href{mailto:xmyang@cqnu.edu.cn}{xmyang@cqnu.edu.cn} (Xinmin Yang)}

\address[mymainaddress]{College of Mathematics, Sichuan University, Chengdu 610065, China}
\address[mysecondaryaddress]{National Center for Applied Mathematics in 
	Chongqing, Chongqing Normal University, Chongqing 401331, China}
\begin{abstract}
In this paper, we develop a unified majorization-minimization scheme and convergence analysis with first-order surrogate functions for unconstrained vector optimization problems (VOPs). By selecting different surrogate functions, the unified method can be reduced to various existing first-order methods. The unified convergence analysis reveals that the slow convergence of the steepest descent method is primarily attributed to the significant gap between the surrogate and objective functions. Consequently, narrowing this surrogate gap can enhance the performance of first-order methods for VOPs. To strike a better trade-off in terms of surrogate gap and per-iteration cost, we reformulate the direction-finding subproblem and elucidate that selecting a tighter surrogate function is equivalent to using an appropriate base of the dual cone in the direction-finding subproblem. Building on this insight, we employ the Barzilai-Borwein method to narrow the surrogate gap and propose a Barzilai-Borwein descent method for VOPs (BBDVO) with polyhedral cones. By reformulating the corresponding subproblem, we provide a novel perspective on the Barzilai-Borwein descent method, bridging the gap between this method and the steepest descent method. Finally, several numerical experiments are presented to validate the efficiency of the BBDVO.
\end{abstract}


\end{frontmatter}


\section{Introduction}
This paper focuses on the following unconstrained vector optimization problem:
\begin{align*}
	{\rm min}_{K}~F(x), \tag{VOP}\label{VOP}
\end{align*}
where $F:\mathbb{R}^{n}\rightarrow\mathbb{R}^{m}$ is to be optimized under the partial order induced by a closed, {convex}, and pointed cone $K\subset\mathbb{R}^{m}$ with a non-empty interior, defined as follows:
$$y\preceq_{K}(\mathrm{resp.} \prec_{K})y^{\prime}\ \Leftrightarrow\ y^{\prime}-y\in K (\mathrm{resp.}\ \mathrm{int}(K)).$$
Let $K^{*}=\{c^{*}\in\mathbb{R}^{m}:\langle c^{*},y\rangle\geq0, \forall y\in K\}$ be the positive polar cone of $K$, and $C$ be a compact base of $K^{*}$, namely, $C$ is a convex set such that $0\notin \text{cl}C$ and $\mathrm{cone}(C)=K^{*}$. 
In vector optimization, it is often impossible to improve all objectives simultaneously with respect to the partial order. Therefore, the concept of optimality is defined as \textit{efficiency} \citep{J2011}, meaning that there is no better solution for an efficient solution. Specifically, the problem (\ref{VOP}) corresponds to a multiobjective optimization problem when $K=\mathbb{R}^{m}_{+}$, where $\mathbb{R}^{m}_{+}$ denotes the non-negative orthant of $\mathbb{R}^{m}$. Various applications of multiobjective optimization problems (MOPs) can be found in engineering \citep{MA2004}, economics \citep{TC2007,FW2014}, management science \citep{E1984}, environmental analysis \citep{LW1992}, machine learning \citep{SK2018,YL2021}, etc.  Although many real-world problems reformulated as vector-valued problems adhere to the partial order induced by $\mathbb{R}^{m}_{+}$, some applications, such as portfolio selection in securities markets \citep{VOP1,VOP2}, require partial orders induced by closed convex cones other than the non-negative orthant. Consequently, vector optimization problems (VOPs) have garnered significant attention in recent years.
\par Over the past two decades, descent methods have received increasing attention within the multiobjective optimization community, primarily due to the seminal work on the steepest descent method proposed by \citet{FS2000}. Inspired by Fliege and Svaiter's contributions, researchers have extended other numerical algorithms to solve multiobjective optimization problems (MOPs) \citep[see, e.g.,][]{FD2009,QG2011,P2014,FV2016,CL2016,MP2018,MP2019,TFY2019}.
 To the best of our knowledge, the study of descent methods for unconstrained vector optimization problems can be traced back to the work of \citet{GS2005}, who extended the steepest descent method for MOPs (SDMO) proposed by \citet{FS2000} to VOPs. In this context, the direction-finding subproblem at $x^{k}$ is formulated as follows:
 \begin{align*}
 	\min\limits_{d\in\mathbb{R}^{n}} \max\limits_{c^{*}\in C}\dual{c^{*}, JF(x^{k})d}+\frac{1}{2}\|d\|^{2},
 \end{align*}
where $JF(x^{k})\in\mathbb{R}^{m\times n}$ is the Jacobian matrix of $F(\cdot)$ at $x^{k}$. Similar to MOPs, several standard numerical algorithms have been extended to VOPs, including the Newton method \citep{GRS2014}, projected gradient method \citep{GI2004}, proximal point method \citep{BI2005}, conjugated gradient method \citep{LP2018} and conditional gradient method \citep{CYZ2023}. 
\par In recent years, complexity analysis of descent methods for MOPs has been extensively studied. \citet{FVV2019} and \citet{ZDH2019} established the convergence rates of SDMO under different convexity assumptions. \citet{TFY2023c} developed convergence results for the multiobjective proximal gradient method. Additionally, \citet{L2024} studied the complexity of a wide class of multiobjective descent methods with nonconvex assumption. However, \citet{CTY2023b} noted that both theoretical and empirical results indicate that existing multiobjective first-order methods exhibit slow convergence due to objective imbalances. To address this challenge, \citet{CTY2023a} proposed a Barzilai-Borwein descent method for MOPs (BBDMO) that dynamically tunes gradient magnitudes using Barzilai-Borwein's rule \citep{BB1988} in direction-finding subproblem. From a theoretical perspective, an improved linear convergence rate is confirmed by \citet{CTY2023b}, demonstrating that Barzilai-Borwein descent methods can effectively mitigate objective imbalances. 
\par Despite the extensive study of complexity analysis for MOPs, corresponding results have received little attention in VOPs. As described by \citet{CTY2022b}, the linear convergence rates of first-order descent methods for VOPs are essentially affected by $C$, which represents the base of $K^{*}$ in direction-finding subproblem. In general, $K^{*}$ admits infinitely many possible bases, and the choice of $C$ is therefore not unique. Although this choice critically influences the search direction and ultimately the convergence behavior of the algorithm, the existing literature offers no comprehensive theoretical analysis or guiding principle for selecting such a base. In the classical setting of multiobjective optimization, the subproblem of SDMO can be reformulated as
\begin{align*}
	\min\limits_{d\in\mathbb{R}^{n}} \max\limits_{\lambda\in\Delta_{m}}\dual{\lambda, JF(x^{k})d}+\frac{1}{2}\|d\|^{2},
\end{align*}
where $\Delta_{m}:=\{\lambda\succeq0:\sum_{i=1}^{m}\lambda_{i}=1\}$ is a base of $\mathbb{R}^{m}_{+}$. While this choice appears natural and convenient from a computational perspective, its theoretical justification and potential advantages over other alternatives remain largely unexplored. Hence, identifying an appropriate and theoretically sound base for $K^{*}$ constitutes a fundamental yet open problem, which is of both theoretical significance and practical relevance in the design of efficient first-order methods for VOPs. This naturally leads to the following question:
\begin{equation}\tag{Q}\label{Q}
\textit{Can we provide a theoretical guidance for the choice of the base?}
\end{equation}
\par  To answer the proposed questions, we develop a unified framework and convergence analysis of first-order methods for VOPs from a majorization-minimization perspective. The majorization-minimization principle is a versatile tool for designing novel algorithms, which has been successful employed in nonconvex optimization \citep{LMS2017}, constrained optimization \citep{Landeros2023MM}, and incremental optimization \citep{M2015,KWM2022}. The core idea behind the majorization-minimization method is to minimize a difficult optimization problem
by iteratively minimizing a simpler surrogate function that majorizes (upper-bounds) the original objective. In this paper, we extend majorization-minimization method to vector optimization, addressing the aforementioned questions. Our work aims to provide a unified framework and convergence analysis of first-order methods for VOPs from a majorization-minimization perspective. The primary contributions of this paper are summarized as follows:
\par (i) Before presenting the majorization-minimization method and its convergence analysis, we first define the concepts of strong convexity and smoothness for a vector-valued function with respect to a partial order. Leveraging these properties, we extend the notion of the condition number to VOPs, which plays a pivotal role in establishing the linear convergence of first-order methods for VOPs. To the best of our knowledge, it is the first definition of condition number to VOPs.
 \par (ii) We devise a unified majorization-minimization descent method for VOPs and develop its convergence analysis. By selecting different surrogate functions, the unified method can be reduced to several existing first-order methods. 	It is worth noting that the gap between the surrogate and objective functions significantly affects the performance of descent methods, which plays a central role in majorization-minimization optimization. Specifically, the steepest descent method for VOPs exhibits slow convergence due to the large gap between the surrogate and objective functions. To address this issue, we develop an improved descent method with a tighter surrogate function, resulting in improved linear convergence, and the rate of convergence is determined by the condition number. Interestingly, we show that selecting a tighter surrogate function is equivalent to using an appropriate base in the direction-find subproblem (see Remarks \ref{r4.4} and \ref{r5.3}). This provides a positive answer to the proposed question. 
 \par (iii) Theoretical results suggest a tighter surrogate function by using Barzilai-Borwein method, which motivates us to devise a Barzilai-Borwein descent method for VOPs (BBDVO) with polyhedral cones. By reformulating the subproblem, we observe that BBDVO is essentially the steepest descent method with an appropriately chosen base in the direction-finding subproblem (see Remark \ref{r6.3}). Furthermore, a VOP with a polyhedral cone can be transformed into an MOP using a transform matrix, which is often used to define a specific polyhedral cone. We demonstrate that the performance of BBDVO is insensitive to the choice of the transform matrix, whereas the steepest descent method is highly sensitive to it. From a majorization–minimization perspective, we further elucidate why the line search procedure in VOPs leads to a slower linear convergence rate compared with its counterpart without line search. In contrast, the backtracking strategy preserves the same linear convergence rate as first-order methods for VOPs without line search.
\par The remainder of this paper is organized as follows. Section \ref{sec2} introduces the necessary notations and definitions for later use. In Section \ref{sec3}, we present a generic majorization-minimization descent method for VOPs and analyze its convergence rates under various convexity assumptions. Sections \ref{sec4} and \ref{sec5} explore the connections between different descent methods from the perspective of majorization-minimization. Section \ref{sec6} proposes a Barzilai-Borwein descent method for VOPs with a polyhedral cone. 
 Section \ref{sec7} provides numerical results to demonstrate the efficiency of the BBDVO. Finally, conclusions are presented at the end of the paper.

\section{Preliminaries}\label{sec2}
Throughout this paper, $\mathbb{R}^{n}$ and $\mathbb{R}^{m\times n}$ denote the set of $n$-dimensional real column vectors and the set of $m\times n$ real matrices, respectively. The space $\mathbb{R}^{n}$ is equipped with the inner product $\langle\cdot,\cdot\rangle$ and the induced norm $\|\cdot\|$. The interior, boundary and the closure of a set are denoted by $\mathrm{int}(\cdot)$, $\mathrm{bd}(\cdot)$ and $\mathrm{cl}(\cdot)$, respectively. The cone generated by a set is denoted by $\textrm{cone}(\cdot)$.  For simplicity, we denote $[m]:=\{1,2,...,m\}$, $\mathbf{1}_{m}$ and $I_{m}$ the all-ones vector in $\mathbb{R}^{m}$ and identity matrix in $\mathbb{R}^{m\times m}$, respectively.
\par Since $K$ is a closed convex cone, it follows that $K=K^{**}$ (see \citep[Theorem 14.1]{R1970}),
$$K=\{y\in\mathbb{R}^{m}:\langle y,c^{*}\rangle\geq0, \forall c^{*}\in K^{*}\},$$
and $$\textrm{int}(K)=\{y\in\mathbb{R}^{m}:\langle y,c^{*}\rangle>0, \forall c^{*}\in K^{*}\setminus\{0\}\}.$$
Since $\textrm{int}(K)\neq\emptyset$, we assume that there exists a compact and convex set $C$ such that
\begin{align}
	0 &\notin C, \label{eq:notinG}\\
	\operatorname{cone}(C) &= K^*. \label{eq:coneG}
\end{align}
Therefore
\begin{align}
	K=\{y\in\mathbb{R}^{m}:\langle y,c^{*}\rangle\geq0, \forall c^{*}\in C\}, \label{G1}\\
	\textrm{int}(K)=\{y\in\mathbb{R}^{m}:\langle y,c^{*}\rangle>0, \forall c^{*}\in C\}. \label{G3}
\end{align}
The latter equality, together with the compactness of $C$, implies that
\begin{equation}\label{G2}
	\min_{c^*\in C}\{\dual{c^*,y}\}>0, ~~~\forall y\in \mathrm{int}(K).
\end{equation}
For more details on $C$, we refer the readers to \citep[pp. 400]{GS2005}.
\subsection{Vector optimization}
In the subsection, we revisit some definitions and results pertinent to VOPs. Firstly, we introduce the concept of efficiency.

\begin{definition}{\rm\citep[Definition 11.3]{J2011}}
	A vector $x^{*}\in\mathbb{R}^{n}$ is called efficient solution to {\rm(\ref{VOP})} if there exists no $x\in\mathbb{R}^{n}$ such that $F(x)\preceq_{K} F(x^{\ast})$ and $F(x)\neq F(x^{\ast})$.
\end{definition}

\begin{definition}{\rm\citep[Definition 11.5]{J2011}}\label{we}
	A vector $x^{*}\in\mathbb{R}^{n}$ is called weakly efficient solution to {\rm(\ref{VOP})} if there exists no $x\in\mathbb{R}^{n}$ such that $F(x)\prec_{K} F(x^{\ast})$.
\end{definition}

\begin{definition}{\rm\citep{GS2005}}\label{st}
	A vector $x^{\ast}\in\mathbb{R}^{n}$ is called $K$-stationary point to {\rm(\ref{VOP})} if
	$$\mathrm{range}(JF(x^{*}))\cap(-\mathrm{int}(K))=\emptyset,$$
	where $\mathrm{range}(JF(x^{*}))$ denotes the range of linear mapping given by the matrix $JF(x^{*})$.
\end{definition}

\begin{definition}{\rm\citep{GS2005}}\label{cd}
	A vector $d\in\mathbb{R}^{n}$ is called $K$-descent direction for $F(\cdot)$ at $x$ if
	$$JF(x)d\in-\mathrm{int}(K).$$
\end{definition}
\begin{remark}\label{rdd}	 
	Note that if $x\in\mathbb{R}^{n}$ is a non-stationary point, then there exists a $K$-descent direction $d\in\mathbb{R}^{n}$ such that $JF(x)d\in-\mathrm{int}(K)$. 
\end{remark}

Next, we introduce the concept of $K$-convexity for $F(\cdot)$.
\begin{definition}{\rm\citep[Definition 2.4]{J2011}}
	The objective function $F(\cdot)$ is called $K$-convex if
	$$F(\lambda x+(1-\lambda)y)\preceq_{K}\lambda F(x)+(1-\lambda)F(y)$$
	holds for all $x,y\in\mathbb{R}^{n},\ \lambda\in[0,1]$.
\end{definition}

By the continuous differentiability of $F(\cdot)$, $K$-convexity of $F(\cdot)$ is equivalent to
$$JF(x)(y-x)\preceq_{K}F(y)-F(x)$$
holds for all $x,y\in\mathbb{R}^{n}$. (see \citep[Theorem 2.20]{J2011}).

We conclude this section by elucidating the relationship between $K$-stationary points and weakly efficient solutions.
\begin{lemma}{\rm\citep{J2011}}
	Assume that the objective function $F(\cdot)$ is $K$-convex, then $x^{*}\in\mathbb{R}^{n}$ is a $K$-stationary point of {\rm(\ref{VOP})} if and only if $x^{*}$ is a weakly efficient solution of {\rm(\ref{VOP})}.
\end{lemma}

\subsection{Strong convexity and smoothness}
Strong convexity and smoothness of objective functions play a central role of first-order methods in optimization. This subsection is devoted to strong convexity and smoothness of vector-valued functions under partial order.

\begin{definition}{\rm\citep{GRS2014}}\label{sc}
	The objective function $F(\cdot)$ is called strongly $K$-convex with  $\bm\mu\in K$ if
	$$F(\lambda x+(1-\lambda)y)\preceq_{K}\lambda F(x)+(1-\lambda)F(y)-\frac{1}{2}\lambda(1-\lambda)\nm{x-y}^{2}\bm\mu,\ \forall x,y\in\mathbb{R}^{n},\ \lambda\in[0,1],$$
and the above relation does not hold for any $\hat{\bm\mu}$ with $\hat{\bm\mu}\not\preceq_{K}\bm\mu$.
\end{definition}
\begin{remark}
	Comparing with the definition in \citep{GRS2014}, Definition \ref{sc} includes the case $\bm\mu\in{\rm{bd}}(K)$, then it reduces $K$-convexity when $\bm\mu=0$. Furthermore, tthe final statement in the definition establishes the uniqueness of the parameter $\bm\mu$, which is essential for the convergence analysis.
\end{remark}
\begin{lemma}{\rm\citep[Theorem 3.1]{FD2009}}
	Assume that the objective function $F(\cdot)$ is strongly $K$-convex with $\bm\mu\in {\rm int}(K)$, then $x^{*}\in\mathbb{R}^{n}$ is a $K$-stationary point of {\rm(\ref{VOP})} if and only if $x^{*}$ is an efficient solution of {\rm(\ref{VOP})}.
\end{lemma}
By the continuous differentiability of $F(\cdot)$, strong $K$-convexity of $F(\cdot)$ is equivalent to
$$\frac{1}{2}\nm{x-y}^{2}\bm\mu + JF(x)(y-x)\preceq_{K}F(y)-F(x),\ \forall x,y\in\mathbb{R}^{n},$$
it characterizes a quadratic lower-bound of $F(\cdot)$. Intuitively, we use quadratic upper-bound to define the $K$-smoothness of $F(\cdot)$ under partial order.

\begin{definition}\label{sm}
	The objective function $F(\cdot)$ is called $K$-smooth with $\bm\ell\in{\rm int}(K)$ if
	$$F(y)-F(x)\preceq_{K}JF(x)(y-x)+\frac{1}{2}\nm{x-y}^{2}\bm\ell,\ \forall x,y\in\mathbb{R}^{n},$$
	and the above relation does not hold for any $\hat{\bm\ell}$ with $\bm\ell\not\preceq_{K}\hat{\bm\ell}$.
\end{definition}

\begin{remark}
	Assume that $F(\cdot)$ is strongly $K$-convex with  $\bm\mu\in K$ and $K$-smooth with $\bm\ell\in{\rm int}(K)$, then $\bm\mu\preceq_{K} \bm\ell$.
\end{remark}

\begin{remark}
	Comparing with the smoothness and strong convexity in \citep[Definitions 7 and 8]{CTY2022b} with Euclidean distance, i.e., $\omega(\cdot)=\frac{1}{2}\nm{\cdot}^{2}$, Definitions \ref{sc} and \ref{sm} are tighter and do not depend on the reference vector $e$. 
\end{remark}
Next, we characterize the properties of the difference of two vector-valued functions.
\begin{lemma}[regularity of residual functions]\label{l2.2}
	Let $F,G:\mathbb{R}^{n}\rightarrow\mathbb{R}^{m}$ be two vector-valued functions. Define $H(\cdot):=G(\cdot)-F(\cdot)$. Then the following statements hold.
		\begin{itemize}
		\item[(i)] if $G(\cdot)$ is strongly $K$-convex with $\bm\mu\in {\rm int}(K)$ and $F(\cdot)$ is $K$-smooth with $\bm\ell\in{\rm int}(K)$, where $\bm\ell\preceq\bm\mu$, then $H(\cdot)$ is strongly $K$-convex with $\bm\mu-\bm\ell$;
		\item[(ii)] if $G(\cdot)$ is $K$-smooth with $\bm\ell\in{\rm int}(K)$ and $F(\cdot)$ is $K$-convex, then $H(\cdot)$ is $K$-smooth with $\bm\ell\in{\rm int}(K)$.
		\item[(iii)] if $G(\cdot)$ $K$-smooth with $\bm\ell\in{\rm int}(K)$ and $F(\cdot)$ is strongly $K$-convex with $\bm\mu\in {\rm int}(K)$, where $\bm\mu\preceq\bm\ell$, then $H(\cdot)$ is $K$-smooth with $\bm\ell-\bm\mu$.
	\end{itemize}
\end{lemma}
\begin{proof}
	The proof is a consequence of the definition of strong $K$-convexity and $K$-smoothness, we omit it here.
\end{proof}
In SOPs, the condition number (the quotient of smoothness parameter and the modulus of strong convexity) plays a key role in the geometric convergence of first-order methods. We end this section with the definition of the condition number of a strongly $K$-convex function under partial order.
\begin{definition}
	Assume that $F(\cdot)$ is strongly $K$-convex with $\bm\mu\in{\rm int}(K)$ and $K$-smooth with $\bm\ell\in{\rm int}(K)$. Then, we denote
	\begin{equation}\label{cn}
		\kappa_{F,\preceq_{K}}:=\max_{c^{*}\in C}\frac{\dual{c^{*},\bm\ell}}{\dual{c^{*},\bm\mu}}
	\end{equation}
	 the condition number of $F(\cdot)$ under partial order $\preceq_{K}$.
\end{definition}
\begin{remark}\label{r2.5}
	Notice that $0\notin C$ and $K^{*}=\mathrm{cone}(C)$, the condition number can be rewritten as follows:
	$$\kappa_{F,\preceq_{K}}:=\max_{c^{*}\in K^{*}\setminus\{0\}}\frac{\dual{c^{*},\bm\ell}}{\dual{c^{*},\bm\mu}}.$$ In other words, the condition number of $F(\cdot)$ is determined only by $K$, not $C$.
\end{remark}
\par In the following, we will show that the condition number can be reduced to that of MOPs \citep{CTY2023b}.
\begin{proposition}\label{pcn}
	For any $\bm\ell,\bm\mu\in\mathbb{R}^{m}_{++}$, we have $$\max\limits_{\lambda\in\Delta_{m}}\frac{\sum_{i\in[m]}\lambda_{i}\bm\ell_{i}}{\sum_{i\in[m]}\lambda_{i}\bm\mu_{i}}=\max\limits_{i\in[m]}\frac{\bm\ell_{i}}{\bm\mu_{i}}.$$
\end{proposition}
\begin{proof}
	Since $\bm\ell,\bm\mu\in\mathbb{R}^{m}_{++}$, for any $i\in[m]$, we have 
	$$\bm\ell_{i}\leq \bm\mu_{i}\max\limits_{i\in[m]}\frac{\bm\ell_{i}}{\bm\mu_{i}}.$$
Multiply by $\lambda_{i}\geq0$ and sum over $i\in[m]$:
$$\sum_{i\in[m]}\lambda_{i}\bm\ell_{i}\leq\left(\sum_{i\in[m]}\lambda_{i}\bm\mu_{i}\right)\max\limits_{i\in[m]}\frac{\bm\ell_{i}}{\bm\mu_{i}}.$$
Dividing by the positive $\sum_{i\in[m]}\lambda_{i}\bm\mu_{i}$ yields
$$\frac{\sum_{i\in[m]}\lambda_{i}\bm\ell_{i}}{\sum_{i\in[m]}\lambda_{i}\bm\mu_{i}}\leq\max\limits_{i\in[m]}\frac{\bm\ell_{i}}{\bm\mu_{i}}.$$
Therefore, the relation
$$\max\limits_{\lambda\in\Delta_{m}}\frac{\sum_{i\in[m]}\lambda_{i}\bm\ell_{i}}{\sum_{i\in[m]}\lambda_{i}\bm\mu_{i}}\leq\max\limits_{i\in[m]}\frac{\bm\ell_{i}}{\bm\mu_{i}}$$
holds due to the arbitrary of $\lambda$. Let $s$ be an index where the maximum ratio is attained, i.e. $$\frac{\bm\ell_{s}}{\bm\mu_{s}}=\max\limits_{i\in[m]}\frac{\bm\ell_{i}}{\bm\mu_{i}}.$$
Take $\lambda_{s}=1$ and $\lambda_{i}=0$ for $i\neq s$, we have
$$\max\limits_{\lambda\in\Delta_{m}}\frac{\sum_{i\in[m]}\lambda_{i}\bm\ell_{i}}{\sum_{i\in[m]}\lambda_{i}\bm\mu_{i}}\geq\frac{\bm\ell_{s}}{\bm\mu_{s}}=\max\limits_{i\in[m]}\frac{\bm\ell_{i}}{\bm\mu_{i}}.$$
This completes the proof.
\end{proof}

\begin{remark}
	If $K = \mathbb{R}^{m}_{+}$, then $\kappa_{F,\preceq_{K}}=\max_{\lambda\in\Delta_{m}}{\sum_{i\in[m]}\lambda_{i}\bm\ell_{i}}/{\sum_{i\in[m]}\lambda_{i}\bm\mu_{i}}$. By Proposition \ref{pcn}, it follows that $\kappa_{F,\preceq_{K}}=\max_{i\in[m]}{\bm\ell_{i}}/{\bm\mu_{i}}$. In other words, the condition number for multiobjective optimization is the largest condition number among all objective functions.
	
\end{remark}
\begin{proposition}\label{p2.1}
	 Assume that $F(\cdot)$ is strongly $K_{1}$-convex with $\bm\mu_{1}\in{\rm int}(K_{1})$ and $K_{1}$-smooth with $\bm\ell_{1}\in{\rm int}(K_{1})$. Then, for any order cone $K_{2}$ satisfied $K_{1}\subset K_{2}$, we have $F(\cdot)$ is strongly $K_{2}$-convex with $\bm\mu_{2}\in{\rm int}(K_{2})$ and $K_{2}$-smooth with $\bm\ell_{2}\in{\rm int}(K_{2})$. Futhermore, $\bm\ell_{2}\preceq_{K_{2}}\bm\ell_{1}$, $\bm\mu_{1}\preceq_{K_{2}}\bm\mu_{2}$, and $\kappa_{F,\preceq_{K_{2}}}\leq\kappa_{F,\preceq_{K_{1}}}$.
\end{proposition}
\begin{proof}
	The proof is a consequence of the definitions of strong $K$-convexity, $K$-smoothness and condition number, we omit it here.
\end{proof}
\begin{remark}
	Proposition \ref{p2.1} shows that, for a fixed vector-valued function, enlarging the underlying order cone effectively simplifies the associated vector optimization problem; see Lemma \ref{l4.2} for a formal statement. This observation motivates the use of a larger order cone to accelerate first-order methods for VOPs.
\end{remark}
\section{Majorization-minimization with first-order surrogate functions for VOPs}\label{sec3}
\subsection{Majorization-minimization descent method for VOPs}
In this section, we present a unified majorization-minimization scheme for minimizing a vector-valued function in the sense of descent. 

\begin{algorithm}
	\caption{Unified majorization-minimization scheme for VOPs }\label{mm}
	\LinesNumbered 
	\KwData{$x^{0}\in\mathbb{R}^{n}$}
	\For{$k=0,1,...$}{  Choose a strongly $K$-convex surrogate function $G_{k}(\cdot)$ of $F(\cdot)-F(x^{k})$ near $x^{k}$\\
		Choose a base $C_{k}$ of dual cone $K^*$\\
		Update $x^{k+1}:=\mathop{\arg\min}_{x\in\mathbb{R}^{n}}\max_{c^{*}\in C_{k}}\dual{c^{*},G_{k}(x)}$\\
	\If{$x^{k+1}=x^{k}$}{ {\bf{return}} $K$-stationary point $x^{k}$ }}  
\end{algorithm}

\begin{remark}
	It is worth noting that we choose a variable base $C_k$ in each iteration, whreas an invariable base $C$ is used in existing descent methods for VOPs, see \citep{GS2005,GRS2014,GI2004,BI2005,LP2018,CYZ2023}. 
\end{remark}

The surrogate function $G_{k}(\cdot)$ plays a central role in the generic majorization-minimization scheme. Intuitively, $G_{k}(\cdot)$ should well approximate $F(\cdot)-F(x^{k})$ near $x^{k}$ and the related subproblem should be easy to minimize. Therefore, we measure the approximation error by $H_{k}(\cdot):=G_{k}(\cdot) - F(\cdot)+F(x^{k})$. To characterize surrogates, we introduce a class of surrogate functions, which will be used to establish the convergence results of Algorithm \ref{mm}. 
\begin{definition}
	For $x^{k}\in\mathbb{R}^{n}$, we call $G_{k}(\cdot)$ a first-order surrogate function of $F(\cdot)-F(x^{k})$ near $x^{k}$ when
	\begin{itemize}
		\item[(i)] $F(x^{k+1})-F(x^{k})\preceq_{K}G_{k}(x^{k+1})$, where $x^{k+1}$ is the minimizer of $\min_{x\in\mathbb{R}^{n}}\max_{c^{*}\in C_{k}}\dual{c^{*},G_{k}(x)}$, furthermore, when $F(\cdot)-F(x^{k})\preceq_{K}G_{k}(\cdot)$ for all $x\in\mathbb{R}^{n}$, we call $G_{k}(\cdot)$ a majorizing surrogate;
		\item[(ii)] the approximation error $H_{k}(\cdot)$ is $K$-smooth with $\bm \ell\in{\rm int}(K)$, $H_{k}(x^{k})=0$, and $JH_{k}(x^{k})=0$.
	\end{itemize}
	We denote by $\mathcal{S}_{\bm \ell,\bm\mu}(F,x^{k})$ the set of first-order strongly $K$-convex surrogate functions with $\bm\mu\in{\rm int}(K)$.
\end{definition}
Next, we characterize the properties of first-order surrogate functions.
\begin{lemma}\label{l3.1}
	Let $G_{k}(\cdot)\in\mathcal{S}_{\bm \ell,\bm\mu}(F,x^{k})$ and $x^{k+1}$ be the minimizer of $\min_{x\in\mathbb{R}^{n}}\max_{c^{*}\in C_{k}}\dual{c^{*},G_{k}(x)}$. Then, for all $x\in\mathbb{R}^{n}$, we have
	\begin{itemize}
		\item[(i)] $H_{k}(x)\preceq_{K}\frac{1}{2}\nm{x-x^{k}}^{2}\bm\ell$;
		\item[(ii)] $\dual{c^{*}_{k},F(x^{k+1})}+\frac{1}{2}\nm{x^{k+1}-x}^{2}\dual{c^{*}_{k},\bm\mu}\leq\dual{c^{*}_{k},F(x)}+\frac{1}{2}\nm{x^{k}-x}^{2}\dual{c^{*}_{k},\bm\ell}$, where $c^{*}_{k}$ is a maximizer of $\max_{c^{*}\in C_{k}}\min_{x\in\mathbb{R}^{n}}\dual{c^{*},G_{k}(x)}$.
	\end{itemize}
\end{lemma}
\begin{proof}
	Assertion (i) directly follows by the $K$-smoothness of $H_{k}(\cdot)$ and the facts that $H_{k}(x^{k})=0$ and $JH_{k}(x^{k})=0$. Next, we prove the assertion (ii). By Sion's minimax theorem \citep{S1958}, by denoting
	$c^{*}_{k}$ a maximizer of $\max_{c^{*}\in C_{k}}\min_{x\in\mathbb{R}^{n}}\dual{c^{*},G_{k}(x)}$, and $x^{k+1}$ the minimizer of $\min_{x\in\mathbb{R}^{n}}\max_{c^{*}\in C_{k}}\dual{c^{*},G_{k}(x)}$, we have $JG_{k}(x^{k+1})^{T}c^{*}_{k}=0$. This, together with the strong $K$-convexity of $G_{k}(\cdot)$, implies that
	$$\dual{c^{*}_{k},G_{k}(x^{k+1})}+\frac{1}{2}\nm{x^{k+1}-x}^{2}\dual{c^{*}_{k},\bm \mu}\leq\dual{c^{*}_{k},G_{k}(x)},~\forall x\in\mathbb{R}^{n}.$$
	We thus use $F(x^{k+1})-F(x^{k})\preceq_{K}G_{k}(x^{k+1})$ to get
	\begin{align*}
		\dual{c^{*}_{k},F(x^{k+1})-F(x^{k})}+\frac{1}{2}\nm{x^{k+1}-x}^{2}\dual{c^{*}_{k},\bm\mu}&\leq\dual{c^{*}_{k},G_{k}(x^{k+1})}+\frac{1}{2}\nm{x^{k+1}-x}^{2}\dual{c^{*}_{k},\bm \mu}\\
		&\leq\dual{c^{*}_{k},G_{k}(x)}\\
		&=\dual{c^{*}_{k},F(x)-F(x^{k})} +\dual{c^{*}_{k},H_{k}(x)}\\
		&\leq\dual{c^{*}_{k},F(x)-F(x^{k})}+\frac{1}{2}\nm{x^{k}-x}^{2}\dual{c^{*}_{k},\bm\ell},
	\end{align*}
where the equality follows by the definition of $H_{k}(\cdot)$ and the last inequality is due to the assertion (i). This completes the proof.
\end{proof}
\subsection{Convergence analysis}
In Algorithm \ref{mm}, it can be observed that it terminates either with a $K$-stationary point in a finite number of iterations or generates an infinite sequence of non-stationary points. In the subsequent analysis, we will assume that the algorithm produces an infinite sequence of non-stationary points. 
\subsubsection{Global convergence}
Firstly, we establish the global convergence result in the nonconvex setting, the following assumptions are required.
\begin{assumption}\label{a1}
	Assume the following statements hold for $F(\cdot)$ and $G_{k}(\cdot)$:
	\begin{itemize}
		\item[(i)] The level set $\mathcal{L}_{F}(x^{0}):=\{x:F(x)\preceq_{K}F(x^{0})\}$ is bounded;
		\item[(ii)] There exist two compact bases $C_{L}$ and $\tilde{C}$ of $K^*$ such that
		\begin{center}
			$C_{k}\subset \tilde{C}$, and $\max_{c^*\in C_{k}}\dual{c^*,y}\geq\max_{c^*\in C_{L}}\dual{c^*,y}$
		\end{center}
		hold for all $k\geq0$ and $y\in-K$. 
		\item[(iii)] If $C$ is a compact base of $K^*$, $x^{k}\rightarrow x^{*}$, $G_{k}(\cdot)\in\mathcal{S}_{\bm \ell,\bm\mu}(F,x^{k})$ and $\min_{y\in\mathbb{R}^n}\max_{c^{*}\in C}\dual{c^{*},G_{k}(y)}\rightarrow0$, then $x^*$ is a $K$-stationary point to  (\ref{VOP}).
	\end{itemize}
\begin{remark}
	Assumption \ref{a1}(i) is a standard condition for nonconvex cases. Moreover, Assumption \ref{a1}(ii) requires the sequence $\{C_{k}\}$ to be uniformly bounded, which is a mild assumption. Assumption \ref{a1}(iii) holds for  $K$-steepest descent method \citep{GS2005}. Specifically, $\min_{y\in\mathbb{R}^n}\max_{c^{*}\in C}\dual{c^{*},G_{k}(y)}\rightarrow0$ takes the form of $\alpha_{x_k}$ \citep[Definition 3.2]{GS2005}, where $\alpha_{x}$ is continuous \citep[Lemma 3.3(3)]{GS2005} and $\alpha_{x}=0$ if and only if $x$ is a $K$-stationary point to (\ref{VOP}) \citep[Lemma 3.3(1)]{GS2005}.  
	
\end{remark}
\end{assumption}
We are now in a position to establish the global convergence of Algorithm \ref{mm}.
\begin{theorem}\label{t1}
	Suppose that Assumption \ref{a1} holds, let $\{x^{k}\}$ be the sequence generated by Algorithm \ref{mm} with $G_{k}(\cdot)\in\mathcal{S}_{\bm \ell,\bm\mu}(F,x^{k})$. Then, $\{x^{k}\}$ has at least one accumulation point and every accumulation point is a non-stationary point to (\ref{VOP}).
\end{theorem}
\begin{proof}
	Since $G_{k}(\cdot)\in\mathcal{S}_{\bm \ell,\bm\mu}(F,x^{k})$, we have 
	\begin{equation}\label{e1}
		F(x^{k+1})-F(x^{k})\preceq_{K} G_{k}(x^{k+1}),
	\end{equation}
	and 
	$$G_{k}(x^{k+1})\preceq_{K}G_{k}(x^{k})=H_{k}(x^{k})=0.$$ 
	Then, we conclude that $\{F(x^{k})\}$ is decreasing under partial order $\preceq_{K}$. It follows by Assumption \ref{a1}(i) and continuity of $F(\cdot)$ that $\{x^{k}\}$ is bounded and there exists $F^{*}$ such that
	$$F^{*}\preceq_{K}F(x^{k}),~\forall k\geq0.$$
	The boundedness of $\{x^{k}\}$ indicates that $\{x^{k}\}$ has at least one accumulation point. Next, we prove that any accumulation point $x^{*}$ is a non-stationary point.
	By summing (\ref{e1}) from $0$ to infinity, we have
	$$F^{*}-F(x^{0})\preceq_{K}\sum_{k=0}^{\infty}(F(x^{k+1})-F(x^{k}))\preceq_{K}\sum_{k=0}^{\infty}G_{k}(x^{k+1}).$$
	It follows that
	\begin{align*}
		\sum_{k=0}^{\infty}\max_{c^{*}\in {C}_k}\dual{c^{*},G_{k}(x^{k+1})}&\geq\sum_{k=0}^{\infty}\max_{c^{*}\in {C}_L}\dual{c^{*},G_{k}(x^{k+1})}\\
		&\geq\max_{c^{*}\in C_{L}}\dual{c^{*},\sum_{k=0}^{\infty}G_{k}(x^{k+1})}\\
		&\geq\max_{c^{*}\in C_L}\dual{c^{*},F^{*}-F(x^{0})}\\
		&\geq -\infty,
	\end{align*}
	where the first inequality follows by Assumption \ref{a1}(ii) and $G_{k}(x^{k+1})\preceq_{K} 0$, the second inequality is due to the fact $\max_{x}{f_1(x)}+\max_{x}{f_2(x)}\geq \max_{x}\{f_1(x)+f_2(x)\}$.
	This, together with the fact that $G_{k}(x^{k+1})\preceq_{K}G_{k}(x^{k})=0$, implies $\max_{c^{*}\in C_{k}}\dual{c^{*},G_{k}(x^{k+1})}\rightarrow0$. A direct calculation gives:
	\begin{align*}
		0 = \max_{c^{*}\in \tilde{C}}\dual{c^{*},G_{k}(x^{k})} \geq\min_{y\in\mathbb{R}^n}\max_{c^{*}\in \tilde{C}}\dual{c^{*},G_{k}(y)} \geq \min_{y\in\mathbb{R}^n}\max_{c^{*}\in C_k}\dual{c^{*},G_{k}(y)} = \max_{c^{*}\in C_{k}}\dual{c^{*},G_{k}(x^{k+1})}\rightarrow0,
	\end{align*}
where the second inequality follows by $C_{k}\subset \tilde{C}$.
Therefore, $\min_{y\in\mathbb{R}^n}\max_{c^{*}\in \tilde{C}}\dual{c^{*},G_{k}(y)}\rightarrow0$. For the accumulation point $x^{*}$, there exists an infinite index set $\mathcal{K}$ such that $x^{k}\stackrel{\mathcal{K}}{\longrightarrow}x^{*}$. By Assumption \ref{a1}(iii), we conclude that $x^{*}$ is a $K$-stationary point.	
\end{proof}
\subsubsection{Strong convergence}
In the following, we establish the strong convergence result of Algorithm \ref{a1} in $K$-convex setting.
\begin{theorem}\label{t2}
	Suppose that Assumption \ref{a1} holds and $F(\cdot)$ is $K$-convex, let $\{x^{k}\}$ be the sequence generated by Algorithm \ref{mm} with $G_{k}(\cdot)\in\mathcal{S}_{\bm \ell,\bm\mu}(F,x^{k})$ and $0\preceq_{K}\bm\ell=\bm\mu$. Then, the following statements hold:
	\begin{itemize}
		\item[(i)] $\{x^{k}\}$ converges to a weakly efficient solution $x^{*}$ of (\ref{VOP});
		\item[(ii)] $u_{0}(x^{k})\leq\frac{\ell_{\max}R^{2}}{2k},~\forall k\geq1$, where $\ell_{\max}:=\max_{c^{*}\in \tilde{C}}\dual{c^{*},\bm\ell}$, $R:=\{\nm{x-y}:x,y\in\mathcal{L}_{F}(x^{0})\}$, and 
	{$$u_{0}(x^{k}):=\max_{x\in\mathbb{R}^{n}}\min_{c^{*}\in \tilde{C}}\dual{c^{*},F(x^{k})-F(x)}$$}
		is a merit function in the sense of weak efficiency.
	\end{itemize}
\end{theorem}
\begin{proof}
	(i) By the similar arguments in the proof of Theorem \ref{t1}, we conclude that $\{x^{k}\}$ is bounded, and there exists a $K$-stationary point $x^{*}$ such that $F(x^{*})\preceq_{K}F(x^{k})$. Besides, the $K$-convexity of $F(\cdot)$ indicates that $x^{*}$ is a weakly efficient point. From Lemma \ref{l3.1}(ii), for any $x\in\mathbb{R}^{n}$ we have
	\begin{equation}\label{e3}
		\dual{c^{*}_{k},F(x^{k+1})-F(x)}\leq\frac{1}{2}\nm{x^{k}-x}^{2}\dual{c^{*}_{k},\bm\ell}-\frac{1}{2}\nm{x^{k+1}-x}^{2}\dual{c^{*}_{k},\bm\mu}.
	\end{equation} 
Substituting $x=x^{*}$ into the above inequality, we obtain
$$\dual{c^{*}_{k},F(x^{k+1})-F(x^{*})}\leq\frac{1}{2}\nm{x^{k}-x^{*}}^{2}\dual{c^{*}_{k},\bm\ell}-\frac{1}{2}\nm{x^{k+1}-x^{*}}^{2}\dual{c^{*}_{k},\bm\mu}.$$
Recall that $F(x^{*})\preceq_{K}F(x^{k})$, it follows that
$$\nm{x^{k+1}-x^{*}}^{2}\dual{c^{*}_{k},\bm\mu}\leq\nm{x^{k}-x^{*}}^{2}\dual{c^{*}_{k},\bm\ell}.$$
Furthermore, we use the fact that $0\preceq_{K}\bm\ell=\bm\mu$ to get
$$\nm{x^{k+1}-x^{*}}^{2}\leq\nm{x^{k}-x^{*}}^{2}.$$
Therefore, the sequence $\{\nm{x^{k}-x^{*}}\}$ converges. This, together with the fact that $x^{*}$ is an accumulation point of $\{x^{k}\}$, implies that $\{x^{k}\}$ converges to $x^{*}$
\par(ii) Since $0\preceq_{K}\bm\ell=\bm\mu$, we use inequality (\ref{e3}) to obtain 
\begin{equation}\label{Eq6}
	\dual{c^{*}_{k},F(x^{k+1})-F(x)}\leq\frac{1}{2}\nm{x^{k}-x}^{2}\dual{c^{*}_{k},\bm\ell}-\frac{1}{2}\nm{x^{k+1}-x}^{2}\dual{c^{*}_{k},\bm\mu}=\frac{\dual{c^{*}_{k},\bm\ell}}{2}\left(\nm{x^{k}-x}^{2}-\nm{x^{k+1}-x}^{2}\right).
\end{equation} Taking the sum of the  preceding inequality over $0$ to $k-1$, we have
$$\sum\limits_{s=0}^{k-1}\dual{c^{*}_{s},F(x^{s+1})-F(x)}\leq\sum\limits_{s=0}^{k-1}\frac{\dual{c^{*}_{s},\bm\ell}}{2}\left(\nm{x^{s}-x}^{2}-\nm{x^{s+1}-x}^{2}\right).$$
Notice that $F(x^{k})\preceq_{K}F(x^{s+1})$ for all $s\leq k-1$, it leads to
$$\sum\limits_{s=0}^{k-1}\dual{c^{*}_{s},F(x^{k})-F(x)}\leq\sum\limits_{s=0}^{k-1}\frac{\dual{c^{*}_{s},\bm\ell}}{2}\left(\nm{x^{s}-x}^{2}-\nm{x^{s+1}-x}^{2}\right).$$
Denote $\hat{c}^{*}_{k}:={\sum_{s=0}^{k-1}c^{*}_{s}}/{k}$. It follows from the convexity of $\tilde{C}$ and the fact that $c^{*}_{s} \in \tilde{C}$ that $\hat{c}^{*}_{k} \in \tilde{C}$. Therefore, we conclude that
$$\dual{\hat{c}^{*}_{k},F(x^{k})-F(x)}\leq\sum\limits_{s=0}^{k-1}\frac{\dual{c^{*}_{s},\bm\ell}}{2k}\left(\nm{x^{s}-x}^{2}-\nm{x^{s+1}-x}^{2}\right).$$ 
Select $y^{k}\in\mathop{\arg\max}_{x\in\mathbb{R}^{n}}\min_{c^{*}\in \tilde{C}}\dual{c^{*},F(x^{k})-F(x)}$, it holds that
\begin{align*}
	u_{0}(x^{k})&=\max_{x\in\mathbb{R}^{n}}\min_{c^{*}\in \tilde{C}}\dual{c^{*},F(x^{k})-F(x)}=\min_{c^{*}\in \tilde{C}}\dual{c^{*},F(x^{k})-F(y^{k})}\\
	&\leq\dual{\hat{c}^{*}_{k},F(x^{k})-F(y^{k})}\leq\sum\limits_{s=0}^{k-1}\frac{\dual{c^{*}_{s},\bm\ell}}{2k}\left(\nm{x^{s}-y^{k}}^{2}-\nm{x^{s+1}-y^{k}}^{2}\right).
\end{align*}
By the definition of $y^{k}$, we deduce that $y^{k}\in\{x:F(x)\preceq_{K}F(x^{k})\}\subset\mathcal{L}_{F}(x^{s})$ for all $s\leq k-1$. Substituting this relation into (\ref{Eq6}), we have $\nm{x^{s}-y^{k}}^{2}-\nm{x^{s+1}-y^{k}}^{2}\geq0$ for all $s\leq k-1$. Therefore, 
$$u_{0}(x^{k})\leq\frac{\ell_{\max}}{2k}\sum\limits_{s=0}^{k-1}\left(\nm{x^{s}-y^{k}}^{2}-\nm{x^{s+1}-y^{k}}^{2}\right)\leq\frac{\ell_{\max}\nm{x^{0}-y^{k}}^{2}}{2k}.$$ Recall that $y^{k}\in\mathcal{L}_{F}(x^{0})$, the desired result follows.
\end{proof}
\subsubsection{Linear convergence}
By further assuming that $F(\cdot)$ is strongly $K$-convex, the linear convergence result of Algorithm \ref{mm} can be derived as follows.
\begin{theorem}\label{t3}
	Suppose that Assumption \ref{a1}(ii) holds and $F(\cdot)$ is strongly $K$-convex, let $\{x^{k}\}$ be the sequence generated by Algorithm \ref{mm} with $G_{k}(\cdot)\in\mathcal{S}_{\bm \ell,\bm\mu}(F,x^{k})$ and $0\preceq_{K}\bm\ell\prec_{K}\bm\mu$. Then, the following statements hold:
	\begin{itemize}
		\item[(i)] $\{x^{k}\}$ converges to an efficient solution $x^{*}$ of (\ref{VOP});
		\item[(ii)] $\nm{x^{k+1}-x^{*}}\leq\sqrt{\max\limits_{c^{*}\in C_{k}}\frac{\dual{c^{*},\bm\ell}}{\dual{c^{*},\bm\mu}}}\nm{x^{k}-x^{*}},~\forall k\geq0$.
	\end{itemize}
\end{theorem}
\begin{proof}
	(i) Since $F(\cdot)$ is strongly $K$-convex, then Assumption \ref{a1}(i) holds and every weakly efficient solution is actually an efficient solution. Therefore, assertion (i) is a consequence of Theorem \ref{t2}(i).
	\par(ii) By substituting $x=x^{*}$ into inequality (\ref{e3}), we have
	$$\dual{c^{*}_{k},F(x^{k+1})-F(x^{*})}\leq\frac{1}{2}\nm{x^{k}-x^{*}}^{2}\dual{c^{*}_{k},\bm\ell}-\frac{1}{2}\nm{x^{k+1}-x^{*}}^{2}\dual{c^{*}_{k},\bm\mu}.$$
	It follows by $F(x^{*})\preceq_{K}F(x^{k+1})$ that
	$$\nm{x^{k+1}-x^{*}}\leq\sqrt{\frac{\dual{c^{*}_{k},\bm\ell}}{\dual{c^{*}_{k},\bm\mu}}}\nm{x^{k}-x^{*}}.$$
	The desired result follows .
\end{proof}
\begin{remark}
	It seems that the convexity of $F(\cdot)$ plays no role in the proof of Theorems \ref{t2} and \ref{t3}. However, it can indeed be shown that $F(\cdot)$ is necessarily $K$-convex if $\bm\ell=\bm\mu$ and strongly $K$-convex with $\bm\mu - \bm\ell$ if $\bm\ell\prec_{K}\bm\mu$. In the next section, we will give some examples where such a condition holds.
\end{remark}
\begin{remark}\label{r3.4}
 Note that $0\notin C_{k}$ and $\text{cone}(C_{k})=K^{*}$, we have $$\max\limits_{c^{*}\in C_{k}}\frac{\dual{c^{*},\bm\ell}}{\dual{c^{*},\bm\mu}}=\max\limits_{c^{*}\in K^{*}\setminus\{0\}}\frac{\dual{c^{*},\bm\ell}}{\dual{c^{*},\bm\mu}}.$$
 Therefore, the linear convergence rate is related to $\{G_{k}(\cdot)\}$, not $\{C_{k}\}$, which confirms that the rate of convergence can be improved by choosing a tighter surrogate. 
\end{remark}
\section{First-order methods for VOPs with majorizing surrogate functions}\label{sec4}
It is worth noting that Remark~\ref{r3.4} may suggest that the choice of the base in the subproblem is inessential, which could make the question (\ref{Q}) raised in the introduction seem trivial. However, as will be demonstrated in this section, the selection of such a base not only influences the convergence rate but also determines the computational complexity of the subproblem. Both aspects are of central importance in the framework of majorization-minimization optimization. 
\par In what follows, we first revisit the classical steepest descent method for VOPs (SDVO) \citep{GS2005} and establish its connection with Algorithm \ref{mm}. To mitigate the slow convergence of SDVO, we then investigate, from a majorization-minimization perspective, how an appropriate choice of the base can effectively accelerate first-order methods for VOPs. 
\subsection{$K$-steepest descent method for VOPs without line search}
\par For $x\in\mathbb{R}^{n}$, recall that $d^{k}$, the $K$-steepest descent direction \citep{GS2005} at $x^{k}$, is defined as the optimal solution of 
\begin{align}\label{eq3.1}
	\min\limits_{d\in\mathbb{R}^{n}} \max\limits_{c^{*}\in C}\dual{c^{*}, JF(x^{k})d}+\frac{1}{2}\|d\|^{2}.
\end{align}
Select a vector $e\in{\rm int}(K)$, and denote $C_{e}=\{c^{*}\in K^{*}:\dual{c^{*},e}=1\}$. If we set $C=C_{e}$ in (\ref{eq3.1}), then the $K$-steepest descent direction can be reformulated as the optimal solution of 
\begin{equation}\label{eq5}
	\min\limits_{d\in\mathbb{R}^{n}} \max\limits_{c^{*}\in C_{e}}\dual{c^{*}, JF(x^{k})d+\frac{1}{2}\|d\|^{2}e}.
\end{equation}
\begin{remark}
	If $K=\mathbb{R}^{m}_{+}$, and $C_{e}=\Delta_{m}$, then $e=\mathbf{1}_{m}$ and the subproblem (\ref{eq5}) reduces to that of steepest descent method for MOPs \citep{FS2000}. In what follows, we refer to subproblems of the forms (\ref{eq3.1}) and (\ref{eq5}) as \textbf{\textit{seperate}} and \textbf{\textit{coupled}} subproblems, respectively.
\end{remark}
From now on, we assume that $F(\cdot)$ is $K$-smooth with $\bm \ell\in{\rm int}(K)$,
denote $$L_{\max}:=\max_{c^{*}\in C_{e}}\dual{c^{*},\bm\ell}.$$
Let us revisit the $K$-steepest descent method without line search:
\begin{algorithm} 
	\caption{$K$-steepest descent method for VOPs}\label{sd}
	\LinesNumbered 
	\KwData{$x^{0}\in\mathbb{R}^{n},L\geq L_{\max}$}
	\For{$k=0,1,...$}{  
		Update $x^{k+1}:=\mathop{\arg\min}_{x\in\mathbb{R}^{n}} \max_{c^{*}\in C_{e}}\dual{c^{*}, JF(x^{k})(x-x^{k})}+\frac{L}{2}\|x-x^{k}\|^{2}$\\
		\If{$x^{k+1}=x^{k}$}{ {\bf{return}} $K$-stationary point $x^{k}$ }}  
\end{algorithm}

We consider the following surrogate:
\begin{equation}\label{m1}
	G_{k,Le}(x):=JF(x^{k})(x-x^{k})+\frac{L}{2}\nm{x-x^{k}}^{2}e.
\end{equation}
It is obvious that Algorithm \ref{sd} is a special case of Algorithm \ref{mm} with $C_{k}=C_{e}$ and $G_{k}(\cdot)=G_{k,Le}(\cdot)$. As described in Remark \ref{r3.4}, the peformance of Algorithm \ref{sd} is mainly attributed to $G_{k,Le}(\cdot)$. The following results show that $G_{k,Le}(\cdot)$ is a majorizing surrogate function of $F(\cdot)-F(x^{k})$ near $x^{k}$.  
\begin{proposition}\label{p4.1}
	Let $G_{k,Le}(\cdot)$ be defined as (\ref{m1}). Then, the following statements hold.
	\begin{itemize}
		\item[(i)] For any $L\geq L_{\max}$, $G_{k,Le}(\cdot)$ is a majorizing surrogate of $F(\cdot)-F(x^{k})$, i.e., $F(\cdot)-F(x^{k})\preceq_{K}G_{k,Le}(\cdot)$.
		\item[(ii)] If $F(\cdot)$ is $K$-convex, then $G_{k,Le}(\cdot)\in\mathcal{S}_{Le,Le}(F,x^{k})$ for all $L\geq L_{\max}$.
		\item[(iii)] If $F(\cdot)$ is strongly $K$-convex with $\bm\mu\in{\rm int}(K)$, then $G_{k,Le}(\cdot)\in\mathcal{S}_{Le-\bm\mu,Le}(F,x^{k})$ for all $L\geq L_{\max}$.
	\end{itemize}
\end{proposition}
\begin{proof}
	By the definition of $L_{\max}$, we have $\bm\ell\preceq_{K}L_{\max}e$, it follows from the $K$-smoothness of $F(\cdot)$ that assertion (i) holds. Notice that $G_{k,Le}(\cdot)$ is strongly $K$-convex and $K$-smooth with $Le$, and $\bm\mu\preceq_{K} Le$, then we obtain assertion (ii) and (iii) by Lemma \ref{l2.2} (ii) and (iii), respectively. 
\end{proof}
Note that for a strongly $K$-convex objective function, 
$G_{k,Le}(\cdot) \in \mathcal{S}_{Le-\bm{\mu},Le}(F,x^{k})$ 
for all $L \geq L_{\max}$. We are now in a position to present the rate of linear convergence for SDVO.
\begin{lemma}\label{l4.1}
		Assume that $F(\cdot)$ is strongly $K$-convex with $\bm\mu\in{\rm int}(K)$, let $\{x^{k}\}$ be the sequence generated by Algorithm \ref{sd}. Then, the following statements hold:
		\begin{itemize}
			\item[(i)] $\{x^{k}\}$ converges to an efficient solution $x^{*}$ of (\ref{VOP});
			\item[(ii)] $\nm{x^{k+1}-x^{*}}\leq\sqrt{1-{\mu_{\min}}/{L}}\nm{x^{k}-x^{*}},~\forall k\geq0$, where $\mu_{\min}:=\min_{c^{*}\in C_{e}}\dual{c^{*},\bm\mu}$.
		\end{itemize}
\end{lemma}
\begin{proof}
	Since $F(\cdot)$ is strongly $K$-convex, it follows that $G_{k,Le}(\cdot)\in\mathcal{S}_{Le-\bm\mu,Le}(F,x^{k})$ and Assumption \ref{a1} holds in this case. By setting $C_{k}=C_{e}$, Theorem \ref{t3} (i) and (ii) reduce to the assertions (i) and (ii), respectively.
\end{proof}
\begin{remark}\label{r4.2}
	If $K=\mathbb{R}^{m}_{+}$ and $e=\mathbf{1}_{m}$, then $C_{e}=\Delta_{m}$ is a base of $\mathbb{R}^{m}_{+}$, the convergence rate in Lemma \ref{l4.1} reduces to that of \citep[Theorem 5.3]{TFY2023c} with $g(\cdot)=0$. Specifically, the linear convergence rate is worse than $\mathcal{O}((\sqrt{1-{\mu_{\min}}/{L_{\max}}})^k)$ (setting $L = L_{\max}$), where $L_{\max}=\max_{i\in[m]}\{\bm\ell_{i}\}$ and $\mu_{\min}=\min_{i\in[m]}\{\bm\mu_{i}\}$. Therefore, even each of objective functions is well-conditioned ($\max_{i\in[m]}\{\bm\ell_{i}/\bm\mu_{i}\}$ is relative small), the linear convergence rate can be very slow due to objective imbalances ($L_{\max}/\mu_{\min}$ can be extremely large). It is worth noting that the rate of convergence is related to $C_e$, since in the seperate subproblem the surrogate function is inherently determined by $C_e$. To the best of our knowledge, apart from $\Delta_{m}$, it remains an open problem for the better choice of the base in MOPs.
\end{remark}
\subsection{Improved $K$-steepest descent method for VOPs without line search}
As detailed in Remark \ref{r4.2}, the linear convergence rate can be very slow with imbalanced objectives, this is mainly due to the large gap between $F(\cdot)-F(x^{k})$ and $G_{k,Le}(\cdot)$ from a majorization-minimization perspective. To reduce this gap, one natural strategy is to construct a tighter surrogate function that better approximates the local behavior of $F(\cdot) - F(x^{k})$.
Notice that $\bm\ell\preceq_{K}L_{\max}e$, we denote the following tighter majorizing surrogate:
\begin{equation}\label{m2}
	G_{k,\bm\ell}(x):=JF(x^{k})(x-x^{k})+\frac{1}{2}\nm{x-x^{k}}^{2}\bm\ell.
\end{equation}

The properties of $G_{k,\bm\ell}(\cdot)$ is presented as follows.
\begin{proposition}\label{p4.2}
	Let $G_{k,\bm\ell}(\cdot)$ be defined as (\ref{m2}). Then, the following statements hold.
	\begin{itemize}
		\item[(i)] $G_{k,\bm\ell}(\cdot)$ is a tight majorizing surrogate of $F(\cdot)-F(x^{k})$, i.e., $F(\cdot)-F(x^{k})\preceq_{K}G_{k,\bm\ell}(\cdot)$, and the relation does not hold for any $G_{k,\hat{\bm\ell}}(\cdot)$ such that $\bm\ell\not\preceq_{K}\hat{\bm\ell}$.
		\item[(ii)] If $F(\cdot)$ is $K$-convex, then $G_{k,\bm\ell}(\cdot)\in\mathcal{S}_{\bm\ell,\bm\ell}(F,x^{k})$.
		\item[(iii)] If $F(\cdot)$ is strongly $K$-convex with $\bm\mu\in{\rm int}(K)$, then $G_{k,\bm\ell}(\cdot)\in\mathcal{S}_{\bm\ell-\bm\mu,\bm\ell}(F,x^{k})$.
	\end{itemize}
\end{proposition}
\begin{proof}
	The assertions can be obtained by using the similar arguments as in the proof of Proposition \ref{p4.1}.
\end{proof}
By using the tighter surrogate, we devise the following improved $K$-steepest descent method with coupled subproblems for VOPs.

\begin{algorithm}
	\caption{improved $K$-steepest descent method for VOPs with coupled subproblems}\label{isd1}
	\LinesNumbered 
	\KwData{$x^{0}\in\mathbb{R}^{n}$}
	\For{$k=0,1,...$}{  
		Update $x^{k+1}:=\mathop{\arg\min}_{x\in\mathbb{R}^{n}} \max_{c^{*}\in C_{e}}\dual{c^{*}, JF(x^{k})(x-x^{k})+\frac{1}{2}\|x-x^{k}\|^{2}\bm\ell}$\\
		\If{$x^{k+1}=x^{k}$}{ {\bf{return}} $K$-stationary point $x^{k}$ }}  
\end{algorithm}

\begin{lemma}\label{l4.2}
	Assume that $F(\cdot)$ is strongly $K$-convex with $\bm\mu\in{\rm int}(K)$, let $\{x^{k}\}$ be the sequence generated by Algorithm \ref{isd1}. Then, the following statements hold:
	\begin{itemize}
		\item[(i)] $\{x^{k}\}$ converges to an efficient solution $x^{*}$ of (\ref{VOP});
		\item[(ii)] $\nm{x^{k+1}-x^{*}}\leq\sqrt{1-{1}/{\kappa_{F,\preceq_{K}}}}\nm{x^{k}-x^{*}},~\forall k\geq0$.
	\end{itemize}
\end{lemma}
\begin{proof}
	The assertions can be obtained by using the similar arguments as in the proof of Lemma \ref{l4.1}.
\end{proof}
\begin{remark}\label{r4.3}
	If $K=\mathbb{R}^{m}_{+}$, and $e=\mathbf{1}_{m}$, the convergence rate in Lemma \ref{l4.2} reduces to that of \citep[Corollary 4.3]{CTY2023b} with $g(\cdot)=0$. Notice that $1 / \kappa_{F,\preceq_{K}} \geq\mu_{\min}/L $, which indicates that Algorithm \ref{isd1} enjoys faster linear convergence than Algorithm \ref{sd}. Furthermore, by Remark \ref{r3.4}, we conclude that the improved linear convergence does not depend on the choice of $C_e$.
\end{remark}
\subsection{Trade-off between surrogate gap and per-iteration cost}
Although Algorithms \ref{isd1} exhibits improved linear convergence by using a tighter surrogate function, the per-iteration cost is more expensive than that of Algorithm \ref{sd} due to coupled subproblems. Details on solving these subproblems will be provided in Section \ref{sec6} (see Remark \ref{r6.2}). In the spirit of majorization-minimization optimization, a direct question arises: how to strike a better trade-off in terms of surrogate gap and per-iteration cost?
\par Recall that the linear convergence rate of Algorithm \ref{isd1} does not depend on the choice of $C_e$, which is mainly due to coupled subproblems. By denoting
$$C_{\bm\ell}:=\{c^{*}\in K^{*}:\dual{c^{*},\bm\ell}=1\},$$
we propose the following improved $K$-steepest descent method with seperate subproblems for VOPs.

\begin{algorithm} 
	\caption{improved $K$-steepest descent method for VOPs with seperate subproblems}\label{isd2}
	\LinesNumbered 
	\KwData{$x^{0}\in\mathbb{R}^{n}$}
	\For{$k=0,1,...$}{  
		Update $x^{k+1}:=\mathop{\arg\min}_{x\in\mathbb{R}^{n}} \max_{c^{*}\in C_{\bm\ell}}\dual{c^{*}, JF(x^{k})(x-x^{k})} +\frac{1}{2}\|x-x^{k}\|^{2}$\\
		\If{$x^{k+1}=x^{k}$}{ {\bf{return}} $K$-stationary point $x^{k}$ }}  
\end{algorithm}

Using the definition of $C_{\bm\ell}$, the seperate subproblem in Algorithm \ref{isd2} can be rewritten equivalently as the following coupled form:
$$\min_{x\in\mathbb{R}^{n}} \max_{c^{*}\in C_{\bm\ell}}\dual{c^{*}, JF(x^{k})(x-x^{k})+\frac{1}{2}\|x-x^{k}\|^{2}\bm\ell}.$$
Therefore, Algorithm \ref{isd2} enjoys the same improved linear convergence as that of Algorithm \ref{isd1}.
\begin{lemma}\label{l4.2.1}
	Assume that $F(\cdot)$ is strongly $K$-convex with $\bm\mu\in{\rm int}(K)$, let $\{x^{k}\}$ be the sequence generated by Algorithm \ref{isd2}. Then, the following statements hold:
	\begin{itemize}
		\item[(i)] $\{x^{k}\}$ converges to an efficient solution $x^{*}$ of (\ref{VOP});
		\item[(ii)] $\nm{x^{k+1}-x^{*}}\leq\sqrt{1-{1}/{\kappa_{F,\preceq_{K}}}}\nm{x^{k}-x^{*}},~\forall k\geq0$.
	\end{itemize}
\end{lemma}
\begin{remark}\label{r4.4}
	If $K=\mathbb{R}^{m}_{+}$, and $\Delta_{m}^{\bm\ell}:=\{c^{*}\in\mathbb{R}^{m}_{++}:\dual{c^{*},\bm\ell}=1\}$, the Algorithm \ref{isd2} reduces to  \citep[Algorithm 5]{CTY2023b} with $g(\cdot)=0$. Interestingly, the relations between Algorithms \ref{sd}, \ref{isd1} and \ref{isd2} depend solely on the choice of $C_e$. If $C_e=C_{\bm\ell}$, Algorithms \ref{sd}, \ref{isd1} and \ref{isd2} are equivalent. Consequently, regarding the open problem of selecting a better $C_e$ mentioned in Remark \ref{r4.2}, we provide a theoretical answer by setting $C_e=C_{\bm\ell}$. Furthermore, we can summarize that choosing a tighter surrogate function is equivalent to selecting an appropriate base in the seperate subproblem.
\end{remark}
\begin{remark}
Although Algorithms \ref{isd1} and \ref{isd2} both exhibit similar improved linear convergence, the computational cost of solving seperate subproblems is generally lower in Algorithm \ref{isd2}. Details on solving these subproblems will be provided in Section \ref{sec6} (see Remark \ref{r6.2}). 
\end{remark}
\section{First-order methods for VOPs with non-majorizing surrogate functions}\label{sec5}
In the previous section, the majorization-minimization optimization methods were developed using majorizing surrogate functions; however, these surrogates may be overly conservative due to reliance on global upper bounds. From the perspective of majorization-minimization, selecting a non-majorizing surrogate function could potentially enhance performance.

\subsection{$K$-steepest descent method with line search}

Firstly, we revisit $K$-steepest descent method for VOPs with line search.

\begin{algorithm}[H]
	\caption{K-steepest descent method for VOPs with line search}\label{alg4}
	\LinesNumbered 
	\KwData{$x^{0}\in\mathbb{R}^{n},\gamma\in(0,1)$}
	\For{$k=0,1,...$}{  Update $d^{k}:=\mathop{\arg\min}_{d\in\mathbb{R}^{n}} \max_{c^{*}\in C_{e}}\dual{c^{*}, JF(x^{k})d}+\frac{1}{2}\|d\|^{2}$\\
		\eIf{$d^{k}=0$}{ {\bf{return}} $K$-stationary point $x^{k}$ }{Compute the stepsize $t_{k}\in(0,1]$ in the following way:
			$$t_{k}:=\max\left\{\gamma^{j}:j\in\mathbb{N},F(x^{k}+\gamma^{j}d^{k})-F(x^{k})\preceq_{K}\gamma^{j}\left(JF(x^{k})d^{k}+\frac{1}{2}\nm{d^{k}}^{2}e\right)\right\}$$\\
			$x^{k+1}:= x^{k}+t_{k}d^{k}$}}  
\end{algorithm}
 
The stepsize has the following lower bound.

\begin{proposition}\rm\label{p5.1}
	The stepsize generated in Algorithm \ref{alg4} satisfies $t_{k}\geq t_{\min}:=\min\left\{\frac{\gamma}
	{L_{\max}},1\right\}$.
\end{proposition}
\begin{proof}
	By the line search condition in Algorithm \ref{alg4}, we have
	$$	F(x^{k}+\frac{t_{k}}{\gamma}d^{k})-F(x^{k}) \not\preceq_{K} \frac{t_{k}}{\gamma}\left(JF(x^{k})d^{k}+\frac{1}{2}\nm{d^{k}}^{2}e\right).$$
	Then there exists $c^{*}_{1}\in C_{e}$ such that
	\begin{equation}\label{ine1}
		\dual{c^{*}_{1},F(x^{k}+\frac{t_{k}}{\gamma}d^{k})-F(x^{k})}>\dual{c^{*}_{1},\frac{t_{k}}{\gamma}\left(JF(x^{k})d^{k}+\frac{1}{2}\nm{d^{k}}^{2}e\right)}.
	\end{equation}
	On the other hand, the $K$-smoothness of $F(\cdot)$ implies
	$$F(x^{k}+\frac{t_{k}}{\gamma}d^{k})-F(x^{k})\preceq_{K}\frac{t_{k}}{\gamma}JF(x^{k})d^{k}+\frac{1}{2}\nm{\frac{t_{k}}{\gamma}d^{k}}^{2}\bm\ell.$$
	Therefore, we have
	$$\dual{c^{*}_{1},F(x^{k}+\frac{t_{k}}{\gamma}d^{k})-F(x^{k})}\leq\dual{c^{*}_{1},\frac{t_{k}}{\gamma}JF(x^{k})d^{k}}+\frac{1}{2}\nm{\frac{t_{k}}{\gamma}d^{k}}^{2}\dual{c^{*}_{1},\bm\ell}.$$
	This, together with inequality (\ref{ine1}), yields
	$$t_{k}\geq \frac{\gamma}{\dual{c^{*}_{1},\bm\ell}}.$$
	Then the desired result follows.
\end{proof}

We consider the following  surrogate:
\begin{equation}\label{m3}
	G_{k,e/t_{\min}}(x):=JF(x^{k})(x-x^{k})+\frac{1}{2t_{\min}}\nm{x-x^{k}}^{2}e.
\end{equation}
The following results show that $G_{k,e/t_{\min}}(\cdot)$ is a non-majorizing surrogate function of $F(\cdot)-F(x^{k})$ near $x^{k}$.
\begin{proposition}\label{p5.2}
	Let $G_{k,e/t_{\min}}(\cdot)$ be defined as (\ref{m3}). Then, the following statements hold.
	\begin{itemize}
		\item[(i)] $F(x^{k+1})-F(x^{k})\preceq_{K}G_{k,e/t_{\min}}(x^{k+1})$.
		\item[(ii)] If $F(\cdot)$ is $K$-convex, then $G_{k,e/t_{\min}}(\cdot)\in\mathcal{S}_{e/t_{\min},e/t_{\min}}(F,x^{k})$.
		\item[(iii)] If $F(\cdot)$ is strongly $K$-convex with $\bm\mu\in{\rm int}(K)$, then $G_{k,e/t_{\min}}(\cdot)\in\mathcal{S}_{e/t_{\min}-\bm\mu,e/t_{\min}}(F,x^{k})$.
	\end{itemize}
\end{proposition}
\begin{proof}
	By the line search condition, we have
	$$F(x^{k+1})-F(x^{k})\preceq_{K}JF(x^{k})(x^{k+1}-x^{k})+\frac{1}{2t_{k}}\nm{x^{k+1}-x^{k}}^{2}e\preceq_{K}JF(x^{k})(x^{k+1}-x^{k})+\frac{1}{2t_{\min}}\nm{x^{k+1}-x^{k}}^{2}e.$$ Then, the assertion (i) holds. The assertions (ii) and (iii) can be obtained by using the similar arguments as in the proof of Proposition \ref{p4.1}.
\end{proof}
\begin{lemma}\label{l5.1}
	Assume that $F(\cdot)$ is strongly $K$-convex with $\bm\mu\in{\rm int}(K)$, let $\{x^{k}\}$ be the sequence generated by Algorithm \ref{alg4}. Then, the following statements hold:
	\begin{itemize}
		\item[(i)] $\{x^{k}\}$ converges to an efficient solution $x^{*}$ of (\ref{VOP});
		\item[(ii)] $\nm{x^{k+1}-x^{*}}\leq\sqrt{1-t_{\min}{\mu_{\min}}}\nm{x^{k}-x^{*}},~\forall k\geq0$, where $\mu_{\min}:=\min_{c^{*}\in C_{e}}\dual{c^{*},\bm\mu}$.
	\end{itemize}
\end{lemma}
\begin{proof}
	The assertions can be obtained by using the similar arguments as in the proof of Lemma \ref{l4.1}.
\end{proof}
\begin{remark}
	If $K=\mathbb{R}^{m}_{+}$, and $e=\mathbf{1}_{m}$, i.e., $C_{e}=\Delta_{m}$, the convergence rate in Lemma \ref{l5.1}(ii) reduces to those established in \citep[Theorem 4.2]{FVV2019} and \citep[Theorem 5.6]{ZDH2019}.
\end{remark}

\subsection{Generic first-order method for VOPs with line search}

To reduce the gap between $F(\cdot)-F(x^{k})$ and $G_{k,e/t_{\min}}(\cdot)$, we select $e_{k}\in{\rm int}(K)$ and devise the following generic first-order method:

\begin{algorithm}
	\caption{Generic first order method for VOPs with line search}\label{alg5}
	\LinesNumbered 
	\KwData{$x^{0}\in\mathbb{R}^{n},\gamma\in(0,1)$}
	\For{$k=0,1,...$}{Select $e_{k}\in{\rm int}(K)$\\
		  Update $d^{k}:=\mathop{\arg\min}_{d\in\mathbb{R}^{n}} \max_{c^{*}\in C_{e}}\dual{c^{*}, JF(x^{k})d+\frac{1}{2}\|d\|^{2}e_{k}}$\\
		\eIf{$d^{k}=0$}{ {\bf{return}} $K$-stationary point $x^{k}$ }{Compute the stepsize $t_{k}\in(0,1]$ in the following way:
			$$t_{k}:=\max\left\{\gamma^{j}:j\in\mathbb{N},F(x^{k}+\gamma^{j}d^{k})-F(x^{k})\preceq_{K}\gamma^{j}\left(JF(x^{k})d^{k}+\frac{1}{2}\nm{d^{k}}^{2}e_{k}\right)\right\}$$\\
			$x^{k+1}:= x^{k}+t_{k}d^{k}$}}  
\end{algorithm}
\par It is worth noting that we don't specify how to select $e_{k}$ in Algorithm \ref{alg5}. This naturally raises the question: what role does $e_{k}$ play in determining the convergence rate? Firstly, we derive the lower bound of stepsize in each iteration.
\begin{proposition}
	The stepsize generated in Algorithm \ref{alg5} satisfies $t_{k}\geq t_{k}^{\min}:=\min\left\{\min\limits_{c^{*}\in C_{e}}\frac{\gamma\dual{c^{*},e_{k}}}{\dual{c^{*},\bm\ell}},1\right\}$.
\end{proposition}
\begin{proof}
	The result can be obtained by using the similar arguments as in the proof of Proposition \ref{p5.1}.
\end{proof}
We consider the following surrogate:
\begin{equation}\label{m4}
	G_{k,e_{k}/t_{k}}(x):=JF(x^{k})(x-x^{k})+\frac{1}{2t^{\min}_{k}}\nm{x-x^{k}}^{2}e_{k}.
\end{equation}
We can show that $G_{k,e_{k}/t^{\min}_{k}}(\cdot)$ is a non-majorizing surrogate function of $F(\cdot)-F(x^{k})$ near $x^{k}$.
\begin{proposition}
	Let $G_{k,e_{k}/t^{\min}_{k}}(\cdot)$ be defined as (\ref{m4}). Then, the following statements hold.
	\begin{itemize}
		\item[(i)] $F(x^{k+1})-F(x^{k})\preceq_{K}G_{k,e_{k}/t^{\min}_{k}}(x^{k+1})$..
		\item[(ii)] If $F(\cdot)$ is $K$-convex, then $G_{k,e_{k}/t^{\min}_{k}}(\cdot)\in\mathcal{S}_{e_{k}/t^{\min}_{k},e_{k}/t^{\min}_{k}}(F,x^{k})$.
		\item[(iii)] If $F(\cdot)$ is strongly $K$-convex with $\bm\mu\in{\rm int}(K)$, then $G_{k,e_{k}/t^{\min}_{k}}(\cdot)\in\mathcal{S}_{e_{k}/t^{\min}_{k}-\bm\mu,e_{k}/t^{\min}_{k}}(F,x^{k})$.
	\end{itemize}
\end{proposition}
\begin{proof}
	The assertions can be obtained by using the similar arguments as in the proof of Proposition \ref{p5.2}.
\end{proof}
\par The following results show that $e_{k}$ plays a significant role in the convergence rate of Algorithm \ref{alg5}.  
\begin{lemma}\label{l5.2}
	Assume that $F(\cdot)$ is strongly $K$-convex with $\bm\mu\in{\rm int}(K)$, let $\{x^{k}\}$ be the sequence generated by Algorithm \ref{alg5}. Then, the following statements hold:
	\begin{itemize}
		\item[(i)] $\{x^{k}\}$ converges to an efficient solution $x^{*}$ of (\ref{VOP}).
		\item[(ii)] $\nm{x^{k+1}-x^{*}}\leq\sqrt{1-\min\limits_{c^{*}\in C_{e}}\frac{\dual{c^{*},\bm\mu}}{\dual{c^{*},e_{k}/t^{\min}_{k}}}}\nm{x^{k}-x^{*}},~\forall k\geq0$.
		\item[(iii)] If $e_{k}=e$, we have $$\nm{x^{k+1}-x^{*}}\leq\sqrt{1-t_{\min}{\mu_{\min}}}\nm{x^{k}-x^{*}},~\forall k\geq0,$$ where $\mu_{\min}:=\min_{c^{*}\in C_{e}}\dual{c^{*},\bm\mu}$.
		\item[(iv)] For any $e_{k}\in\mathrm{int}K$, we have
		\begin{equation}\label{cr1}
		\min\limits_{c^{*}\in C_{e}}\frac{\dual{c^{*},\bm\mu}}{\dual{c^{*},e_{k}/t^{\min}_{k}}}\leq\frac{\gamma}{\kappa_{F,\preceq_{K}}}.	
		\end{equation}
		
Moreover, the equality holds with $e_{k}=\bm\mu$ or $e_{k}=\bm\ell$. In these cases, we have
\begin{equation}\label{cr2}
	\nm{x^{k+1}-x^{*}}\leq\sqrt{1-{\gamma}/{\kappa_{F,\preceq_{K}}}}\nm{x^{k}-x^{*}},~\forall k\geq0.	
\end{equation}
	\end{itemize}
\end{lemma}
\begin{proof}
	The assertions (i) and (ii) can be obtained by using the similar arguments as in the proof of Lemma \ref{l4.1}. By substituting $e_{k}=e$ into (ii), we can obtain the assertion (iii). Next, we prove the assertion (iv). Since $C_{e}$ is a compact set, there exists a vector $c_{0}^{*}\in C_{e}$ such that $1/\kappa_{F,\preceq_{K}}={\dual{c_{0}^{*},\bm\mu}}/{\dual{c_{0}^{*},\bm\ell}}.$ On the other hand, by the definition of $t_{k}^{\min}$ we can deduce 
	$$\min\limits_{c^{*}\in C_{e}}\frac{\dual{c^{*},\bm\mu}}{\dual{c^{*},e_{k}/t^{\min}_{k}}}\leq\frac{\dual{c_{0}^{*},\bm\mu}}{\dual{c_{0}^{*},e_{k}}}\frac{\gamma\dual{c_{0}^{*},e_{k}}}{\dual{c_{0}^{*},\bm\ell}}=\frac{\gamma}{\kappa_{F,\preceq_{K}}}.$$ 
	Then the relation (\ref{cr1}) follows. The equality can be obtain by substituting $e_{k}=\bm\mu$ or $e_{k}=\bm\ell$ into the left-hand side of (\ref{cr1}) . Moreover, The equality leads to the relation (\ref{cr2}).
\end{proof}
\begin{remark}
	As described in Lemma \ref{l5.2}(iv), by setting $e_{k}=\bm\mu$ or $e_{k}=\bm\ell$, we can derive the optimal linear convergence rate for Algorithm \ref{alg5}, and the convergence rate reduces to that of Lemma \ref{l4.2.1}(ii) with constant $\gamma$. Intuitively, to explore the local curvature information of $F(\cdot)$,  we can devise a tighter local surrogate $G_{k,e_{k}/t_{k}}(\cdot)$ with $\bm\mu\preceq_{K}e_{k}\preceq_{K}\bm\ell$. In this case, the performance of Algorithm \ref{alg5} can be further improved by using a tighter local surrogate $G_{k,e_{k}/t_{k}}(\cdot)$.
\end{remark}
To narrow the surrogate gap and better capture the local curvature information, we compute $e_{k}$ by Barzilai-Borwein method, namely, we set
\begin{equation}\label{ek}
	e_{k}:=\frac{\dual{JF(x^{k})-JF(x^{k-1}),x^{k}-x^{k-1}}}{\nm{x^{k}-x^{k-1}}^{2}}.
\end{equation}
\begin{lemma}
	Assume that $F(\cdot)$ is strongly $K$-convex with $\bm\mu\in{\rm int}(K)$, let $\{x^{k}\}$ be the sequence generated by Algorithm \ref{alg5}, where $e_{k}$ is defined as in (\ref{ek}). Then, the following statements hold:
	\begin{itemize}
		\item[(i)]  $\bm\mu\preceq_{K}e_{k}\preceq_{K}\bm\ell$;
		\item[(ii)] $t_{k}\geq\min_{c^{*}\in C_{e}}\{\gamma\dual{c^{*},e_{k}}/\dual{c^{*},\bm\ell}\}$;
	\end{itemize}
\end{lemma}
\begin{proof}
	Assertion (i) follows by the strong $k$-convexity and $K$-smoothness of $F(\cdot)$, and the definition of $e_{k}$. We can obtain the assertion (ii) by using the similar arguments as in the proof of Proposition \ref{p5.1}.
\end{proof}
\begin{remark}\label{r5.3}
	To reduce per-iteration cost, we set $e=e_{k}$, i.e., $C_e=C_{e_{k}}$, so that the coupled subproblem in Algorithm \ref{alg5} can be reformulated into the following seperate form:
	$$\min_{d\in\mathbb{R}^{n}} \max_{c^{*}\in C_{e_{k}}}\dual{c^{*}, JF(x^{k})d}+\frac{1}{2}\|d\|^{2}.$$ Hence, we conclude that using a variable $C_{e_{k}}$ serves as an appropriate choice of base in SDVO, which provides a theoretical answer to (\ref{Q}). For $K=\mathbb{R}^{2}_{+}$, Fig. \ref{fig1} illustrates the choice of the base under the assumption of strong convexity. It suggests that bases should be adaptively selected from the pink region according to (\ref{ek}).
	\begin{figure}[htbp]
		\begin{center}	
			\includegraphics[scale=0.22]{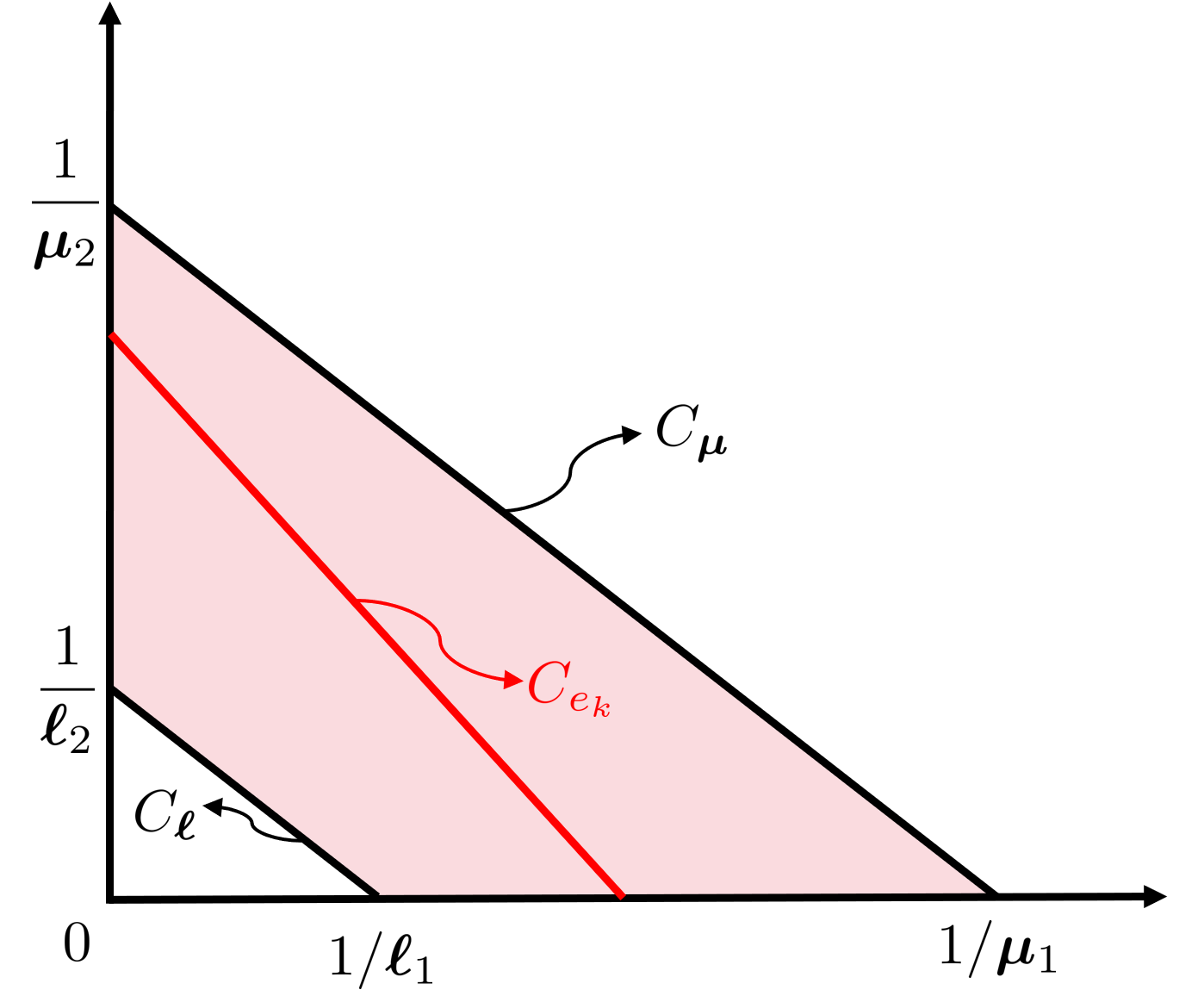} 
		\end{center}
		\caption{Illustration of the $C_{e_{k}}$ of $K=\mathbb{R}^{2}_{+}$.}\label{fig1}
		
	\end{figure}
\end{remark}

\section{First-order methods for VOPs with polyhedral cones}\label{sec6}
In this section, we consider the VOPs that $K$ is a polyhedral cone with nonempty interior. Without loss of generality, for the polyhedral cone $K$, there exists a transform matrix $A\in\mathbb{R}^{l\times m}$ with $m\leq l$ such that
$$K:=\{x\in\mathbb{R}^{m}:0\preceq Ax\}.$$
In this case, for any $a,b\in\mathbb{R}^{m}$, $a\preceq_{K}b$ can be equivalently represented as $Aa\preceq Ab$. Denote $A_{i}$ the $i$-th row vector of $A$. For the polyhedral cone $K$, we denote the set of transform matrices as follows:
$$\mathcal{A}:=\left\{A\in\mathbb{R}^{l\times m}:AK=\mathbb{R}^{l}_{+}\right\}.$$
Notably, in practice, the polyhedral cone $K$ is often defined by a specific transform matrix $A$. This raises a crucial question: does the transform matrix $A$ affect the performance of descent methods for VOPs with polyhedral cones? 
\subsection{Steepest descent method for VOPs with polyhedral cones}
By using a transform matrix $A\in\mathcal{A}$, the steepest descent direction subproblem for VOPs with polyhedral cones is formulated as follows:
\begin{equation}\label{sdp}
	\min\limits_{d\in\mathbb{R}^{n}}\max\limits_{\lambda\in\Delta_{l}}\dual{\lambda,AJF(x^{k})d}+\frac{1}{2}\nm{d}^{2}.
\end{equation}
The complete $K$-steepest descent method for VOPs with polyhedral cones is described as follows:

\begin{algorithm}
	\caption{K-steepest descent method for VOPs with polyhedral cones}\label{alg7}
	\LinesNumbered 
	\KwData{$x^{0}\in\mathbb{R}^{n},\gamma\in(0,1)$}
	Select a transform matrix $A\in\mathcal{A}$\\
	\For{$k=0,1,...$}{
		Update $d^{k}$ as the minimizer of (\ref{sdp})\\
		\eIf{$d^{k}=0$}{ {\bf{return}} $K$-stationary point $x^{k}$ }{Compute the stepsize $t_{k}\in(0,1]$ in the following way:
			$$t_{k}:=\max\left\{\gamma^{j}:j\in\mathbb{N},A(F(x^{k}+\gamma^{j}d^{k})-F(x^{k}))\preceq\gamma^{j} \left(AJF(x^{k})d^{k}+\frac{1}{2}\nm{d^{k}}^{2}\mathbf{1}_{l}\right)\right\}$$
			$x^{k+1}:= x^{k}+t_{k}d^{k}$}}  
\end{algorithm}

We consider the following surrogate:
\begin{equation}\label{m5}
	G_{k,A,\mathbf{1}_{l}/t_{k}}(x):=AJF(x^{k})(x-x^{k})+\frac{1}{2t_{k}}\nm{x-x^{k}}^{2}\mathbf{1}_{l}.
\end{equation}
We can show that $G_{k,A,\mathbf{1}_{l}/t_{k}}(\cdot)$ is a non-majorizing surrogate function of $A(F(\cdot)-F(x^{k}))$ near $x^{k}$.
\begin{proposition}
	Let $G_{k,A,\mathbf{1}_{l}/t_{k}}(\cdot)$ be defined as (\ref{m5}). Then, the following statements hold.
	\begin{itemize}
		\item[(i)] $A(F(x^{k+1})-F(x^{k}))\preceq G_{k,A,\mathbf{1}_{l}/t_{k}}(x^{k+1})$.
		\item[(ii)] If $F(\cdot)$ is $K$-convex, then $G_{k,A,\mathbf{1}_{l}/t_{k}}(\cdot)\in\mathcal{S}_{\mathbf{1}_{l}/t_{k},\mathbf{1}_{l}/t_{k}}(AF,x^{k})$.
		\item[(iii)] If $F(\cdot)$ is strongly $K$-convex with $\bm\mu\in{\rm int}(K)$, then $G_{k,A,\mathbf{1}_{l}/t_{k}}(\cdot)\in\mathcal{S}_{\mathbf{1}_{l}/t_{k}-A\bm\mu,\mathbf{1}_{l}/t_{k}}(AF,x^{k})$.
	\end{itemize}
\end{proposition}
\begin{proof}
	The assertions can be obtained by using the similar arguments as in the proof of Proposition \ref{p5.2}.
\end{proof}

\begin{lemma}\label{l5.4}
	Assume that $F(\cdot)$ is strongly $K$-convex with $\bm\mu\in{\rm int}(K)$, where $K=\{x\in\mathbb{R}^{m}:0\preceq Ax\}.$ Let $\{x^{k}\}$ be the sequence generated by Algorithm \ref{alg7}. Then, the following statements hold:
	\begin{itemize}
		\item[(i)] $t_{k}\geq\min\{\min_{i\in [l]}\{\gamma/\dual{A_{i},\bm\ell}\},1\}$;
		\item[(ii)] $\{x^{k}\}$ converges to an efficient solution $x^{*}$ of (\ref{VOP});
		\item[(iii)] $\nm{x^{k+1}-x^{*}}\leq\sqrt{1-t_{k}\min_{i\in [l]}\dual{A_{i},\bm\mu}}\nm{x^{k}-x^{*}},~\forall k\geq0$.
	\end{itemize}
\end{lemma}
\begin{proof}
	The assertions can be obtained by using the similar arguments as in the proof of Proposition \ref{p5.1} and Lemma \ref{l4.1}.
\end{proof}
\begin{remark}
	 The subproblem (\ref{sdp}) can be reformulated as 
	$$\min\limits_{d\in\mathbb{R}^{n}}\max\limits_{c^{*}\in C}\dual{c^{*},JF(x^{k})d}+\frac{1}{2}\nm{d}^{2},$$
	where $C:=conv\{A_{i},i\in[l]\}$ is a base of $K^{*}$.  In other words, selecting a transform matrix in (\ref{sdp}) is equivalent to selecting a base of  $K^{*}$. Consequently, the linear convergence rate of Algorithm \ref{alg7} is sensitive to the choice of $A$.  When $K=\mathbb{R}^{m}_{+}$, the steepest descent method \citep{FS2000} fixs $A=I_{m}$, i.e., $C=\Delta_{m}$. 
\end{remark}
\subsection{Barzilai-Borwein descent method for VOPs with polyhedral cones}
 In general, for the $e_{k}$ defined as (\ref{ek}) we have $e_{k}\preceq_{K}\bm\ell$, which can be written as $Ae_{k}\preceq_{}A\bm\ell$. We denote $\alpha^{k}\in\mathbb{R}^{l}_{++}$ as follows:
\begin{equation}\label{ak}
	\alpha_{i}^{k}=\left\{
	\begin{aligned}
		&\max\left\{\alpha_{\min},\min\left\{\frac{\langle s_{k-1},y^{k-1}_{i}\rangle}{\nm{s_{k-1}}^{2}}, \alpha_{\max}\right\}\right\}, & \langle s_{k-1},y^{k-1}_{i}\rangle&>0, \\
		&\max\left\{\alpha_{\min},\min\left\{\frac{\nm{y^{k-1}_{i}}}{\nm{s_{k-1}}}, \alpha_{\max}\right\}\right\}, & \langle s_{k-1},y^{k-1}_{i}\rangle&<0, \\
		& \alpha_{\min}, &  \langle s_{k-1},y^{k-1}_{i}\rangle&=0,
	\end{aligned}
	\right.
\end{equation}
for all $i\in[l]$, where $s_{k-1}=x^{k}-x^{k-1},\ y^{k-1}_{i}$ is the $i$-th row vector of $A(JF(x^{k})-JF(x^{k-1}))$,  $\alpha_{\max}$ is a sufficient large positive constant and $\alpha_{\min}$ is a sufficient small positive constant. 
The Barzilai-Borwein descent direction is defined as the minimizer of 
\begin{equation}\label{bb1}
	\min\limits_{d\in\mathbb{R}^{n}}\max\limits_{\lambda\in\Delta_{l}}\dual{\lambda,AJF(x^{k})d+\frac{1}{2}\nm{d}^{2}\alpha^{k}}.
\end{equation}
Alternatively, we use the similar strategy in Algorithm \ref{isd2}, the Barzilai-Borwein descent direction subproblem can be rewritten equivalently as the following coupled form:
\begin{equation}\label{bb2}
	\min\limits_{d\in\mathbb{R}^{n}}\max\limits_{\lambda\in\Delta^{\alpha^{k}}_{l}}\dual{\lambda,AJF(x^{k})d}+\frac{1}{2}\nm{d}^{2},
\end{equation}
where $\Delta^{\alpha^{k}}_{l}:=\{c^{*}\in\mathbb{R}^{l}_{+}:\dual{c^{*},\alpha^{k}}=1\}$. The subproblem can be reformulated as follows:
\begin{equation}\label{bbk}
	\min\limits_{d\in\mathbb{R}^{n}}\max\limits_{\lambda\in\Delta_{l}}\dual{\lambda,\Lambda^{k} AJF(x^{k})d}+\frac{1}{2}\nm{d}^{2},
\end{equation} 
where $$\Lambda^{k}:=
\begin{bmatrix}
	\frac{1}{\alpha^{k}_{1}}   & &  \\
	& \ddots&\\
	&  &\frac{1}{\alpha^{k}_{l}}
\end{bmatrix}.
$$
By Sion's minimax theorem \citep{S1958}, the minimizer of (\ref{bbk}) can be written as 
$$d^{k}=-(\Lambda^{k} AJF(x^{k}))^{T}\lambda^{k},$$
where $\lambda^{k}\in\Delta_{l}$ is a solution of the following dual problem:
\begin{align*}\tag{DP}\label{DP}
	&\min\limits_{\lambda\in\Delta_{l}}\frac{1}{2} \left\|(\Lambda_{k} AJF(x^{k}))^{T}\lambda\right\|^{2}.
\end{align*}
\begin{remark}\label{r6.2}
	In general, the dual problem (\ref{DP}) is a lower dimensional quadratic programming with unit simplex constraint (the vertices of unit simplex constraint are known), then it can be solved by Frank-Wolfe/conditional gradient method efficiently \citep[see, e.g.,][]{SK2018,CTY2023a}. However, dual problem of (\ref{bb1}) reads as
	$$\min\limits_{\lambda\in\Delta_{l}}\frac{1}{2}\frac{\left\|( AJF(x^{k}))^{T}\lambda\right\|^{2}}{\sum\limits_{i\in[l]}\lambda_{i}\alpha^{k}_{i}},$$ which is not easy to solve. 
\end{remark}

\begin{remark}\label{r6.3}
If $K=\mathbb{R}^{m}_{+}$ and $A=I_{m}$, the subproblem (\ref{bbk}) reduces to that of the BBDMO \citep{CTY2023a}. Consequently, transforming (\ref{bb1}) into (\ref{bbk}) gives a new insight into BBDMO from a majorization-minimization perspective. Comparing the subproblems (\ref{sdp}) and (\ref{bb2}), it turns out that the differences between the directions of the steepest descent and the Barzilai-Borwein descent lie in the choice of base for $K^{*}$. 
	
\end{remark}

The complete $K$-Barzilai-Borwein descent method for VOPs  with polyhedral cones is described as follows:
\begin{algorithm}
	\caption{K-Barzilai-Borwein descent method for VOPs with polyhedral cones}\label{alg6}
	\LinesNumbered 
	\KwData{$x^{0}\in\mathbb{R}^{n},\gamma\in(0,1)$}
	Select a transform matrix $A\in\mathcal{A}$\\
	Choose $x^{-1}$ in a small neighborhood of $x^{0}$\\
	\For{$k=0,1,...$}{  Update $\alpha^{k}$ as (\ref{ak})\\
		Update $d^{k}$ as the minimizer of (\ref{bbk})\\
		\eIf{$d^{k}=0$}{ {\bf{return}} $K$-stationary point $x^{k}$ }{Compute the stepsize $t_{k}\in(0,1]$ in the following way:
			$$t_{k}:=\max\left\{\gamma^{j}:j\in\mathbb{N},A(F(x^{k}+\gamma^{j}d^{k})-F(x^{k}))\preceq\gamma^{j} \left(AJF(x^{k})d^{k}+\frac{1}{2}\nm{d^{k}}^{2}\alpha^{k}\right)\right\}$$\\
			$x^{k+1}:= x^{k}+t_{k}d^{k}$}}  
\end{algorithm}

We consider the following surrogate:
\begin{equation}\label{m6}
	G_{k,A,\alpha^{k}/t_{k}}(x):=AJF(x^{k})(x-x^{k})+\frac{1}{2t_{k}}\nm{x-x^{k}}^{2}\alpha^{k}.
\end{equation}
We can show that $G_{k,A,\alpha^{k}/t_{k}}(\cdot)$ is a non-majorizing surrogate function of $A(F(\cdot)-F(x^{k}))$ near $x^{k}$.
\begin{proposition}
	Let $G_{k,A,\alpha^{k}/t_{k}}(\cdot)$ be defined as (\ref{m6}). Then, the following statements hold.
	\begin{itemize}
		\item[(i)] $A(F(x^{k+1})-F(x^{k}))\preceq G_{k,A,\alpha^{k}/t_{k}}(x^{k+1})$.
		\item[(ii)] If $F(\cdot)$ is $K$-convex, then $G_{k,A,\alpha^{k}/t_{k}}(\cdot)\in\mathcal{S}_{\alpha^{k}/t_{k},\alpha^{k}/t_{k}}(AF,x^{k})$.
		\item[(iii)] If $F(\cdot)$ is strongly $K$-convex with $\bm\mu\in{\rm int}(K)$, then $G_{k,A,\alpha^{k}/t_{k}}(\cdot)\in\mathcal{S}_{\alpha^{k}/t_{k}-A\bm\mu,\alpha^{k}/t_{k}}(AF,x^{k})$.
	\end{itemize}
\end{proposition}
\begin{proof}
	The assertions can be obtained by using the similar arguments as in the proof of Proposition \ref{p5.2}.
\end{proof}

\begin{lemma}\label{l6.2}
	Assume that $F(\cdot)$ is strongly $K$-convex with $\bm\mu\in{\rm int}(K)$, where $K=\{x\in\mathbb{R}^{m}:0\preceq Ax\}.$ Let $\{x^{k}\}$ be the sequence generated by Algorithm \ref{alg6}. Then, the following statements hold:
	\begin{itemize}
		\item[(i)]  $A\bm\mu\preceq\alpha^{k}\preceq A\bm\ell$;
		\item[(ii)] $t_{k}\geq\min_{i\in [l]}\{\gamma\alpha^{k}_{i}/\dual{A_{i},\bm\ell}\}$;
		\item[(iii)] $\{x^{k}\}$ converges to an efficient solution $x^{*}$ of (\ref{VOP});
		\item[(iv)] $\nm{x^{k+1}-x^{*}}\leq\sqrt{1-t_{k}\min_{i\in [l]}\{{\dual{A_{i},\bm\mu}}/{\alpha^{k}_{i}}\}}\nm{x^{k}-x^{*}},~\forall k\geq0$.
	\end{itemize}
\end{lemma}
\begin{proof}
	The assertions can be obtained by using the similar arguments as in the proof of Proposition \ref{p5.1} and Lemma \ref{l4.1}.
\end{proof}
A large stepsize may speed up the convergence of Algorithm \ref{alg6}. Accordingly, the Armijo line search can be applied, namely, compute the stepsize $t_{k}\in(0,1]$ in the following way:
\begin{equation}\label{armijo}
	t_{k}:=\max\left\{\gamma^{j}:j\in\mathbb{N},A(F(x^{k}+\gamma^{j}d^{k})-F(x^{k}))\preceq\sigma\gamma^{j} AJF(x^{k})d^{k}\right\},
\end{equation}
where $\sigma\in(0,1)$.

The following result shows that the convergence rate of Algorithm \ref{alg6} is not sensitive to the choice of transform matrix. More specifically, the descent direction $d^{k}$ and stepsize $t_{k}$ of Algorithm \ref{alg6} are invariant for some $A\in\mathcal{A}$. 
\begin{proposition}(Affine Invariance)
	Let $A^{1},A^{2}\in \mathcal{A}$, $d^{k}_{1},t_{k}^{1}$ and $d^{k}_{2},t_{k}^{2}$ be the descent directions and stepsize generated by Algorithm \ref{alg6} with $A^{1}$ and $A^{2}$, respectively. If $\alpha_{\min}<\alpha_{i}^{k,1},\alpha_{i}^{k,2}<\alpha_{\max}$, $i\in[l]$, we have $d^{k}_{1}=d^{k}_{2}$ and $t_{k}^{1}=t_{k}^{2}$.
\end{proposition}
\begin{proof}
	Denote $A_{i}^{1}$ and $A_{i}^{2}$ the the $i$-th row vector of $A^{1}$ and $A^{2}$, respectively. Before presenting the main results, we rewritten the subproblem (\ref{bbk}) as follows:
	\begin{equation*}
		\min\limits_{d\in\mathbb{R}^{n}}\max\limits_{i\in[l]}\dual{\frac{A_{i}}{\alpha^{k}_{i}}, JF(x^{k})d}+\frac{1}{2}\nm{d}^{2}.
	\end{equation*}
Recall that $A^{1},A^{2}\in \mathcal{A}$, there exists a vector $a\in\mathbb{R}^{l}_{++}$ such that 
\begin{equation}\label{eq15}
	\left\{A^{1}_{i}:i\in[l]\right\}=\left\{a_{i}A^{2}_{i}:i\in[l]\right\}.
\end{equation}
We claim the following assertion:
\begin{equation}\label{eq16}
	\left\{\frac{A^{1}_{i}}{\alpha_{i}^{k,1}}:i\in[l]\right\}=\left\{\frac{A^{2}_{i}}{\alpha_{i}^{k,2}}:i\in[l]\right\}.
\end{equation}
This, together with the reformulated subproblem, implies that $d^{k}_{1}=d^{k}_{2}$. Therefore, $t_{k}^{1}=t_{k}^{2}$ is a consequence of (\ref{eq16}). 
\par Next, we prove that assertion (\ref{eq16}) holds. For any $i\in[l]$, it follows by (\ref{eq15}) that there exist $j\in[l]$ such that $A^{1}_{i}=a_{j}A^{2}_{j}$. Notice that $\alpha_{\min}<\alpha_{i}^{k,1},\alpha_{i}^{k,2}<\alpha_{\max}$, $i\in[l]$, we distinguish two cases:
\begin{center}
$\alpha_{i}^{k,1}={\nm{A^{1}_{i}(JF(x^{k})-JF(x^{k-1}))}}/{\nm{s_{k-1}}}$ and $\alpha_{j}^{k,2}={\nm{A^{2}_{j}(JF(x^{k})-JF(x^{k-1}))}}/{\nm{s_{k-1}}}$,
\end{center} 
and
\begin{center}
	$\alpha_{i}^{k,1}={\dual{A^{1}_{i}(JF(x^{k})-JF(x^{k-1})),s_{k-1}}}/{\nm{s_{k-1}}^{2}}$ and $\alpha_{i}^{k,2}={\dual{A^{1}_{i}(JF(x^{k})-JF(x^{k-1})),s_{k-1}}}/{\nm{s_{k-1}}^{2}}$.
\end{center}
Since $A^{1}_{i}=a_{j}A^{2}_{j}$ ($a_{j}>0$), it is easy to verify that
$$\frac{A^{1}_{i}}{\alpha_{i}^{k,1}}=\frac{A^{2}_{j}}{\alpha_{j}^{k,2}}$$
holds in both cases. Thus, we have 
\begin{equation*}
	\left\{\frac{A^{1}_{i}}{\alpha_{i}^{k,1}}:i\in[l]\right\}\subseteq\left\{\frac{A^{2}_{i}}{\alpha_{i}^{k,2}}:i\in[l]\right\}.
\end{equation*}
The relation
\begin{equation*}
	\left\{\frac{A^{2}_{i}}{\alpha_{i}^{k,2}}:i\in[l]\right\}\subseteq\left\{\frac{A^{1}_{i}}{\alpha_{i}^{k,1}}:i\in[l]\right\}
\end{equation*}
follows the similar arguments, this concludes the proof.
\end{proof}
\begin{remark}
	Assume $\hat{\mathcal{A}}\subset\mathcal{A}$ is bounded and $A_{i}$ is bounded away from $0$ for all $A\in\hat{\mathcal{A}}$. Then, there exists $\alpha_{\min}$ and $\alpha_{\max}$ such that the assumption $\alpha_{\min}<\alpha_{i}^{k}<\alpha_{\max}$, $i\in[l]$ holds for all $A\in\hat{\mathcal{A}}$ with $ \dual{s_{k-1},{A_{i}(JF(x^{k})-JF(x^{k-1}))}}\neq0$. For the case with linear objective, $\dual{s_{k-1},{A_{i}(JF(x^{k})-JF(x^{k-1}))}}=0$ may hold. As illustrated in \citep[Example 3]{CTY2023a}, for any $A_{1},A_{2}\in\hat{\mathcal{A}}$, we have $d^{k}_{1}\approx d^{k}_{2}$ with sufficient small $\alpha_{\min}$. As a result, we conclude that the performance of Algorithm \ref{alg6} is not sensitive to the choice of $A$.
\end{remark}
\subsection{Backtracking method for VOPs with polyhedral cones}
As described in Lemma \ref{l6.2}, we apply Barzilai-Borwein method to ensure $A\bm\mu\preceq\alpha^{k}\preceq A\bm\ell$, which in turn is expected to obtain $A\bm\mu\preceq\alpha^{k}/t_{k}\preceq A\bm\ell$, thereby achieving a fast linear convergence rate in practice. However, this strategy actually yields an even slower convergence rate. 
\begin{proposition}
		Assume that $F(\cdot)$ is strongly $K$-convex with $\bm\mu\in{\rm int}(K)$, where $K=\{x\in\mathbb{R}^{m}:0\preceq Ax\}.$ Let $\{x^{k}\}$ be the sequence generated by Algorithm \ref{alg6} and $x^{*}$ be the efficient solution satisfies $F(x^*)\preceq_{K}F(x^k)$ for all $k\geq0$. Then 
	 $$\nm{x^{k+1}-x^{*}}\leq\sqrt{1-\gamma\min\limits_{i\neq j,i,j\in[l]}\left\{\frac{\dual{A_{i},\bm\mu}}{\dual{A_{i},\bm\ell}}\frac{\dual{A_{j},\bm\mu}}{\dual{A_{j},\bm\ell}}\right\}}\nm{x^{k}-x^{*}}$$
	 holds for all $k\geq0$.
\end{proposition}
\begin{proof}
	 By Lemma \ref{l6.2}(i), we have $A\bm\mu\preceq \alpha^{k}\preceq A\bm\ell$. For any $i\neq j$, consider the choice $\alpha^{k}_{i}=\dual{A_{i},\bm\mu}$ and $\alpha^{k}_{j}=\dual{A_{j},\bm\ell}$, then we have
	$$\min\limits_{i\in[l]}\frac{\alpha^{k}_{i}}{\dual{A_{i},\bm\ell}}\min\limits_{i\in[l]}\frac{\dual{A_{i},\bm\mu}}{\alpha^{k}_{i}}\leq\frac{\dual{A_{i},\bm\mu}}{\dual{A_{i},\bm\ell}}\frac{\dual{A_{j},\bm\mu}}{\dual{A_{j},\bm\ell}}.$$
The arbitrary of $i$ and $j$ with $i\neq j$ yields
$$\min\limits_{i\in[l]}\frac{\alpha^{k}_{i}}{\dual{A_{i},\bm\ell}}\min\limits_{i\in[l]}\frac{\dual{A_{i},\bm\mu}}{\alpha^{k}_{i}}\leq\min\limits_{i\neq j,i,j\in[l]}\left\{\frac{\dual{A_{i},\bm\mu}}{\dual{A_{i},\bm\ell}}\frac{\dual{A_{j},\bm\mu}}{\dual{A_{j},\bm\ell}}\right\}.$$
Conversely, notice that 
\begin{center}
	$\min\limits_{i\in[l]}\frac{\alpha^{k}_{i}}{\dual{A_{i},\bm\ell}}\geq \min\limits_{i\in[l]}\frac{\dual{A_{i},\bm\mu}}{\dual{A_{i},\bm\ell}}$~~~and~~~ 	$\min\limits_{i\in[l]}\frac{\dual{A_{i},\bm\mu}}{\alpha^{k}_{i}}\geq \min\limits_{i\in[l]}\frac{\dual{A_{i},\bm\mu}}{\dual{A_{i},\bm\ell}}$,
\end{center}
and no single $\alpha^{k}$ can simultaneously satisfy both equilities unless there exists two distinct indices $i_{0}\neq j_{0}$ such that $$\min\limits_{i\in[l]}\frac{\dual{A_{i},\bm\mu}}{\dual{A_{i},\bm\ell}}=\frac{\dual{A_{i_{0}},\bm\mu}}{\dual{A_{i_{0}},\bm\ell}}=\frac{\dual{A_{j_{0}},\bm\mu}}{\dual{A_{j_{0}},\bm\ell}},$$
which yields
$$\min\limits_{i\in[l]}\frac{\alpha^{k}_{i}}{\dual{A_{i},\bm\ell}}\min\limits_{i\in[l]}\frac{\dual{A_{i},\bm\mu}}{\alpha^{k}_{i}}\geq\min\limits_{i\neq j,i,j\in[l]}\left\{\frac{\dual{A_{i},\bm\mu}}{\dual{A_{i},\bm\ell}}\frac{\dual{A_{j},\bm\mu}}{\dual{A_{j},\bm\ell}}\right\}.$$
Hence, 
$$\min\limits_{i\in[l]}\frac{\alpha^{k}_{i}}{\dual{A_{i},\bm\ell}}\min\limits_{i\in[l]}\frac{\dual{A_{i},\bm\mu}}{\alpha^{k}_{i}}=\min\limits_{i\neq j,i,j\in[l]}\left\{\frac{\dual{A_{i},\bm\mu}}{\dual{A_{i},\bm\ell}}\frac{\dual{A_{j},\bm\mu}}{\dual{A_{j},\bm\ell}}\right\}.$$
The desired result follows by Lemma \ref{l6.2}(ii) and (iv).
\end{proof}
In contrast to the line search method for VOPs, the corresponding method for SOPs preserves the same linear convergence rate as its counterpart without line search. This discrepancy arises because the line search in VOPs is imposed with respect to a \textbf{strict dominance relation} by the underlying partial order. We clarify the distinction with the following example.
\begin{example}
	Let $K=\mathbb{R}^{2}_{+}$ and $A = I_{2}$. Consider first the case where $\alpha^{k}=(\bm\mu_{1},\bm\ell_{2})^{T}$. As illustrated in Fig. \ref{fls}(a), the ratio $\alpha^{k}/t_{k}$ lies on the blue line determined by the strict dominance relation. In this situation, the line search results in a slower linear convergence rate, with the worst-case rate being $\mathcal{O}\left(\left(\sqrt{1-\gamma\frac{\bm\mu_{1}\bm\mu_{2}}{\bm\ell_{1}\bm\ell_{2}}}\right)^{k}\right)$. 
	\begin{figure}[H]
		\centering
		\subfigure[Line search for VOPs]
		{
			\begin{minipage}[H]{.45\linewidth}
				\centering
				\includegraphics[scale=0.22]{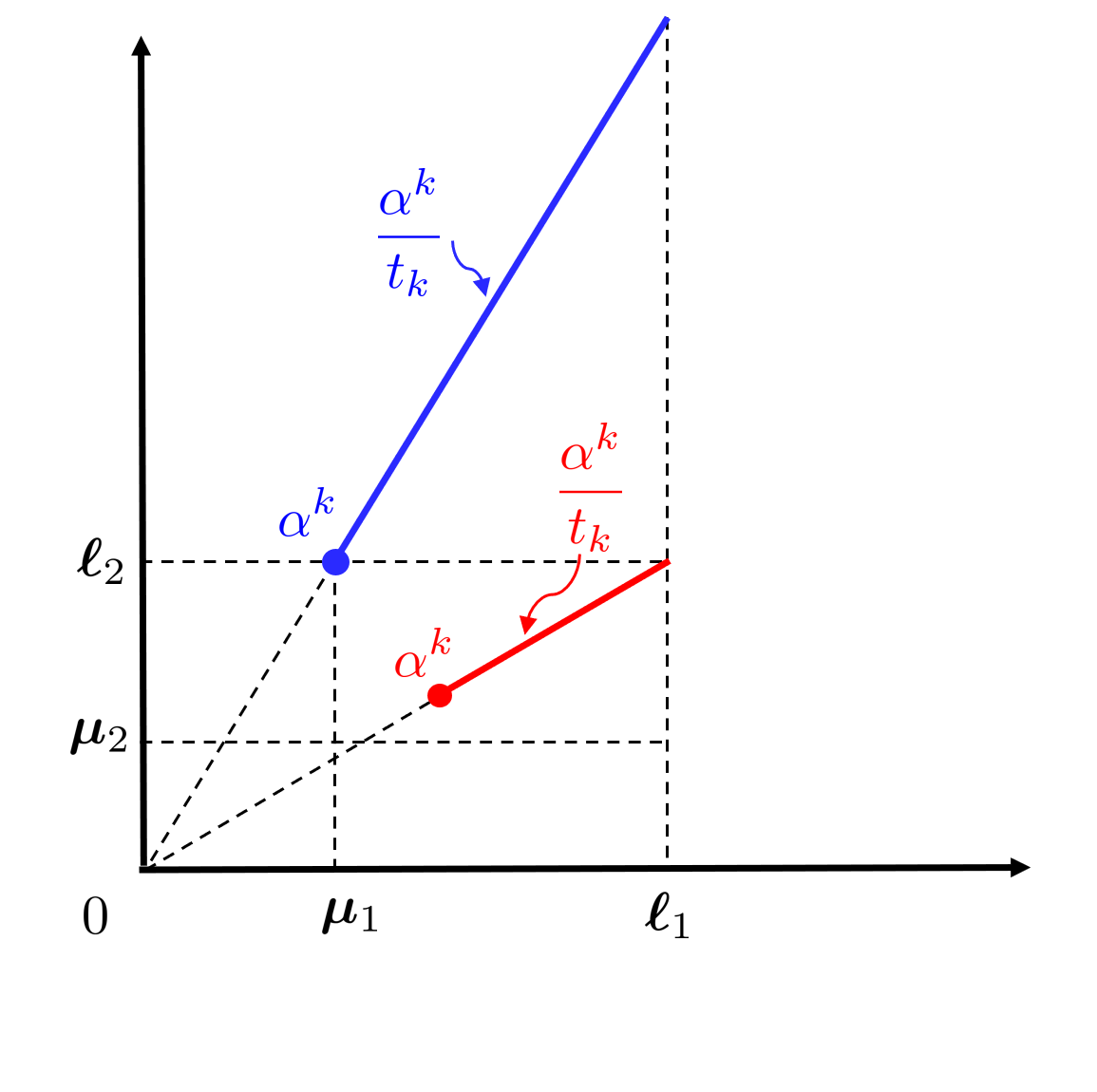} 
			\end{minipage}
		}
		\subfigure[Line search for SOPs]
		{
			\begin{minipage}[H]{.45\linewidth}
				\centering
				\includegraphics[scale=0.22]{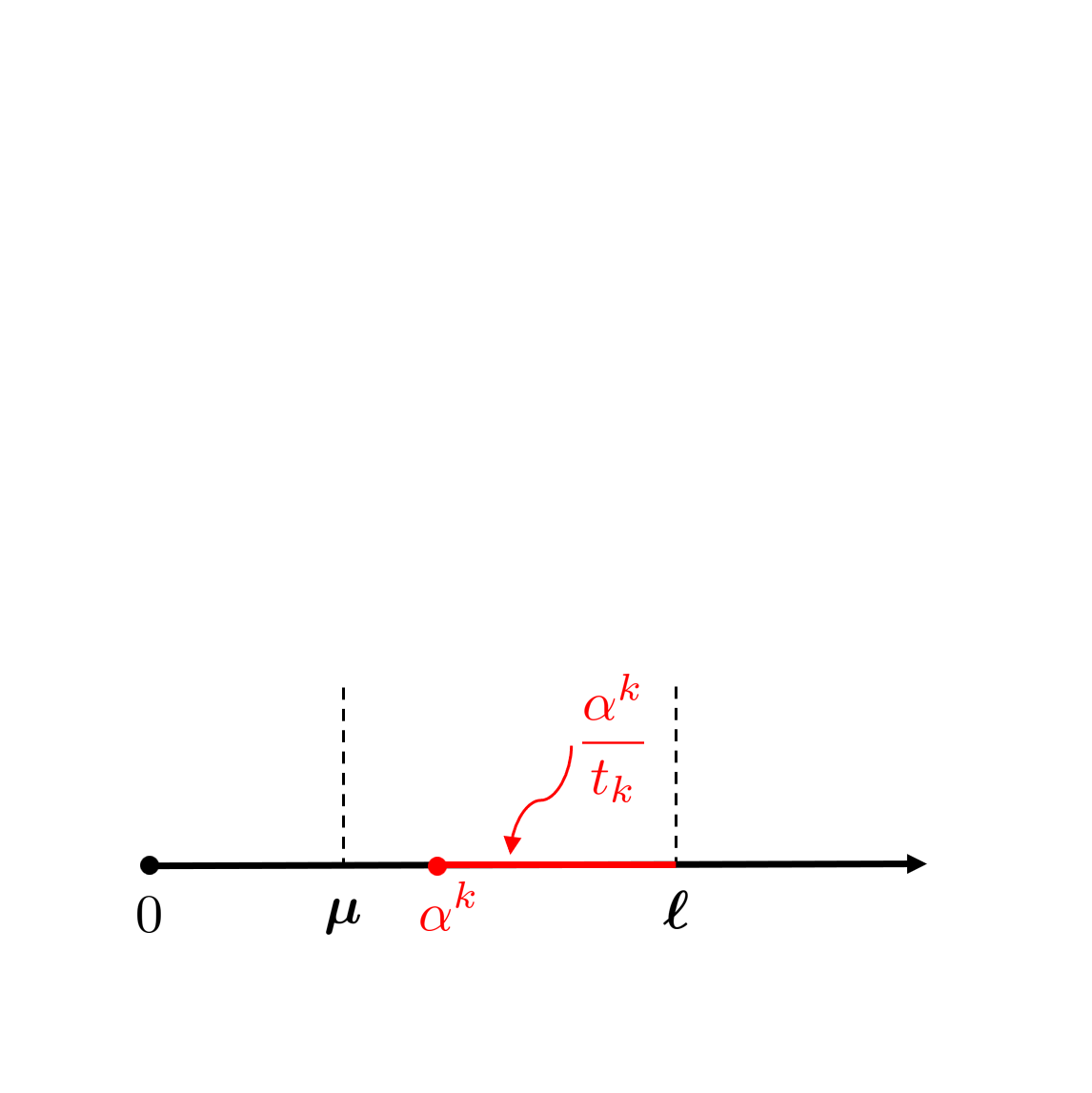} 
			\end{minipage}
		}
		\caption{Differences in line search for VOPs and SOPs.}
		\label{fls}
	\end{figure}
	In contrast, if $\alpha^{k}$ falls on the red point in Fig.~\ref{fls}(a), then $\alpha^{k}/t_{k}$ lies on the red line. In this case, line search for VOPs potentially improves the practical linear convergence rate, with the worst-case rate coinciding with that of the method without line search. As illustrated in Fig. \ref{fls}(b), line search for SOPs tends to improve linear convergence rate, and the worst case rate is the same as that of its counterpart without line search. 
\end{example}

The remaining question is how to preserve the linear convergence rate of descent methods for VOPs without line search when the smoothness parameter $\bm\ell$ is unknown. To address this, we revisit the backtracking method for VOPs, which was first proposed in \citep{CTY2023b}.

\begin{algorithm}
	\caption{Backtracking method for VOPs with polyhedral cones}\label{alg3}
	\LinesNumbered 
	\KwData{$x^{0}\in\mathbb{R}^{n},~A\in\mathcal{A},~0\prec\bm\ell^{0}\preceq A\bm\ell,~\tau>1$}
	\For{$k=0,1,...$}{  Update $\alpha^{k}:=\bm\ell^{k}$\\
		Update $x^{k+1}:=	\mathop{\arg\min}\limits_{x\in\mathbb{R}^{n}}\max\limits_{\lambda\in\Delta^{\alpha^{k}}_{l}}\left\{	\dual{\lambda,AJF(x^{k})(x-x^{k})}+\frac{1}{2}\nm{x-x^{k}}^{2}\right\}$\\
		\eIf{$x^{k+1}=x^{k}$}{ {\bf{return}} $K$-stationary point $x^{k}$ }{$s_{i}=0,i\in[l]$\\
			\Repeat{$\dual{A_{i},F(x^{k+1})-F(x^{k})}\leq\left\langle A_{i},JF(x^{k})(x^{k+1}-x^{k})\right\rangle+\frac{\alpha^{k}_{i}}{2}\|x^{k+1}-x^{k}\|^{2},~i\in[l] $}{
				Update $\alpha_{i}^{k}=\tau^{s_{i}}\bm\ell^{k}_{i},i\in[l]$\\
				Update $x^{k+1}:=	\mathop{\arg\min}\limits_{x\in\mathbb{R}^{n}}\max\limits_{\lambda\in\Delta^{\alpha^{k}}_{l}}\left\{	\dual{\lambda,AJF(x^{k})(x-x^{k})}+\frac{1}{2}\nm{x-x^{k}}^{2}\right\}$\\
				\For{$i=1,\cdots,l$}{
					\If{$\dual{A_{i},F(x^{k+1})-F(x^{k})}>\left\langle A_{i},JF(x^{k})(x^{k+1}-x^{k})\right\rangle+\frac{\alpha^{k}_{i}}{2}\|x^{k+1}-x^{k}\|^{2} $}{Update $s_{i}=s_{i}+1$}
				}
	}}
Update $\bm\ell^{k+1}:=\alpha^{k}/\tau$
}
		
\end{algorithm}

We consider the following surrogate:
\begin{equation}\label{m7}
	G_{k,A,\alpha^{k}}(x):=AJF(x^{k})(x-x^{k})+\frac{1}{2}\nm{x-x^{k}}^{2}\alpha^{k}.
\end{equation}
We can show that $G_{k,A,\alpha^{k}}(\cdot)$ is a non-majorizing surrogate function of $A(F(\cdot)-F(x^{k}))$ near $x^{k}$.
\begin{proposition}
	Let $G_{k,A,\alpha^{k}}(\cdot)$ be defined as (\ref{m7}). Then, the following statements hold.
	\begin{itemize}
		\item[(i)] $A(F(x^{k+1})-F(x^{k}))\preceq G_{k,A,\alpha^{k}}(x^{k+1})$.
		\item[(ii)] If $F(\cdot)$ is $K$-convex, then $G_{k,A,\alpha^{k}}(\cdot)\in\mathcal{S}_{\alpha^{k},\alpha^{k}}(AF,x^{k})$.
		\item[(iii)] If $F(\cdot)$ is strongly $K$-convex with $\bm\mu\in{\rm int}(K)$, then $G_{k,A,\alpha^{k}}(\cdot)\in\mathcal{S}_{\alpha^{k}-A\bm\mu,\alpha^{k}}(AF,x^{k})$.
	\end{itemize}
\end{proposition}
\begin{proof}
	The assertions can be obtained by using the similar arguments as in the proof of Proposition \ref{p5.2}.
\end{proof}

\begin{lemma}\label{l5.5}
	Assume that $F(\cdot)$ is strongly $K$-convex with $\bm\mu\in{\rm int}(K)$, where $K=\{x\in\mathbb{R}^{m}:0\preceq Ax\}.$ Let $\{x^{k}\}$ be the sequence generated by Algorithm \ref{alg6}. Then, the following statements hold:
	\begin{itemize}
		\item[(i)]  $\alpha^{k}\prec \tau A\bm\ell$;
		\item[(ii)] $\{x^{k}\}$ converges to an efficient solution $x^{*}$ of (\ref{VOP});
		\item[(iii)] $\nm{x^{k+1}-x^{*}}\leq\sqrt{1-\tau\min_{i\in [l]}{\dual{A_{i},\bm\mu}}/{\dual{A_{i},\bm\ell}}}\nm{x^{k}-x^{*}},~\forall k\geq0$.
	\end{itemize}
\end{lemma}
\begin{proof}
	(i) Suppose, to the contrary, that $\alpha^{k}_{i}\geq\tau\dual{A_{i},\bm\ell}$ holds for some $i\in[l]$. Then the backtracking procedure would be triggered only when $\alpha^{k}_{i}\geq\dual{A_{i},\bm\ell}$, which contradicts the backtracking condition. Assertions (ii) and (iii) can be obtained by using the similar arguments as in the proof of Lemma \ref{l4.1}.
\end{proof}

In contrast to the line search method for VOPs, the backtracking method preserves the same linear convergence rate as its counterpart  without line search. This discrepancy arises because backtracking for VOPs is imposed with respect to a \textbf{weak dominance relation} induced by the underlying partial order. We clarify the distinction with the following example.

\begin{example}
	Let $K=\mathbb{R}^{2}_{+}$ and $A = I_{2}$. Consider the case where $\bm\ell^{k}=(\bm\mu_{1},\bm\ell_{2})^{T}$. 
	\begin{figure}[htbp]
		\begin{center}	
			\includegraphics[scale=0.22]{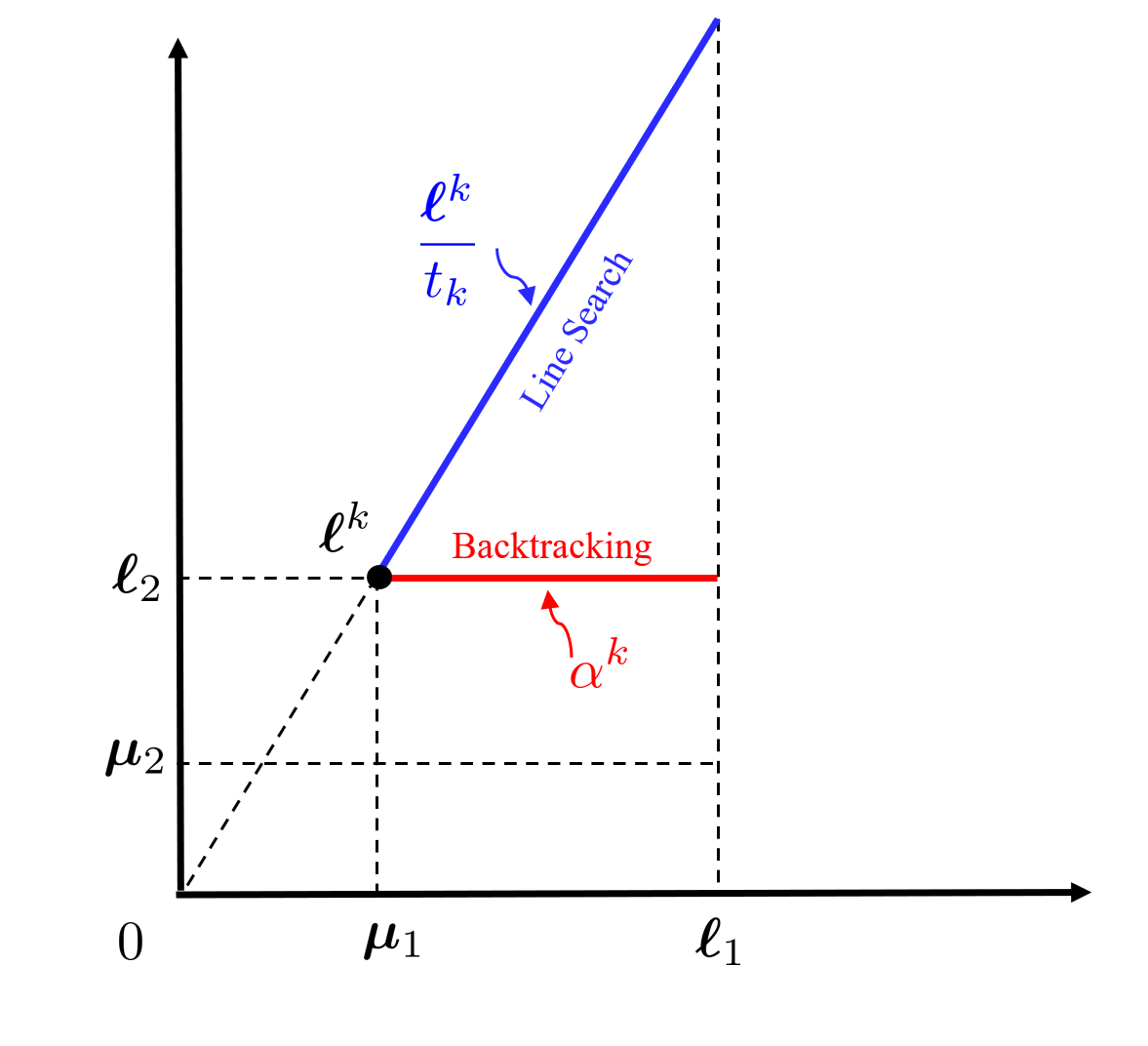} 
		\end{center}
		\caption{Differences between line search and backtracking for VOPs.}\label{fbt}
	\end{figure}

	As illustrated in Fig. \ref{fbt}, the ratio $\bm\ell^{k}/t_{k}$ lies on the blue line determined by the strict dominance relation. In this situation, the line search yields a slower linear convergence rate, with the worst-case rate given by $\mathcal{O}\left(\left(\sqrt{1-\gamma\frac{\bm\mu_{1}\bm\mu_{2}}{\bm\ell_{1}\bm\ell_{2}}}\right)^{k}\right)$. In contrast, the $\alpha^{k}$ lies on the red line determined by the weak dominance relation. In this case, the backtracking procedure for VOPs preserves the same linear convergence rate as its counterpart without line search. 
\end{example}
\begin{remark}
	A notable advantage of backtracking is its ability to adapt to the local smoothness and to preserve the same linear convergence rate as its counterpart  without line search. However, this benefit comes at the price of increased per-iteration computational cost, since backtracking requires repeatedly solving the associated subproblems.
\end{remark}
\section{Numerical Results}\label{sec7} 
In this section, we present numerical results to demonstrate the performance of Barzilai-Borwein descent methods for VOPs (BBDVO) with polyhedral cones.  We also compare BBDVO with steepest descent method for VOPs (SDVO) and equiangular direction method \citep{ADN2020} for VOPs (EDVO). All numerical experiments were implemented in Python 3.7 and executed on a personal computer with an Intel Core i7-11390H, 3.40 GHz processor, and 16 GB of RAM.
\subsection{Implementation details}
 For all tested algorithms, we used Armijo line search (\ref{armijo}) with $\sigma=10^{-4}$ and $\gamma=0.5$.  The test algorithms were executed on several test problems, and the problem illustration is given in Table \ref{tab1}. The dimensions of variables and objective functions are presented in the second and third columns, respectively. $x_{L}$ and $x_{U}$ represent lower bounds and upper bounds of variables, respectively.

\begin{table}[H]
	\begin{center}
		\caption{Description of all test problems used in numerical experiments}
		\label{tab1}\vspace{-5mm}
	\end{center}
	\centering
	\resizebox{.8\columnwidth}{!}{
		\begin{tabular}{lllllllllll}
			\hline
			Problem    &  & $n$     &           & $m$     &  & $x_{L}$               &  & $x_{U}$             &  & Reference \\ \hline
				BK1        &  & 2     & \textbf{} & 2     &  & (-5,-5)         &  & (10,10)       &  & \citep{BK1}         \\
			DD1        &  & 5     & \textbf{} & 2     &  & (-20,...,-20)   &  & (20,...,20)   &  & \citep{DD1998}         \\
			Deb        &  & 2     & \textbf{} & 2     &  & (0.1,0.1)       &  & (1,1)         &  & \citep{D1999}          \\
			FF1        &  & 2     & \textbf{} & 2     &  & (-1,-1)         &  & (1,1)         &  & \citep{BK1}         \\
			Hil1       &  & 2     & \textbf{} & 2     &  & (0,0)           &  & (1,1)         &  & \citep{Hil1}         \\
			Imbalance1 &  & 2     & \textbf{} & 2     &  & (-2,-2)         &  & (2,2)         &  & \citep{CTY2023a}         \\
			JOS1a      &  & 50    & \textbf{} & 2     &  & (-2,...,-2)     &  & (2,...,2)     &  & \citep{JO2001}         \\
			LE1        &  & 2     & \textbf{} & 2     &  & (-5,-5)         &  & (10,10)       &  & \citep{BK1}        \\
			PNR        &  & 2     & \textbf{} & 2     &  & (-2,-2)         &  & (2,2)         &  & \citep{PN2006}        \\
			WIT1       &  & 2     & \textbf{} & 2     &  & (-2,-2)         &  & (2,2)         &  & \citep{W2012}       \\
	         \hline
		\end{tabular}
	}
\end{table}
For the tested problems, the partial order are induced by polyhedral cones $\mathbb{R}^{2}_{+}$, $K_{1}$, and $K_{2}$, respectively, where
$$K_{1}:=\{x\in\mathbb{R}^{2}:5x_{1}-x_{2}\geq0,~-x_{1}+5x_{2}\geq0\}\subseteq\mathbb{R}^{2}_{+},$$
and
$$K_{2}:=\{x\in\mathbb{R}^{2}:5x_{1}+x_{2}\geq0,~x_{1}+5x_{2}\geq0\}\supseteq\mathbb{R}^{2}_{+}.$$
The polyhedral cones are illustrated in Fig. \ref{fig2}.
\begin{figure}[htbp]
	\begin{center}	
		\includegraphics[scale=0.5]{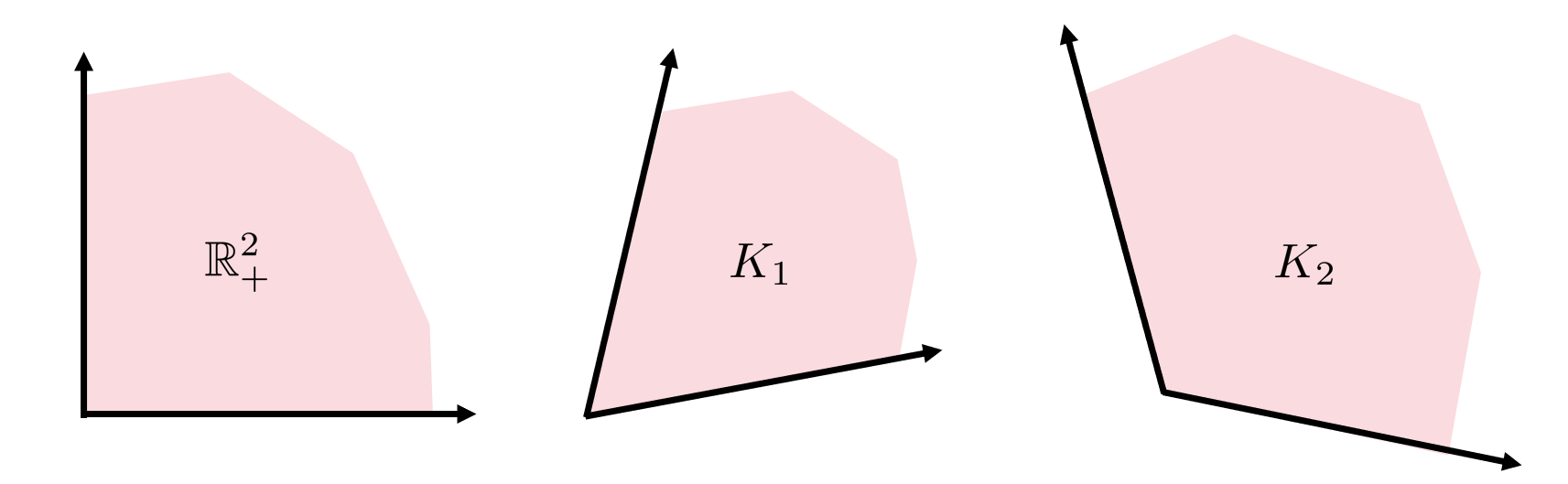} 
	\end{center}
	\caption{Illustration of the polyhedral cones.}\label{fig2}
	
\end{figure}

 For each problem, we used the same initial points for different tested algorithms. The initial points were randomly selected within the specified lower and upper bounds. Dual subproblems of different algorithms were efficiently solved by Frank-Wolfe method. To guarantee a fair comparison, we decided to let the algorithms run until one of the following stopping conditions was satisfied:
\begin{itemize}
	\item the current solution satisfies $\nm{\bar{d}^{k}}\leq10^{-6}$, where $\bar{d}$ is the steepest descent direction (\ref{sdp}) with $A=\bar{A}$ and $\bar{A}\in\mathcal{A}$ with row vectors $\nm{\bar{A}_{i}}=1$ for $i=1,...,l$;
	\item the number of iterations reaches 500.
\end{itemize} 
The recorded averages from the 200 runs include the number of iterations, the number of function evaluations, and the CPU time. The performance profiles \citep{DM2002} in terms of iterations, function evaluations and CPU time are used to illustrate the overall performance of the 200 runs.
\subsection{Numerical results for VOPs with $K=\mathbb{R}^{2}_{+}$}
In this case, we denote the set of transform matrices $\mathcal{A}_{0}:=\{A:A\mathbb{R}^{2}_{+}=\mathbb{R}^{2}_{+}\}$. For SDVO, we choose $A^{0}=I_{2}\in\mathcal{A}_{0}$ and $\hat{A}^{0}\in\mathcal{A}_{0}$ in subproblem, respectively, where $\hat{A}^{0}:=\{A:A_{i}=A^{0}_{i}/\max\{1,\nm{\nabla F_{i}(x^{0})}_{\infty}\},~i=1,2\}$\footnote{The scale strategy is initially proposed in \citep{GLP2022} due to numerical reasons.}. For EDVO, normalization is applied for each of gradients in the transformed subproblem, which implies that EDVO is also not sensitive to the choice of transform matrix. As a result, we choose $A=A^{0}$ in subproblems of EDVO and BBDVO.
\begin{table}[H]
	\begin{center}
		\caption{Number of average iterations (iter), number of average function evaluations (feval), and average CPU time (time($ms$)) of SDVO, EDVO and BBDVO implemented on different test problems with {$K=\mathbb{R}^{2}_{+}$}}\label{tab2}\vspace{-2mm}
	\end{center}
	\centering
	\resizebox{.95\columnwidth}{!}{
		\begin{tabular}{lrrrrrrrrrrrrrrr}
			\hline
			Problem &
			\multicolumn{3}{l}{SDVO with $A=A^{0}$} &
			&
			\multicolumn{3}{l}{SDVO with $A=\hat{A}^{0}$} &
			\multicolumn{1}{l}{} &
			\multicolumn{3}{l}{EDVO with $A=A^{0}$} &
			\multicolumn{1}{l}{} &
			\multicolumn{3}{l}{BBDVO with $A=A^{0}$} \\ \cline{2-4} \cline{6-8} \cline{10-12}  \cline{14-16}
			&
			iter &
			feval &
			time &
			\textbf{} &
			iter &
			feval &
			time &
			&
			iter &
			feval &
			time &
			&
			iter &
			feval &
			time \\ \hline
			BK1        & \textbf{1.00} & 2.00  & 0.24  &  & 36.93 & 40.15 & 4.03  &  & 30.12 & 34.07 & 2.67  &  & \textbf{1.00} & \textbf{1.00} & \textbf{0.21} \\
			DD1        & 70.95 & 199.40 & 12.14 &  & 248.33 & 250.81 & 30.41 &  & 376.96 & 376.96 & 38.11 &  & \textbf{7.33} & \textbf{8.51} & \textbf{1.24} \\
			Deb        & 57.79 & 376.02 & 13.67 &  & 5.67  & 13.15 & 1.36  &  & 10.48 & 26.39 & 1.43  &  & \textbf{4.51} & \textbf{6.67} & \textbf{0.79} \\
			FF1        & 28.95 & 29.11  & 3.93  &  & 8.22  & 9.15  & 1.38  &  & 8.68  & 9.51  & 1.03  &  & \textbf{4.68} & \textbf{5.90} & \textbf{0.84} \\
			Hil1       & 24.89 & 82.53  & 5.15  &  & 15.62 & 36.52 & 3.05  &  & 20.09 & 51.60 & 3.06  &  & \textbf{11.42} & \textbf{12.25} & \textbf{2.07} \\
			Imbalance1 & 88.23 & 178.31 & 12.35 &  & 387.91 & 388.32 & 43.47 &  & 420.04 & 420.99 & 42.01 &  & \textbf{2.62} & \textbf{3.60} & \textbf{0.50} \\
			JOS1a      & 302.72 & 302.72 & 35.80 &  & 16.51 & 18.03 & 3.40  &  & 74.13 & 74.44 & 9.29  &  & \textbf{1.00} & \textbf{1.00} & \textbf{0.26} \\
			LE1        & 22.44 & 52.57  & 3.72  &  & 7.98  & 29.78 & 1.70  &  & 11.80 & 29.93 & 1.47  &  & \textbf{4.65} & \textbf{7.13} & \textbf{0.82} \\
			PNR        & 11.24 & 48.40  & 2.83  &  & 12.32 & 16.09 & 1.85  &  & 14.37 & 23.24 & 1.75  &  & \textbf{4.28} & \textbf{4.77} & \textbf{0.71} \\
			WIT1       & 73.30 & 461.37 & 19.93 &  & 139.88 & 146.58 & 15.15 &  & 10.64 & 17.97 & 1.33  &  & \textbf{3.59} & \textbf{3.68} & \textbf{0.62} \\
			\hline
		\end{tabular}
	}
\end{table}

\begin{figure}[H]
	\centering
	\subfigure[Iterations]
	{
		\begin{minipage}[H]{.3\linewidth}
			\centering
			\includegraphics[scale=0.25]{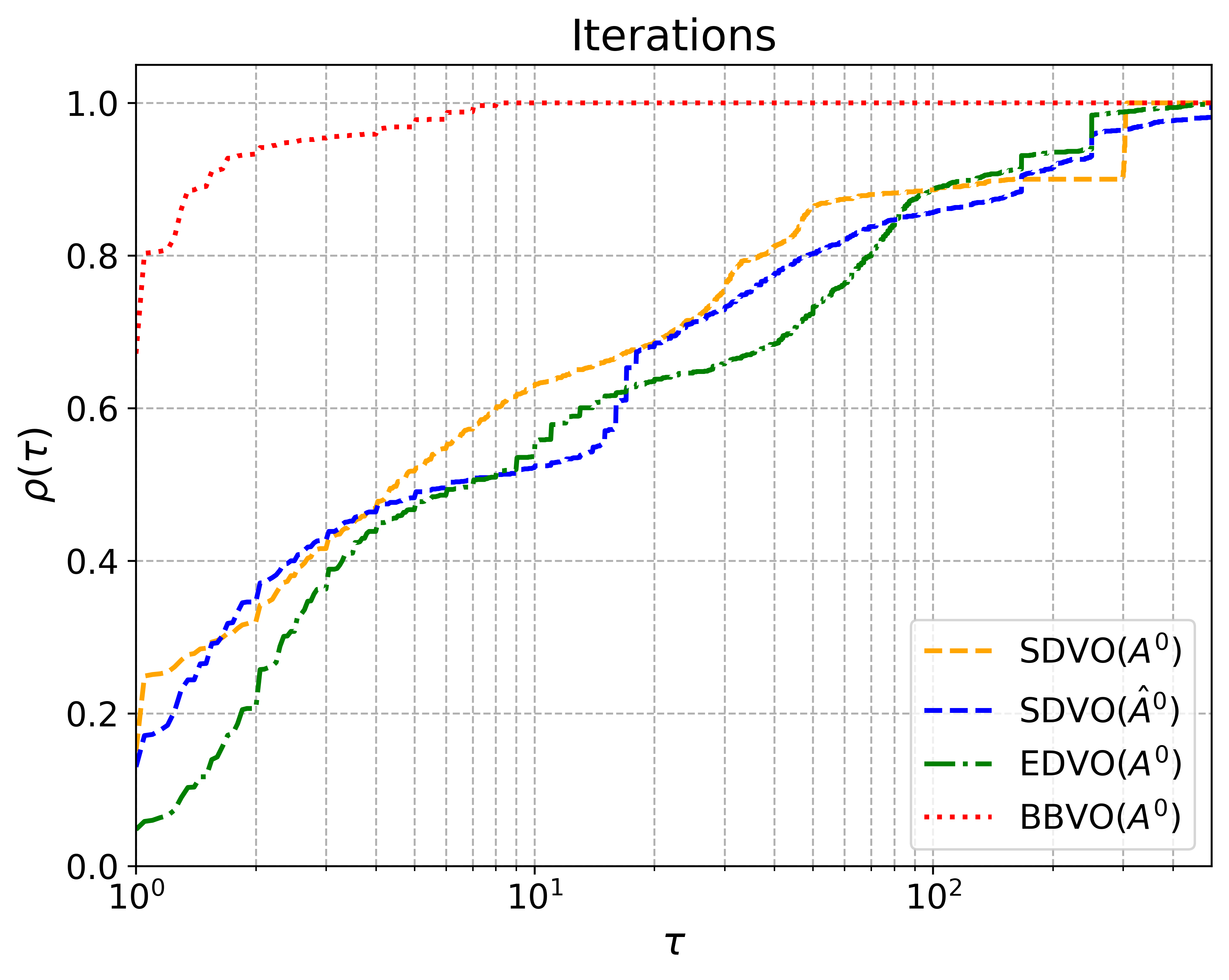} 
		\end{minipage}
	}
	\subfigure[Function Evaluations]
	{
		\begin{minipage}[H]{.3\linewidth}
			\centering
			\includegraphics[scale=0.25]{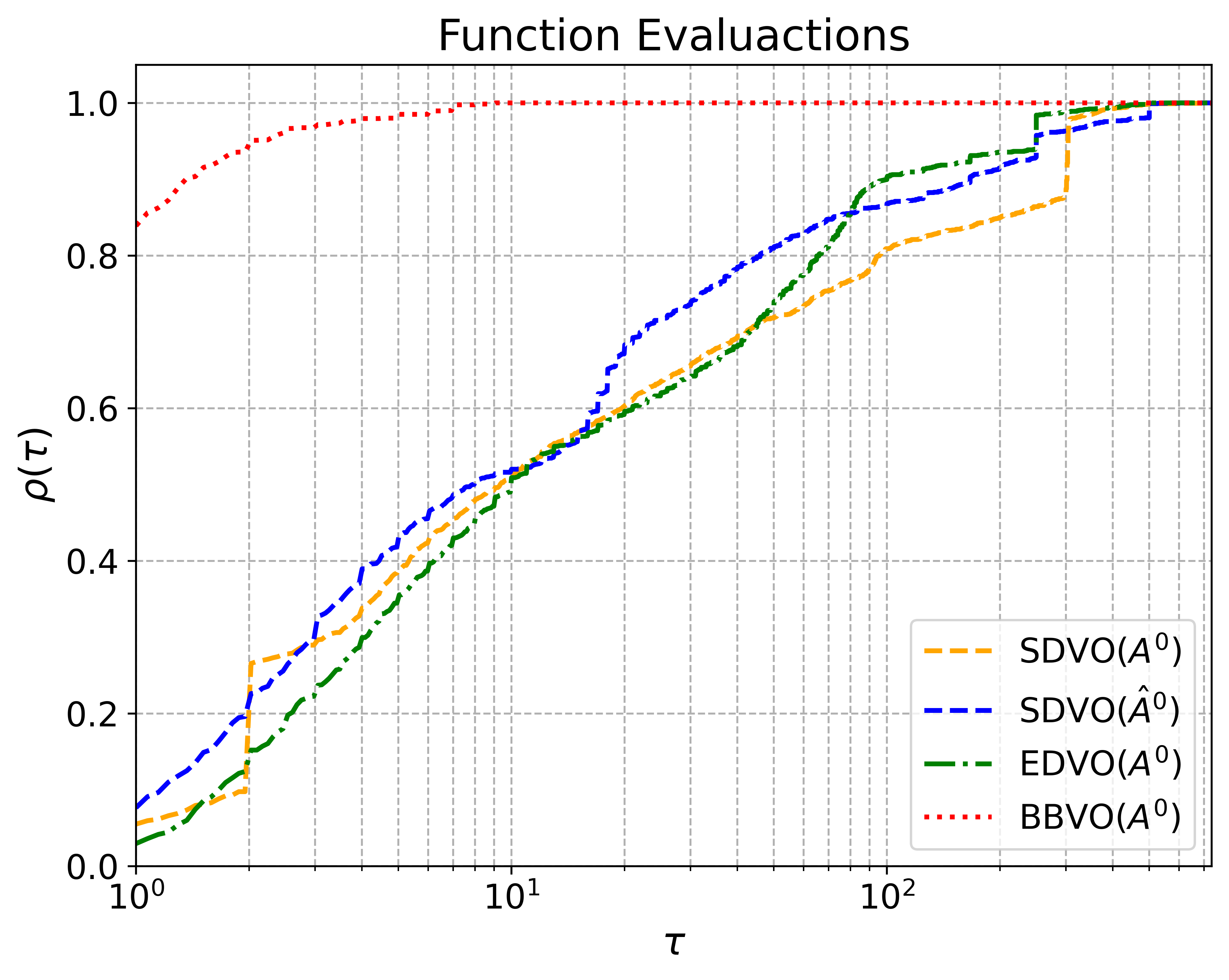} 
		\end{minipage}
	}
	\subfigure[CPU Time]
	{
		\begin{minipage}[H]{.3\linewidth}
			\centering
			\includegraphics[scale=0.25]{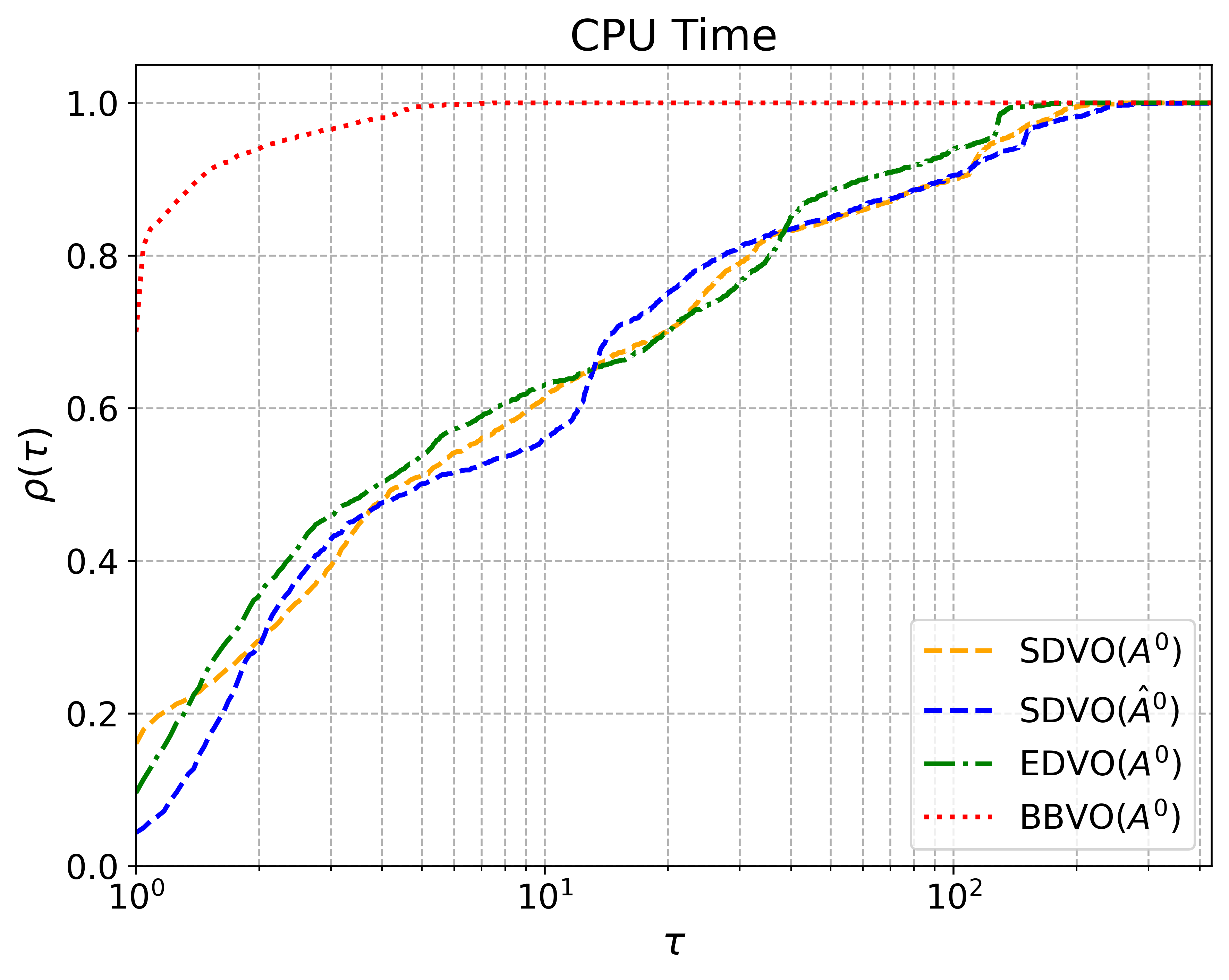} 
		\end{minipage}
	}
	\caption{Performance profiles on the test problems in Table \ref{tab2} with $K=\mathbb{R}^{2}_{+}$.}
	\label{fk0}
\end{figure}
\subsection{Numerical results for VOPs with $K=K_{1}$}
In this case, we denote the set of transform matrices $\mathcal{A}_{1}:=\{A:AK_{1}=\mathbb{R}^{2}_{+}\}$. For SDVO, we choose $A^{1}=\begin{pmatrix} 5&\ -1\\-1&\ 5 \end{pmatrix}\in\mathcal{A}_{0}$ and $\hat{A}^{1}\in\mathcal{A}_{1}$ in subproblem, respectively, where $\hat{A}^{1}:=\{A:A_{i}=A^{1}_{i}/\max\{1,\nm{\nabla F_{i}(x^{0})}_{\infty}\},~i=1,2\}$.
\begin{table}[H]
	\begin{center}
		\caption{Number of average iterations (iter), number of average function evaluations (feval), and average CPU time (time($ms$)) of SDVO, EDVO and BBDVO implemented on different test problems with {$K=K_{1}$}}\label{tab3}\vspace{-2mm}
	\end{center}
	\centering
	\resizebox{.95\columnwidth}{!}{
		\begin{tabular}{lrrrrrrrrrrrrrrr}
			\hline
			Problem &
			\multicolumn{3}{l}{SDVO with $A=A^{1}$} &
			&
			\multicolumn{3}{l}{SDVO with $A=\hat{A}^{1}$} &
			\multicolumn{1}{l}{} &
			\multicolumn{3}{l}{EDVO with $A=A^{1}$} &
			\multicolumn{1}{l}{} &
			\multicolumn{3}{l}{BBDVO with $A=A^{1}$} \\ \cline{2-4} \cline{6-8} \cline{10-12}  \cline{14-16} 
			&
			iter &
			feval &
			time &
			\textbf{} &
			iter &
			feval &
			time &
			&
			iter &
			feval &
			time &
			&
			iter &
			feval &
			time \\ \hline
			BK1        & \textbf{1.00} & 2.65 & \textbf{0.21} &  & 7.51 & 9.87 & 1.21 &  & 48.30 & 51.56 & 4.07 &  & \textbf{1.00} & \textbf{1.00} & 0.22 \\
			DD1        & 92.40 & 373.32 & 18.42 &  & 118.35 & 120.30 & 15.34 &  & 496.99 & 496.99 & 53.13 &  & \textbf{41.85} & \textbf{47.12} & \textbf{6.39} \\
			Deb        & 102.21 & 876.19 & 27.68 &  & \textbf{27.82} & 146.28 & \textbf{6.29} &  & 81.45 & 161.00 & 9.44 &  & 35.69 & \textbf{71.88} & 6.48 \\
			FF1        & 33.19 & 105.29 & 5.85 &  & 17.23 & 56.82 & 3.16 &  & 45.35 & 65.10 & 4.73 &  & \textbf{16.00} & \textbf{16.95} & \textbf{2.61} \\
			Hil1       & 28.28 & 151.71 & 7.15 &  & 19.41 & 86.98 & 4.71 &  & 23.09 & 56.07 & 3.51 &  & \textbf{17.74} & \textbf{18.35} & \textbf{3.10} \\
			Imbalance1 & 77.34 & 309.32 & 13.55 &  & 397.37 & 397.44 & 46.97 &  & 500.00 & 500.00 & 50.14 &  & \textbf{28.78} & \textbf{31.18} & \textbf{4.49} \\
			JOS1a      & 69.86 & 69.86 & 6.82 &  & 6.57 & 17.06 & 2.55 &  & 117.30 & 117.43 & 13.78 &  & \textbf{1.00} & \textbf{1.00} & \textbf{0.26} \\
			LE1        & 14.02 & 69.28 & 2.88 &  & 12.37 & 55.82 & 2.46 &  & 15.99 & 39.60 & 1.92 &  & \textbf{6.31} & \textbf{7.49} & \textbf{1.14} \\
			PNR        & 27.10 & 157.72 & 5.71 &  & 12.76 & 37.65 & 2.15 &  & 21.80 & 27.54 & 2.41 &  & \textbf{9.78} & \textbf{10.97} & \textbf{1.64} \\
			WIT1       & 244.44 & 2084.16 & 111.19 &  & 199.03 & 245.92 & 24.42 &  & 414.29 & 416.78 & 42.52 &  & \textbf{158.78} & \textbf{164.91} & \textbf{22.80} \\
			\hline
		\end{tabular}
	}
\end{table}
\begin{figure}[H]
	\centering
	\subfigure[Iterations]
	{
		\begin{minipage}[H]{.3\linewidth}
			\centering
			\includegraphics[scale=0.25]{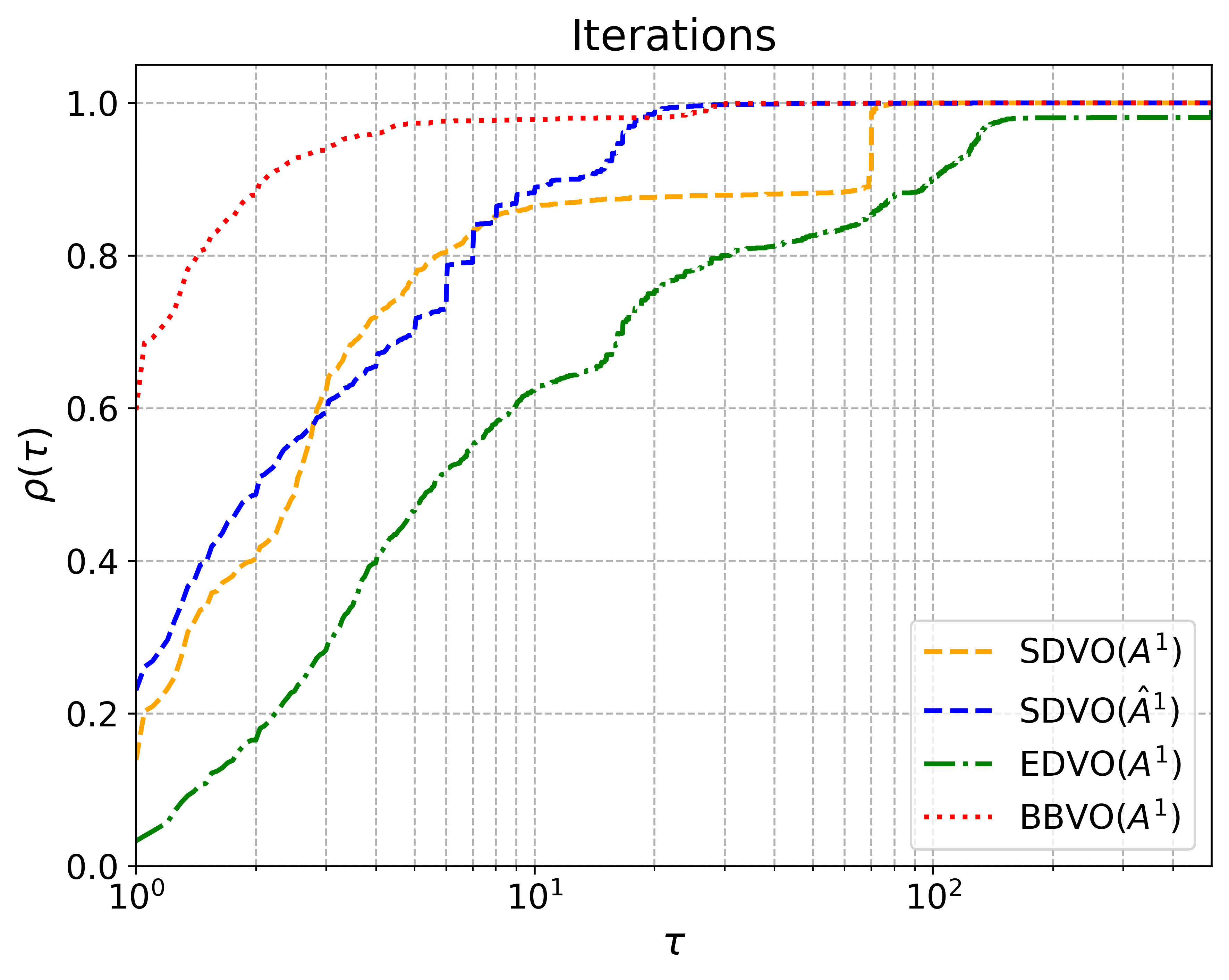} 
		\end{minipage}
	}
	\subfigure[Function Evaluations]
	{
		\begin{minipage}[H]{.3\linewidth}
			\centering
			\includegraphics[scale=0.25]{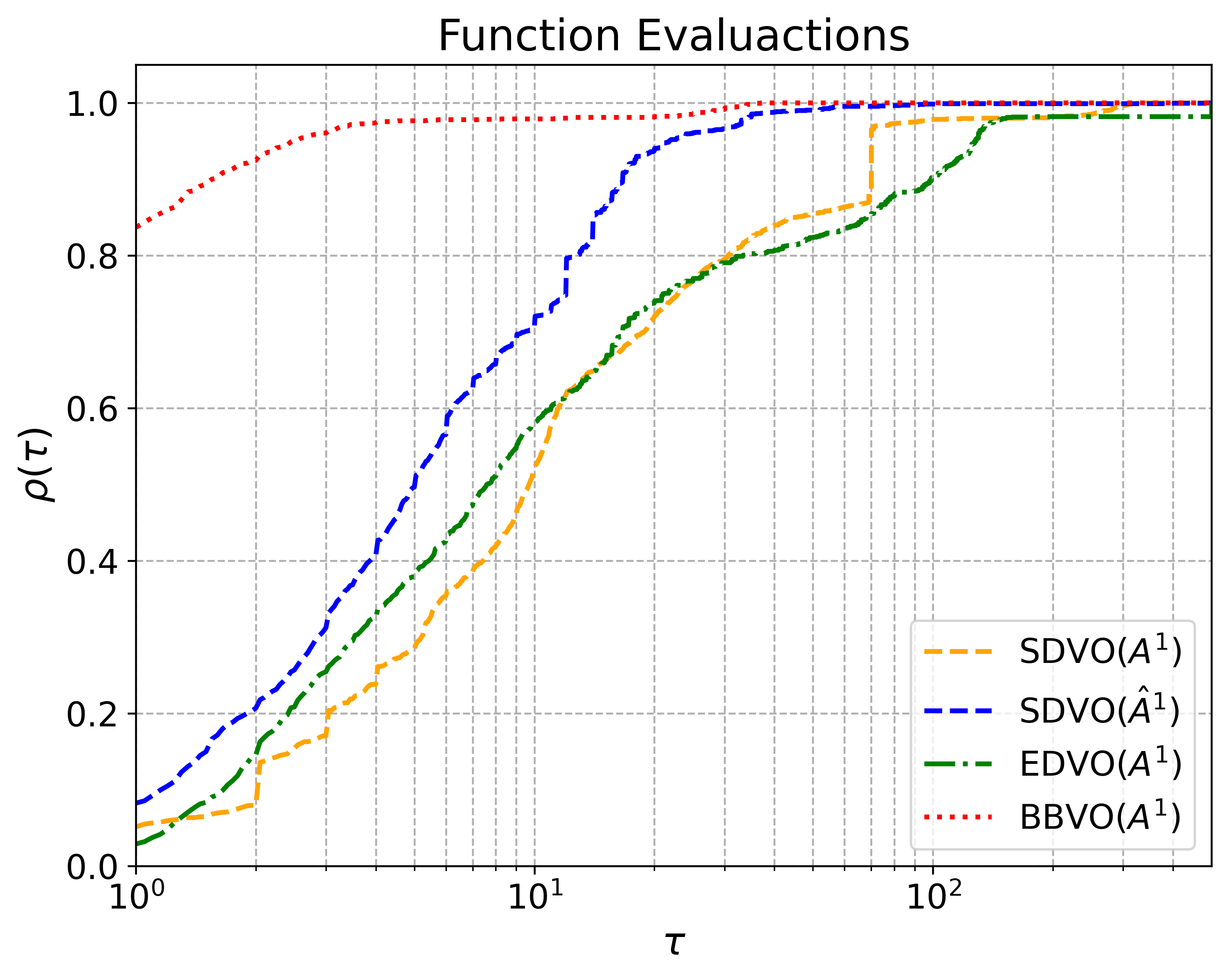} 
		\end{minipage}
	}
	\subfigure[CPU Time]
	{
		\begin{minipage}[H]{.3\linewidth}
			\centering
			\includegraphics[scale=0.25]{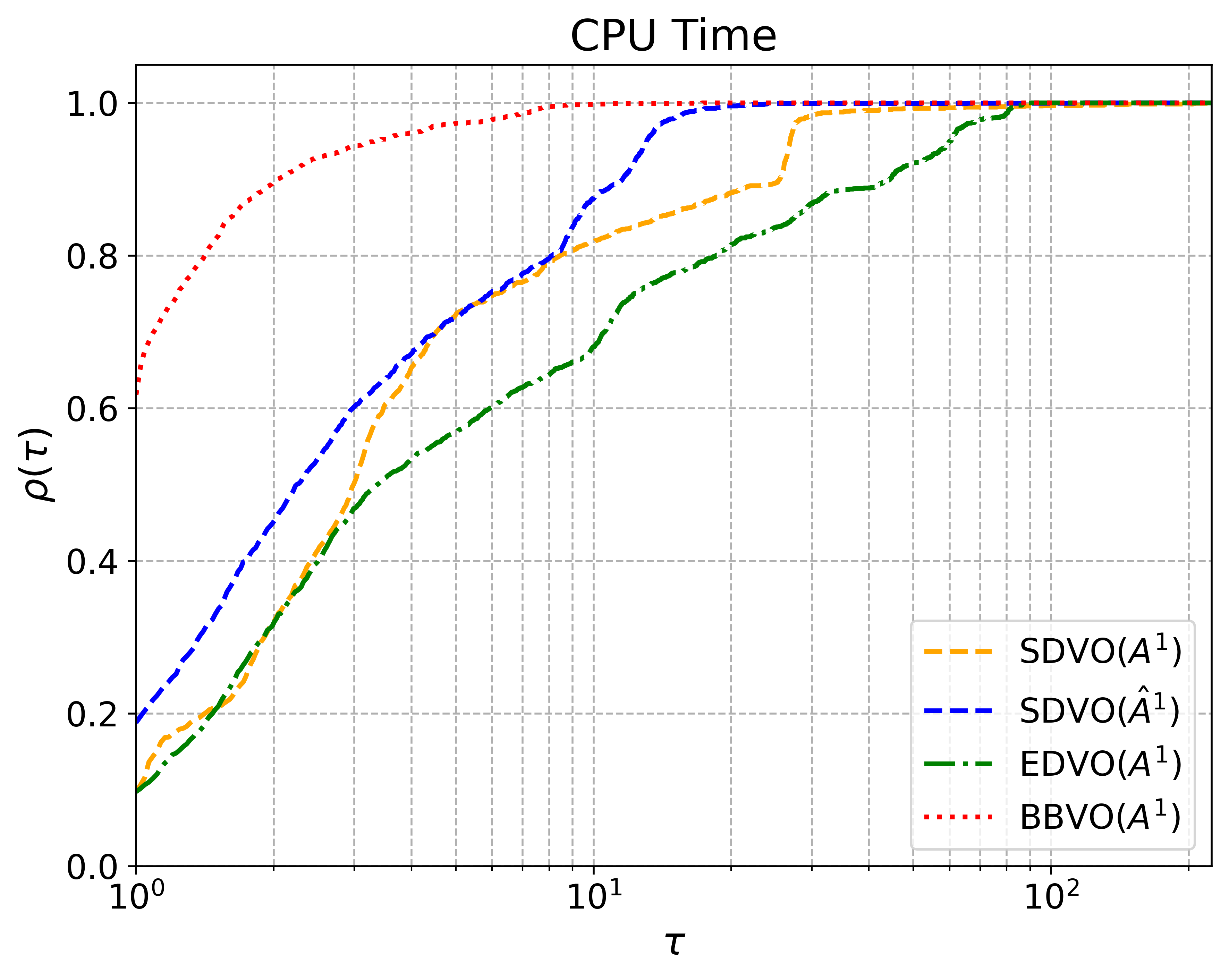} 
		\end{minipage}
	}
	\caption{Performance profiles on the test problems in Table \ref{tab2} with $K=K_{1}$.}
	\label{fk1}
\end{figure}

\subsection{Numerical results for VOPs with $K=K_{2}$}
We denote the set of transform matrices $\mathcal{A}_{2}:=\{A:AK_{2}=\mathbb{R}^{2}_{+}\}$. For SDVO, we choose $A^{2}=\begin{pmatrix} 5&\ 1\\1&\ 5 \end{pmatrix}\in\mathcal{A}_{0}$ and $\hat{A}^{2}\in\mathcal{A}_{2}$ in subproblem, respectively, where $\hat{A}^{2}:=\{A:A_{i}=A^{2}_{i}/\max\{1,\nm{\nabla F_{i}(x^{0})}_{\infty}\},~i=1,2\}$.
\begin{table}[H]
	\begin{center}
		\caption{Number of average iterations (iter), number of average function evaluations (feval), and average CPU time (time($ms$)) of SDVO, EDVO and BBDVO implemented on different test problems with {$K=K_{2}$}}\label{tab4}\vspace{-2mm}
	\end{center}
	\centering
	\resizebox{.95\columnwidth}{!}{
		\begin{tabular}{lrrrrrrrrrrrrrrr}
			\hline
			Problem &
			\multicolumn{3}{l}{SDVO with $A=A^{2}$} &
			&
			\multicolumn{3}{l}{SDVO with $A=\hat{A}^{2}$} &
			\multicolumn{1}{l}{} &
			\multicolumn{3}{l}{EDVO with $A=A^{2}$} &
			\multicolumn{1}{l}{} &
			\multicolumn{3}{l}{BBDVO with $A=A^{2}$} \\ \cline{2-4} \cline{6-8} \cline{10-12}  \cline{14-16} 
			&
			iter &
			feval &
			time &
			\textbf{} &
			iter &
			feval &
			time &
			&
			iter &
			feval &
			time &
			&
			iter &
			feval &
			time \\ \hline
			BK1        & 22.29 & 85.27  & 2.52  &  & 6.67  & 11.79 & 1.25  &  & 20.69 & 26.53 & 1.92  &  & \textbf{1.00} & \textbf{1.00} & \textbf{0.21} \\
			DD1        & 17.16 & 62.14  & 3.37  &  & 46.38 & 47.71 & 5.03  &  & 133.89 & 134.00 & 13.33 &  & \textbf{4.84} & \textbf{5.29} & \textbf{0.83} \\
			Deb        & 19.48 & 164.15 & 5.43  &  & 18.22 & 135.69 & 4.63  &  & 21.95 & 141.27 & 4.40  &  & \textbf{9.47} & \textbf{48.94} & \textbf{2.02} \\
			FF1        & 13.52 & 18.80  & 1.96  &  & 14.76 & 67.80 & 3.04  &  & 17.68 & 122.48 & 4.04  &  & \textbf{4.72} & \textbf{5.74} & \textbf{0.91} \\
			Hil1       & 14.69 & 75.19  & 3.86  &  & 17.60 & 82.61 & 4.44  &  & 16.61 & 42.27 & 2.65  &  & \textbf{8.38} & \textbf{9.24} & \textbf{1.52} \\
			Imbalance1 & 25.76 & 105.22 & 4.60  &  & 500.00 & 610.90 & 61.64 &  & 500.00 & 500.04 & 48.91 &  & \textbf{4.35} & \textbf{5.76} & \textbf{0.73} \\
			JOS1a      & 46.11 & 46.11  & 4.47  &  & 4.58  & 10.16 & 1.77  &  & 48.35 & 51.67 & 6.34  &  & \textbf{1.00} & \textbf{1.00} & \textbf{0.26} \\
			LE1        & 9.64  & 47.62  & 2.05  &  & 14.79 & 79.00 & 3.08  &  & 11.40 & 54.09 & 1.95  &  & \textbf{7.58} & \textbf{43.00} & \textbf{1.63} \\
			PNR        & 17.22 & 103.38 & 3.60  &  & 13.30 & 50.20 & 2.41  &  & 11.12 & 36.16 & 1.69  &  & \textbf{6.80} & \textbf{8.83} & \textbf{1.04} \\
			WIT1       & 23.12 & 140.31 & 5.32  &  & 311.78 & 975.91 & 44.38 &  & 10.46 & 29.67 & 1.52  &  & \textbf{8.66} & \textbf{9.95} & \textbf{1.27} \\
			\hline
		\end{tabular}
	}
\end{table}
\begin{figure}[H]
	\centering
	\subfigure[Iterations]
	{
		\begin{minipage}[H]{.3\linewidth}
			\centering
			\includegraphics[scale=0.25]{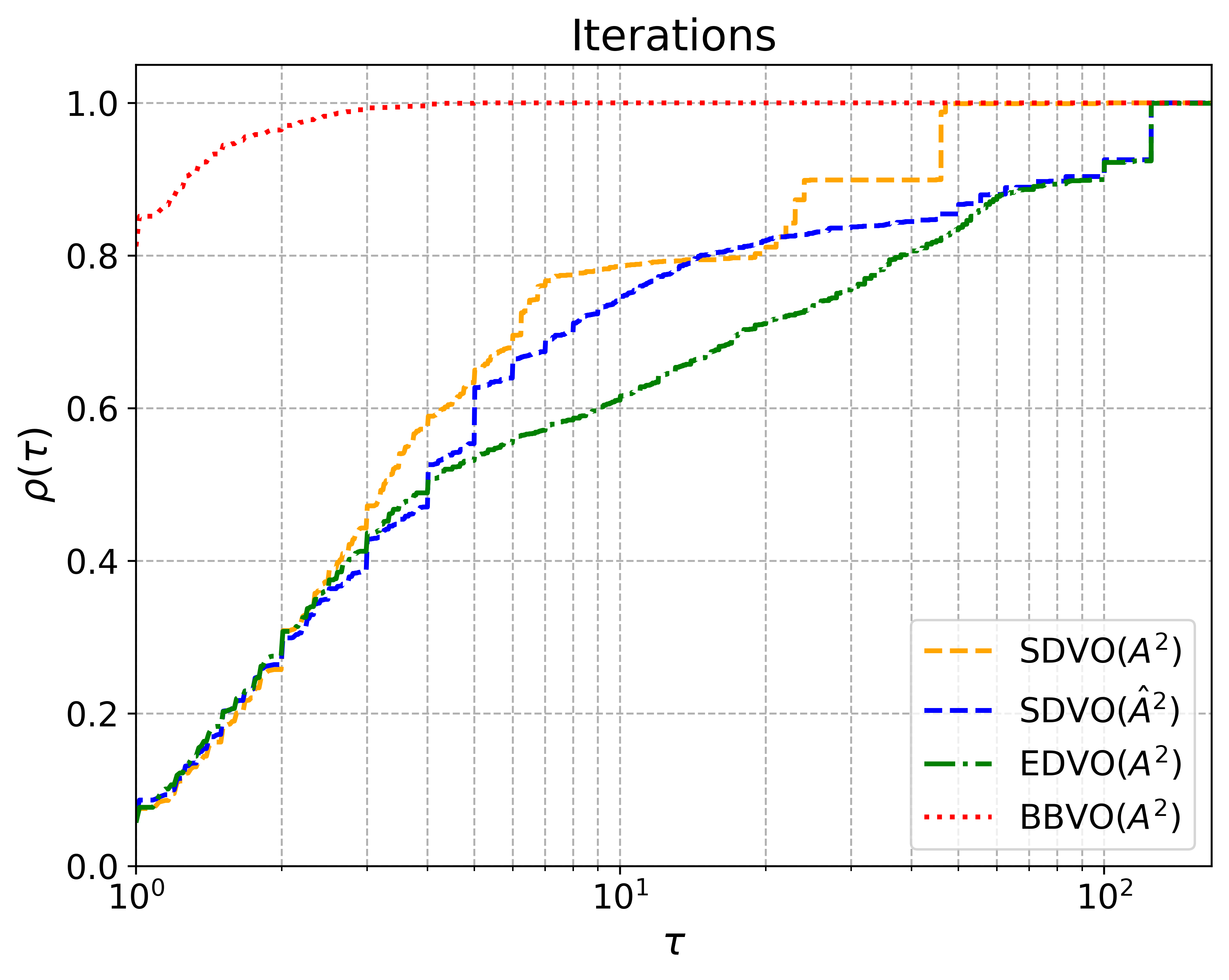} 
		\end{minipage}
	}
	\subfigure[Function Evaluations]
	{
		\begin{minipage}[H]{.3\linewidth}
			\centering
			\includegraphics[scale=0.25]{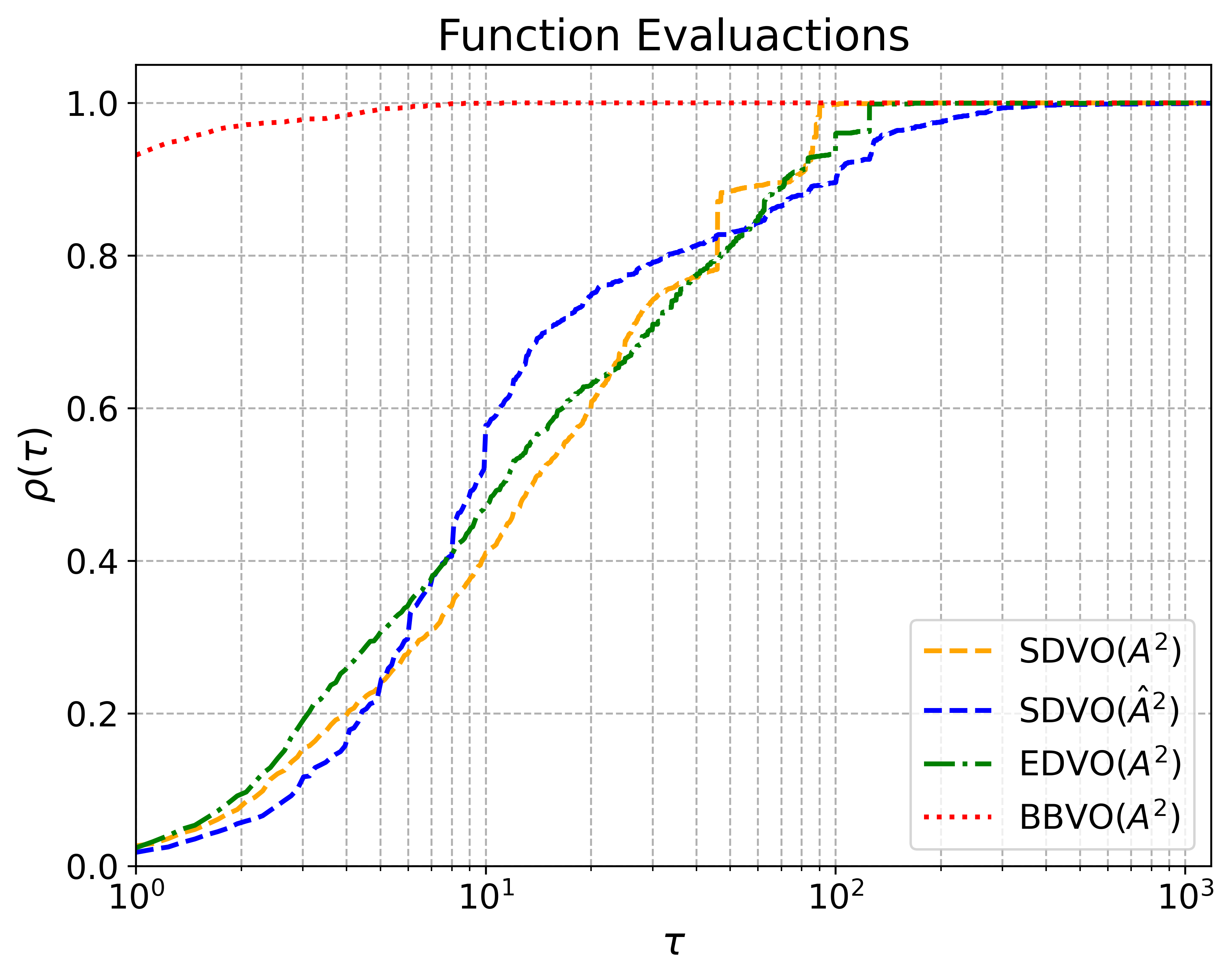} 
		\end{minipage}
	}
	\subfigure[CPU Time]
	{
		\begin{minipage}[H]{.3\linewidth}
			\centering
			\includegraphics[scale=0.25]{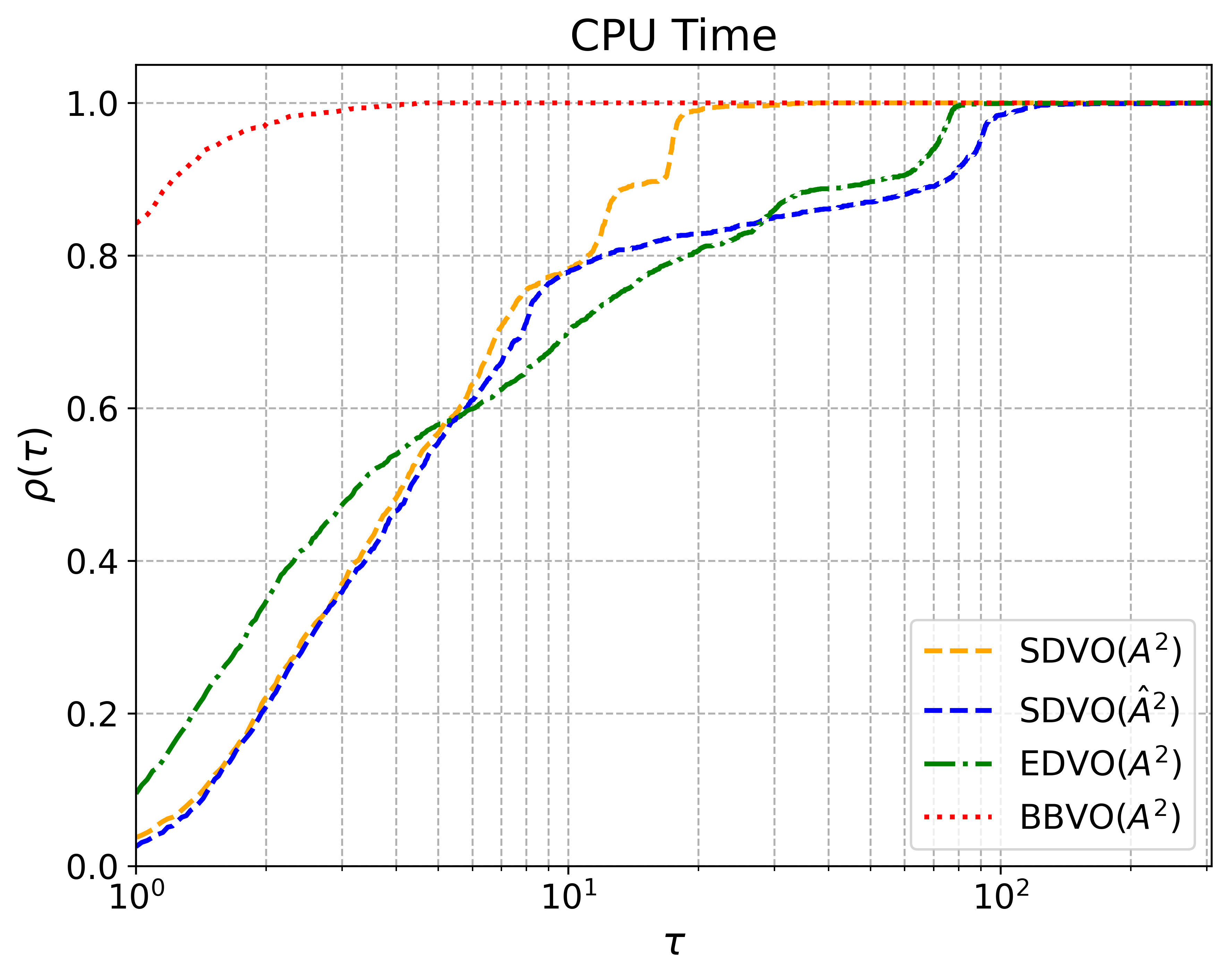} 
		\end{minipage}
	}
	\caption{Performance profiles on the test problems in Table \ref{tab2} with $K=K_{2}$.}
	\label{fk2}
\end{figure}

\noindent\begin{figure}[H]
	\centering 
	\begin{tabular}{rccccc}
		\rotatebox[origin=l]{90}{~~~~$K=\mathbb{R}^{2}_{+}$} &
		\hspace{-12pt}\includegraphics[scale=0.2]{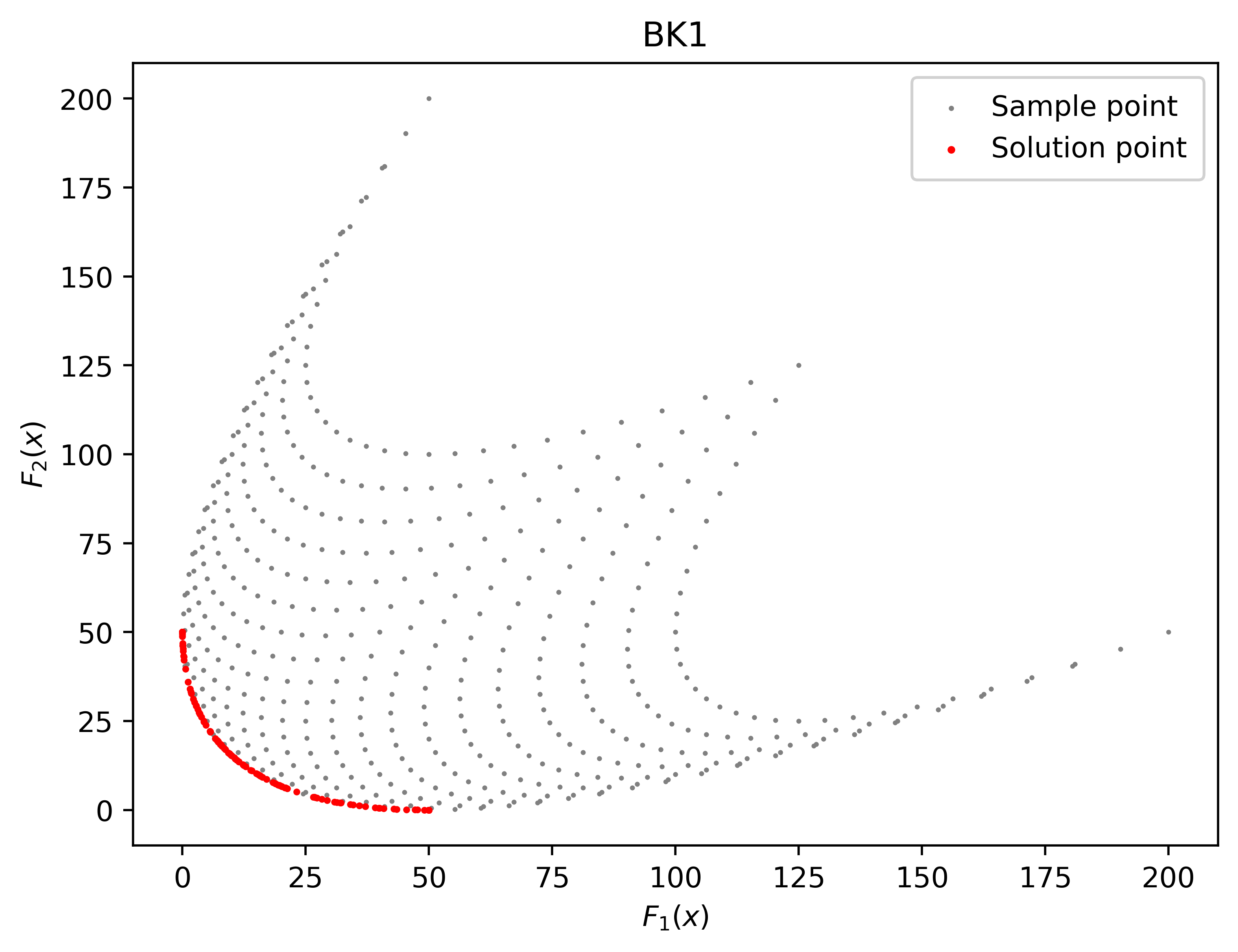} & \hspace{-12pt} \includegraphics[scale=0.2]{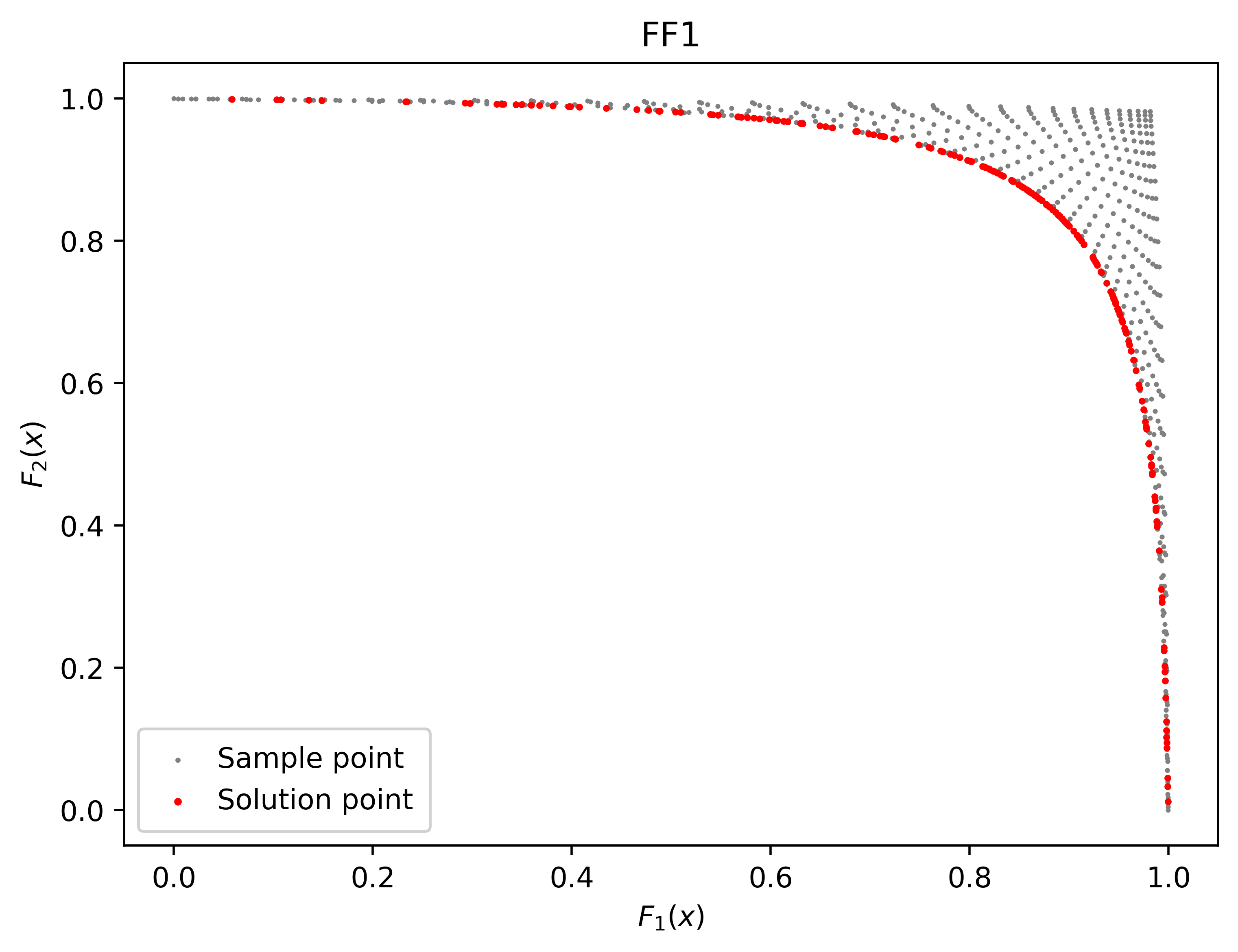}&\hspace{-12pt}\includegraphics[scale=0.2]{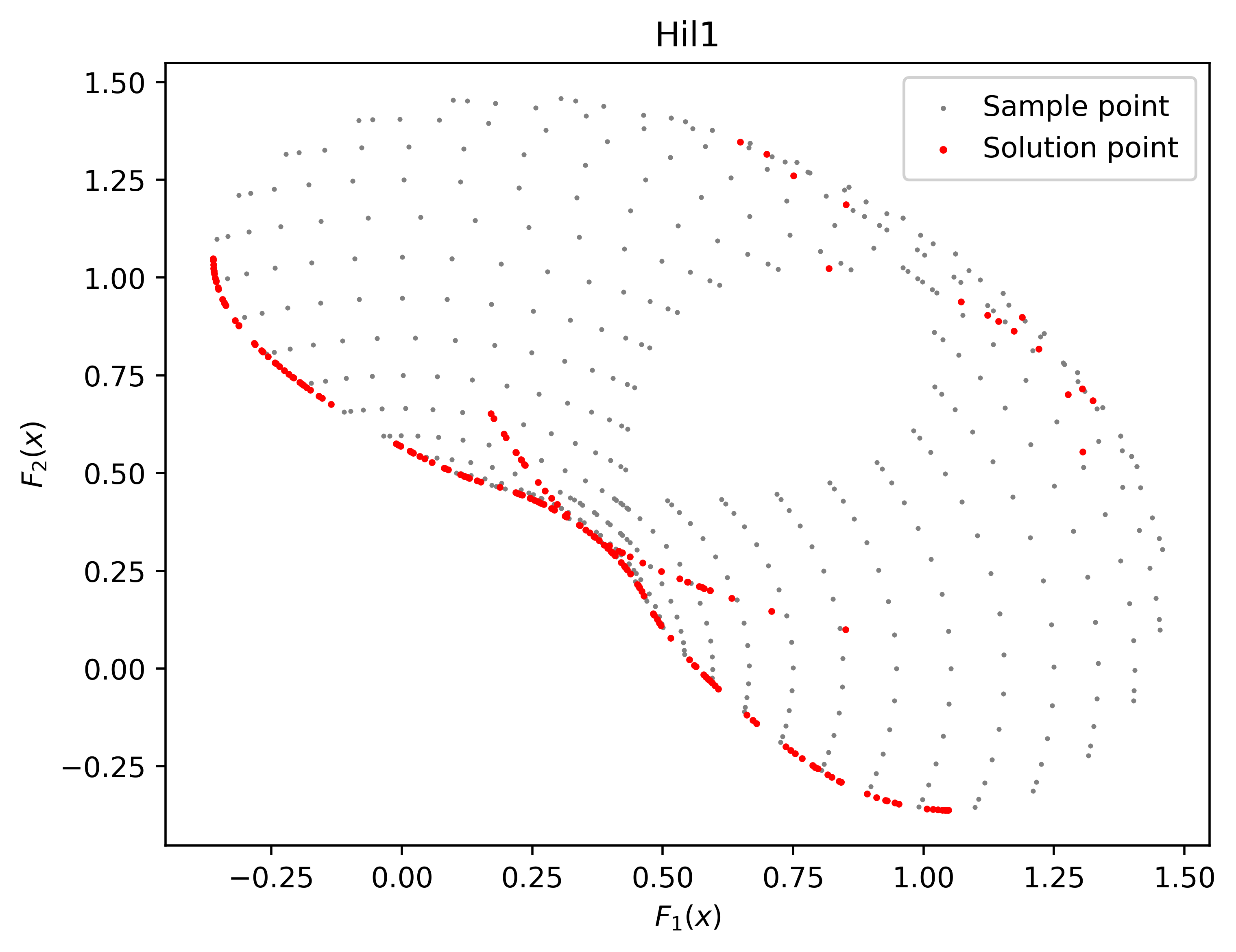}&\hspace{-12pt}\includegraphics[scale=0.2]{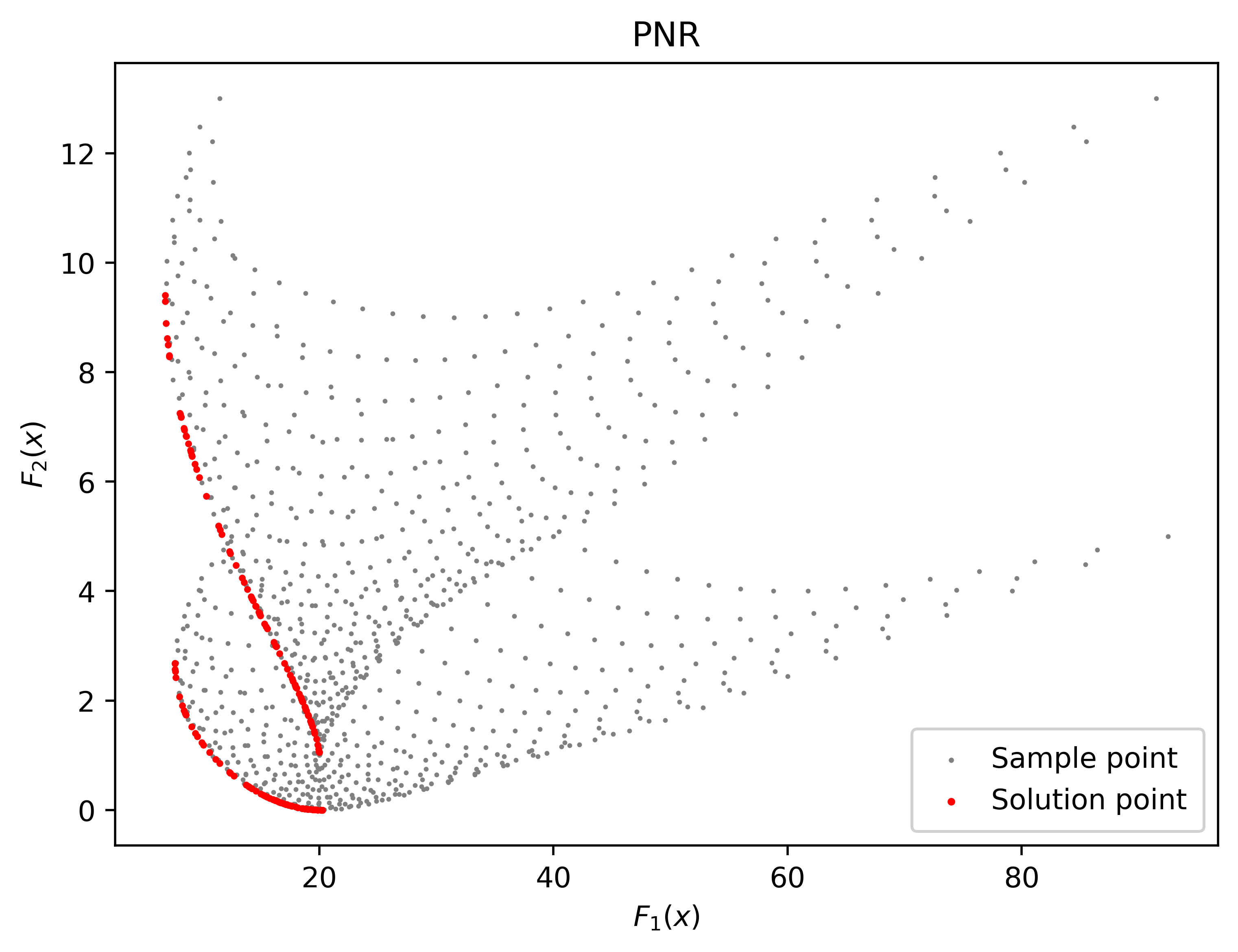}&\hspace{-12pt}\includegraphics[scale=0.2]{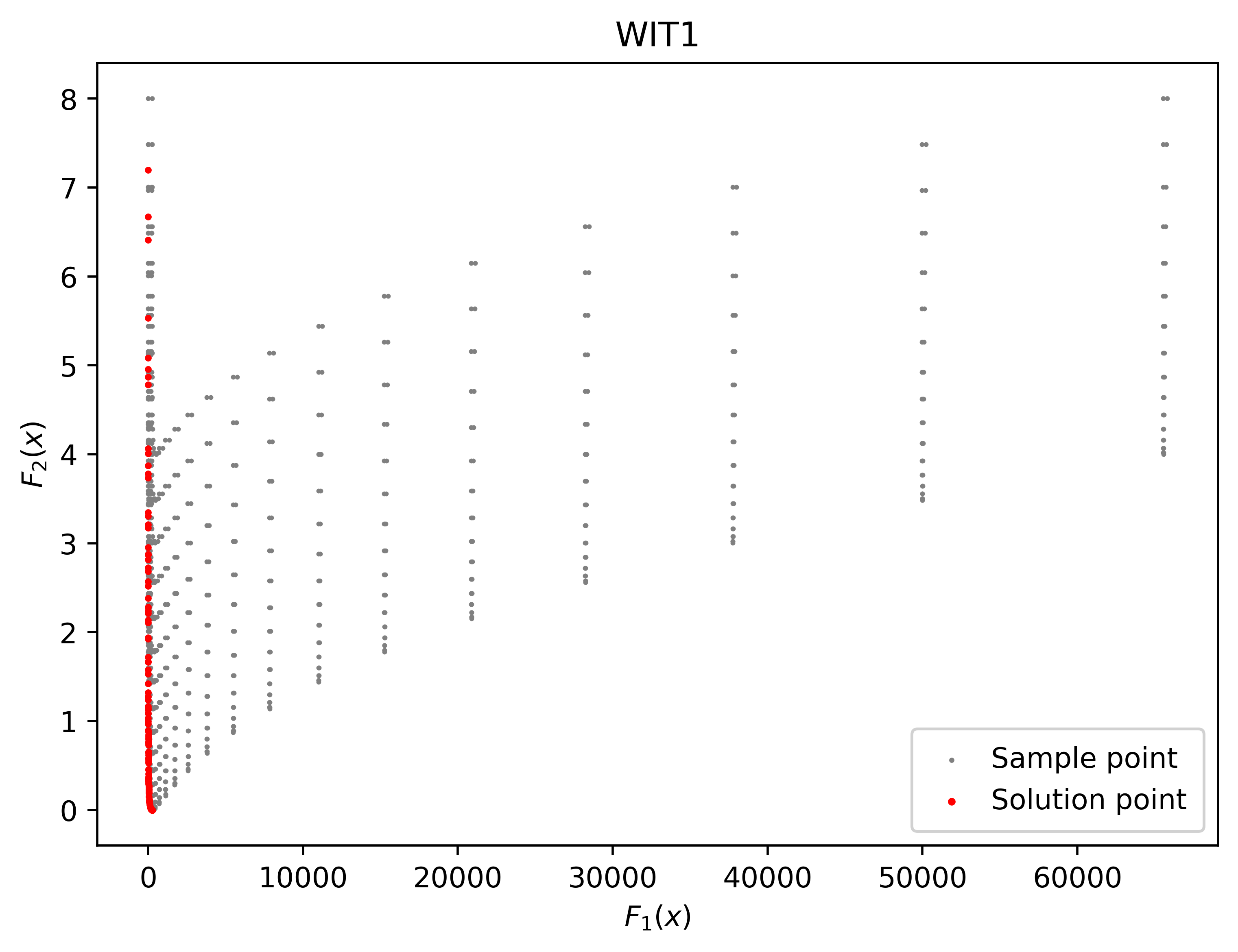}\\
		
		\rotatebox[origin=l]{90}{~~~~$K=K_{1}$} &
		\hspace{-12pt}\includegraphics[scale=0.2]{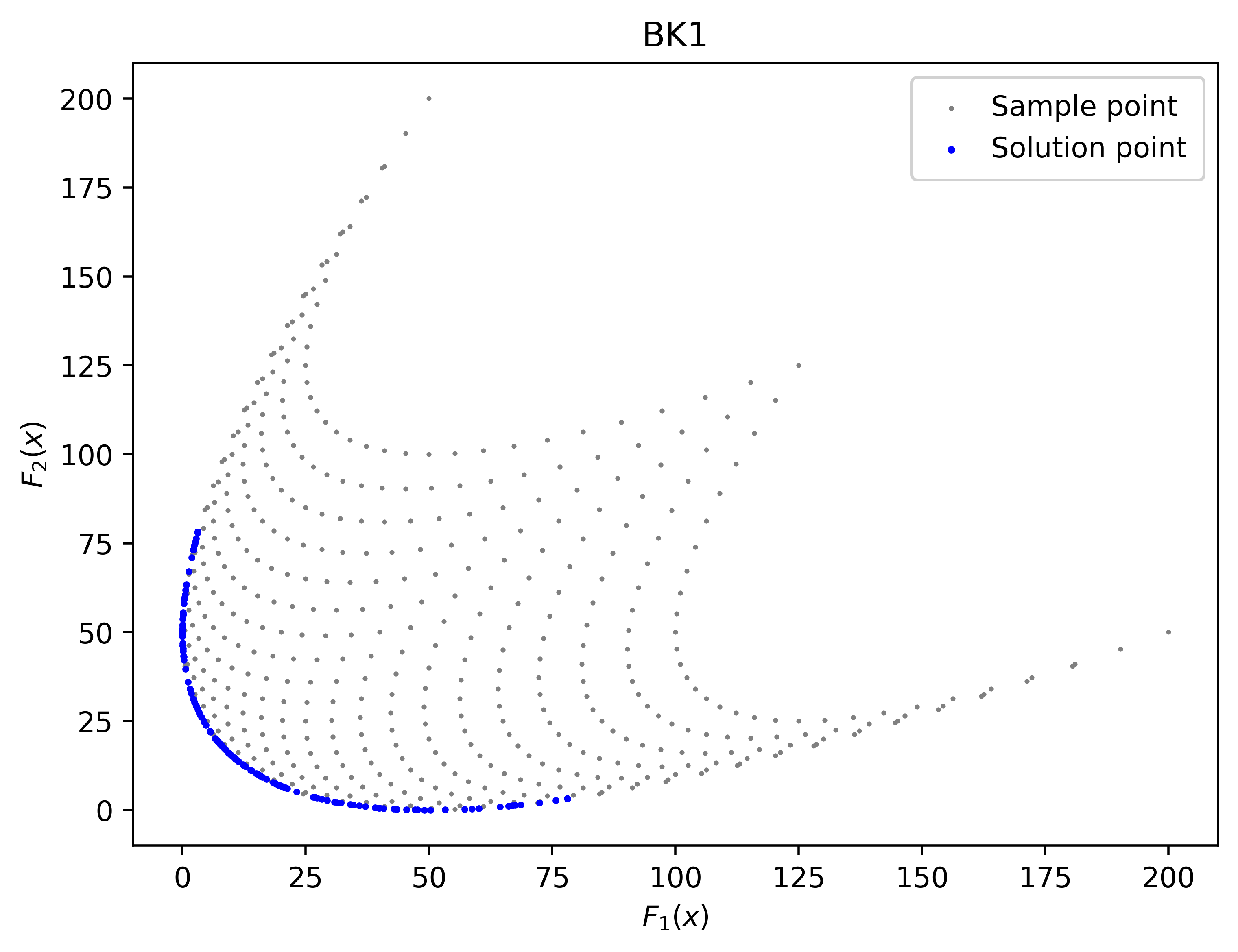} & \hspace{-12pt} \includegraphics[scale=0.2]{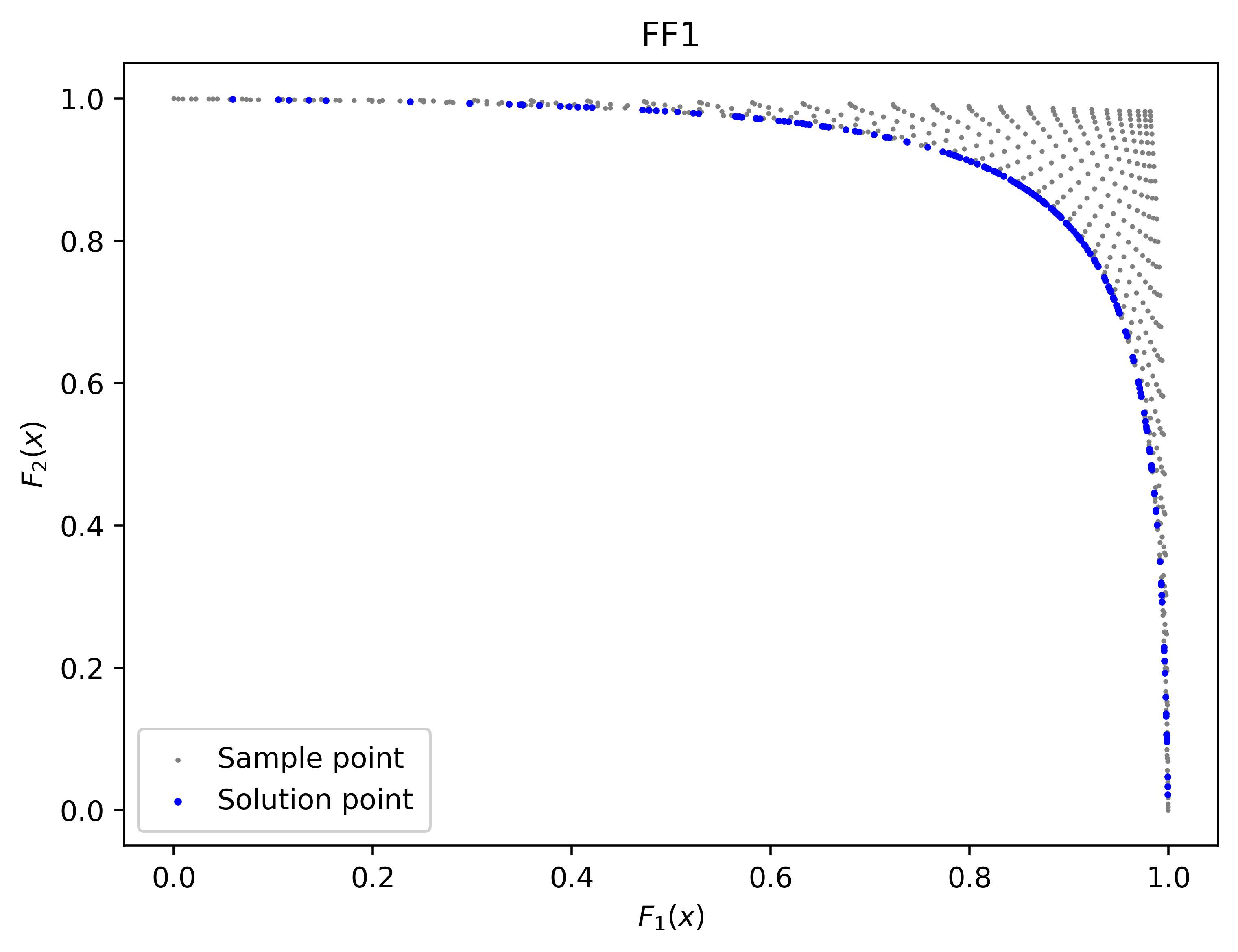}&\hspace{-12pt}\includegraphics[scale=0.2]{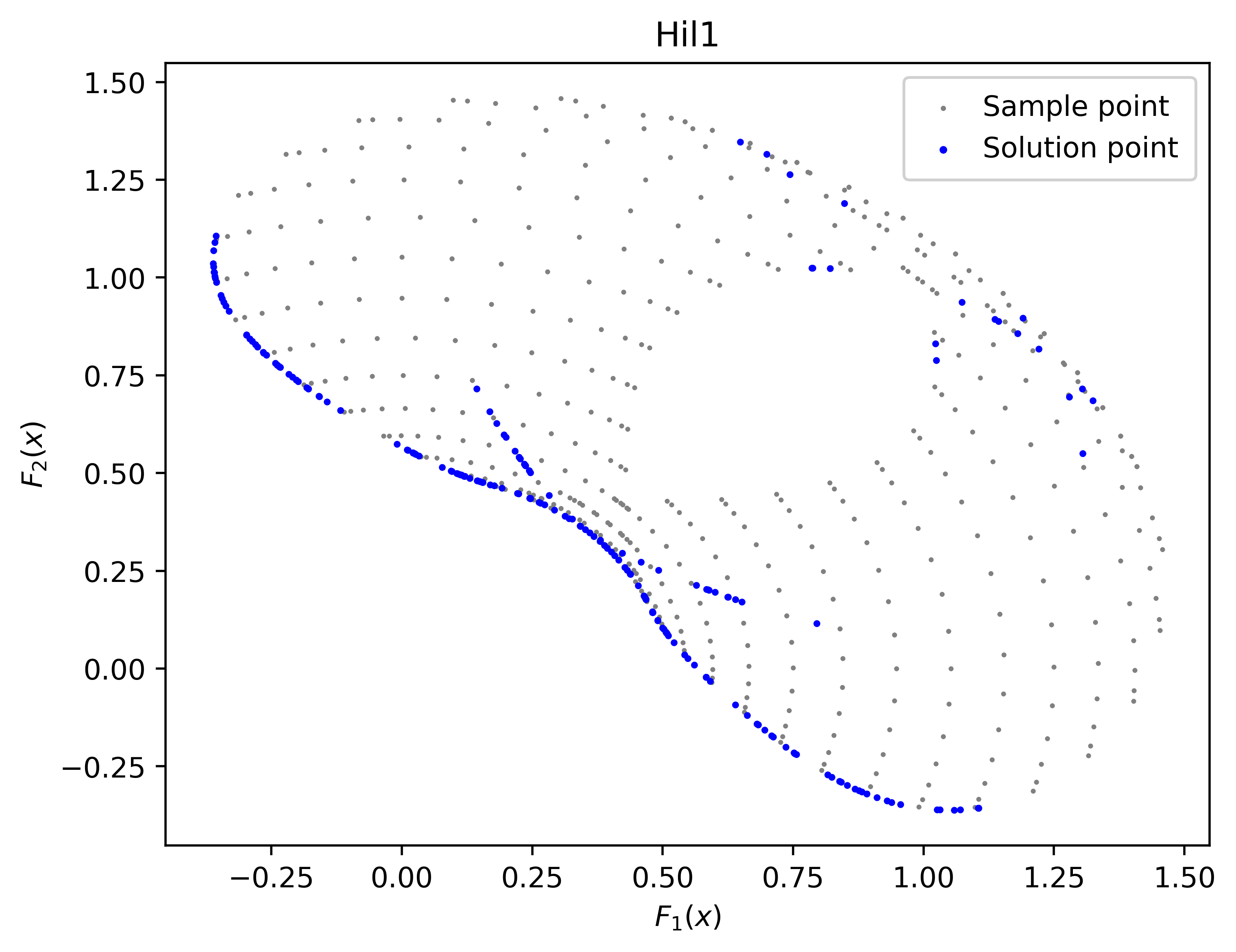}&\hspace{-12pt}\includegraphics[scale=0.2]{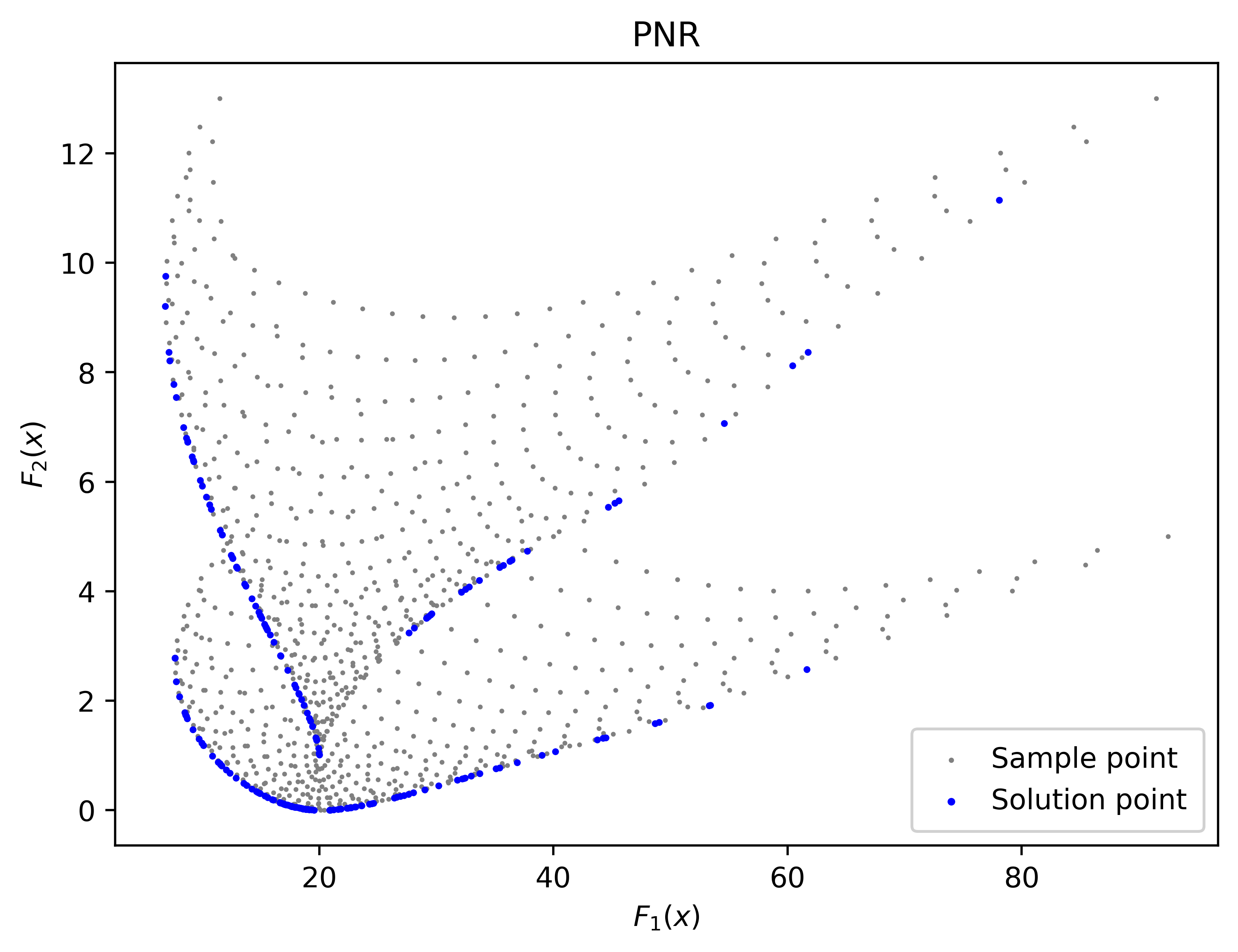}&\hspace{-12pt}\includegraphics[scale=0.2]{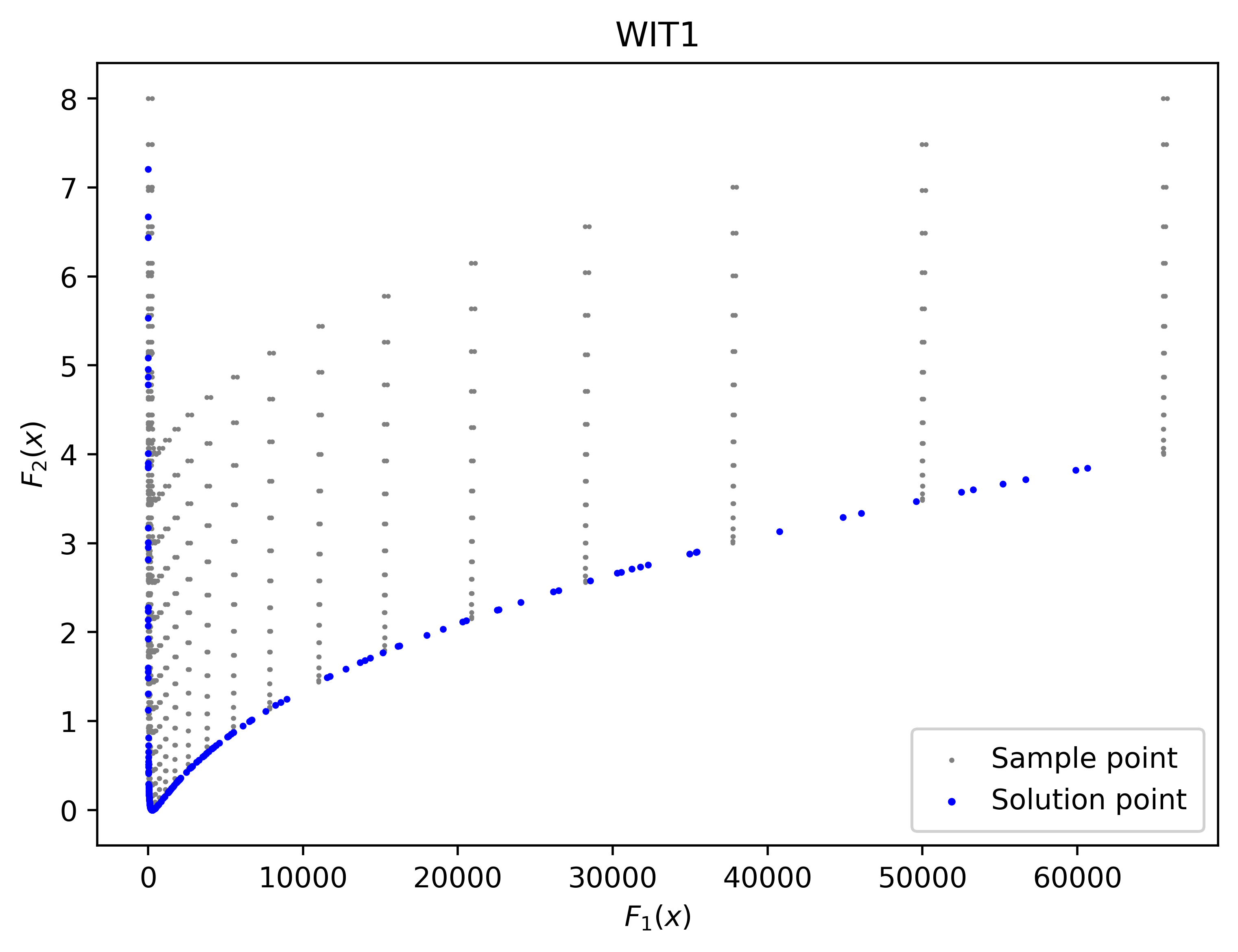}\\
		
		\rotatebox[origin=l]{90}{~~~~$K=K_{2}$} &
		\hspace{-12pt}\includegraphics[scale=0.2]{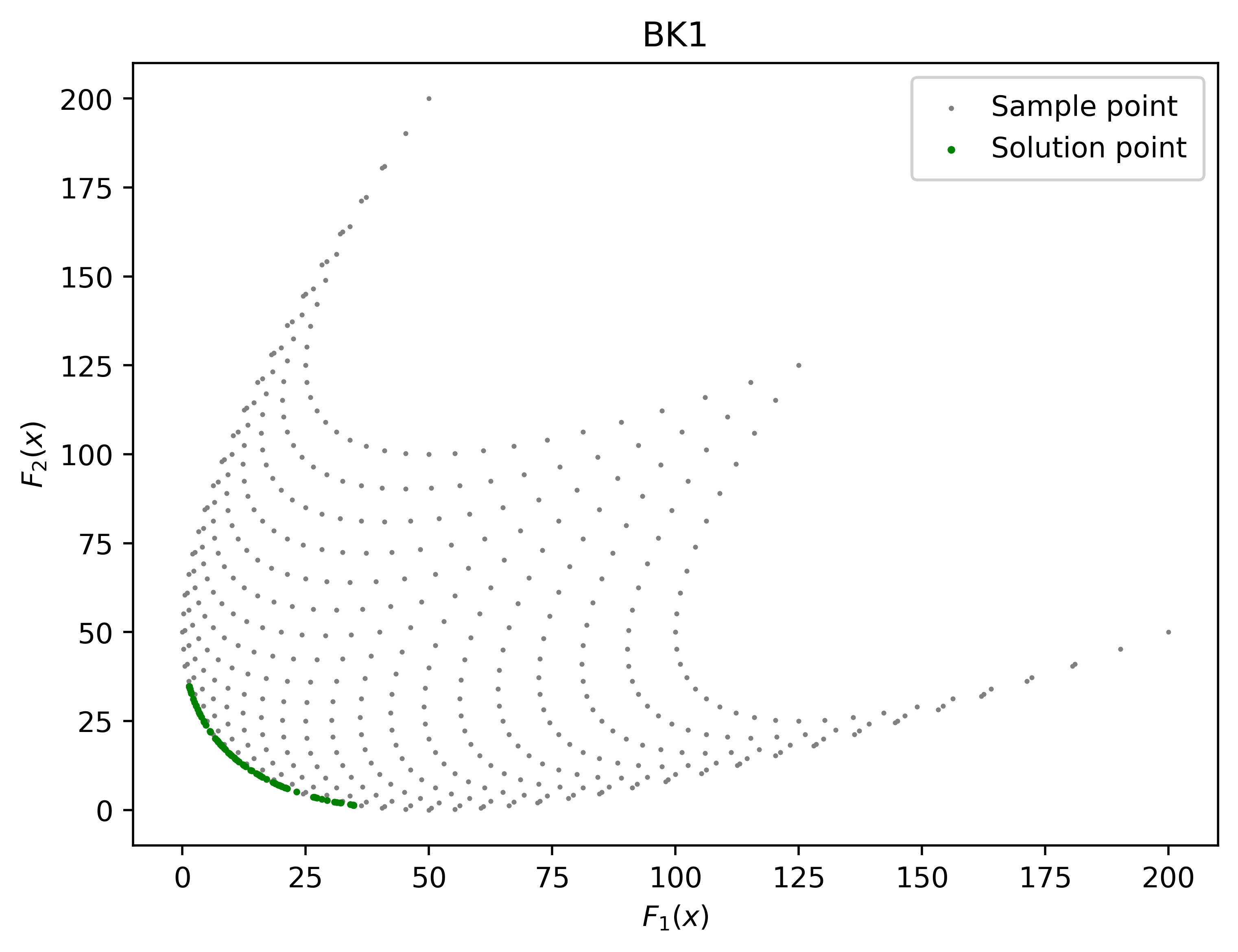} & \hspace{-12pt} \includegraphics[scale=0.2]{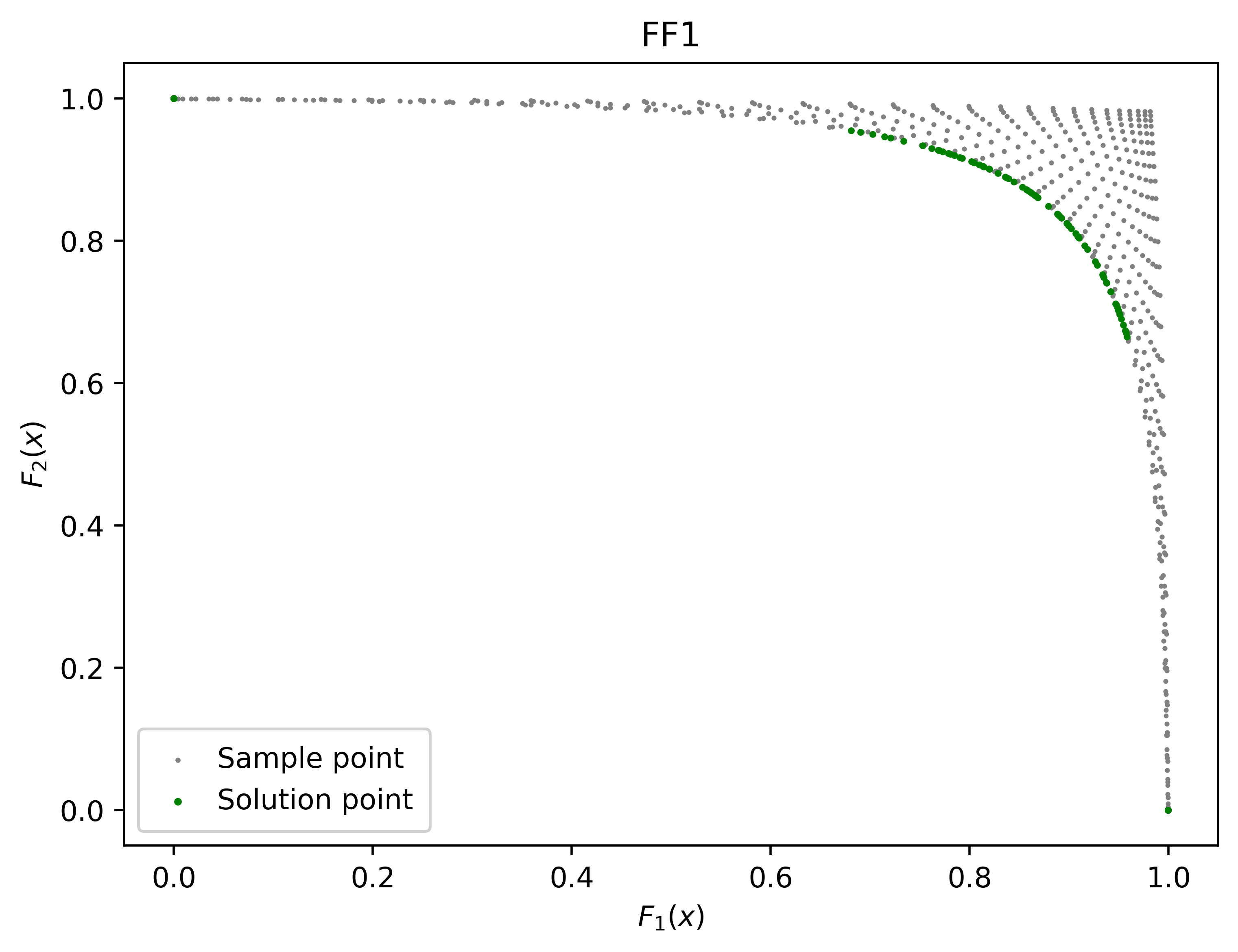}&\hspace{-12pt}\includegraphics[scale=0.2]{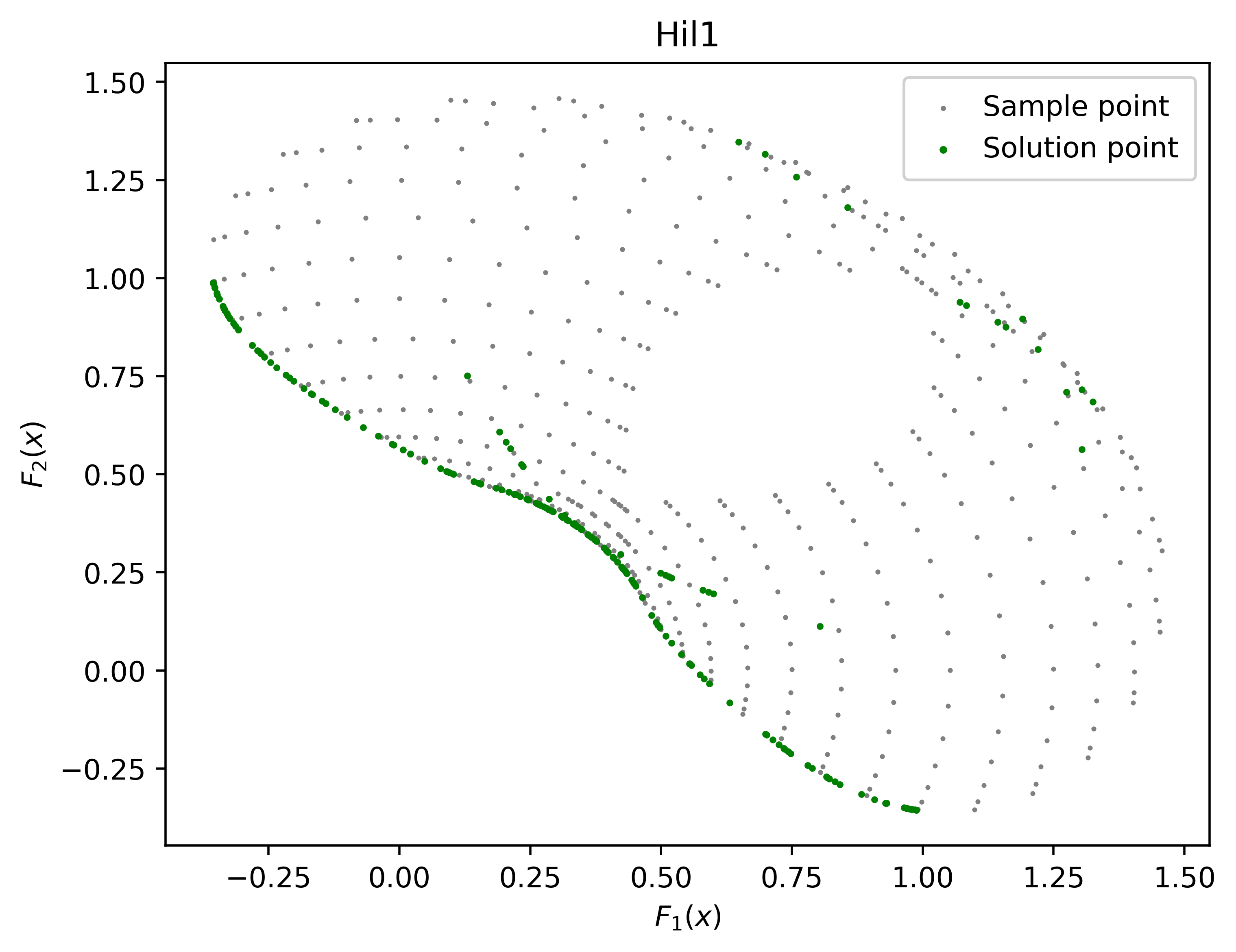}&\hspace{-12pt}\includegraphics[scale=0.2]{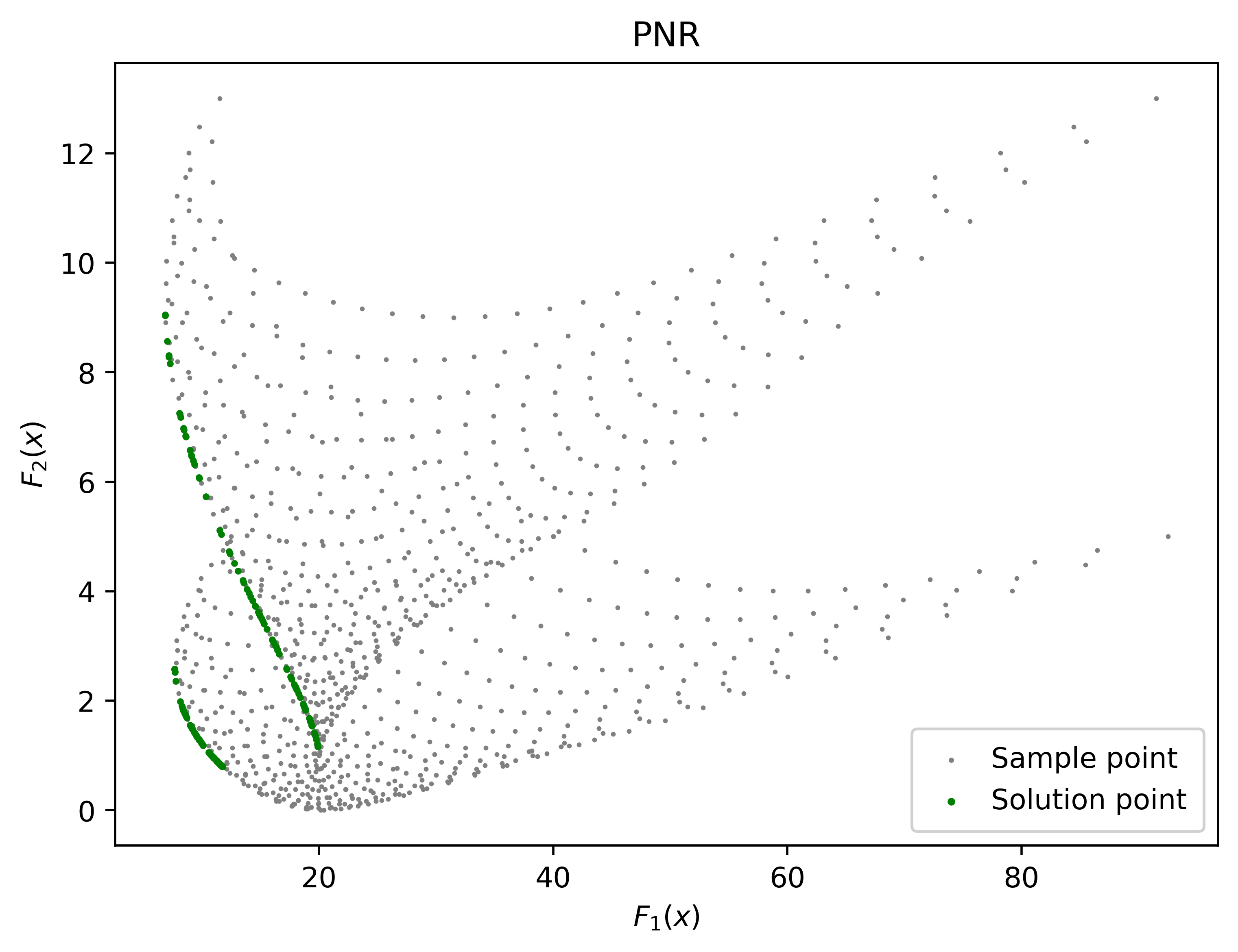}&\hspace{-12pt}\includegraphics[scale=0.2]{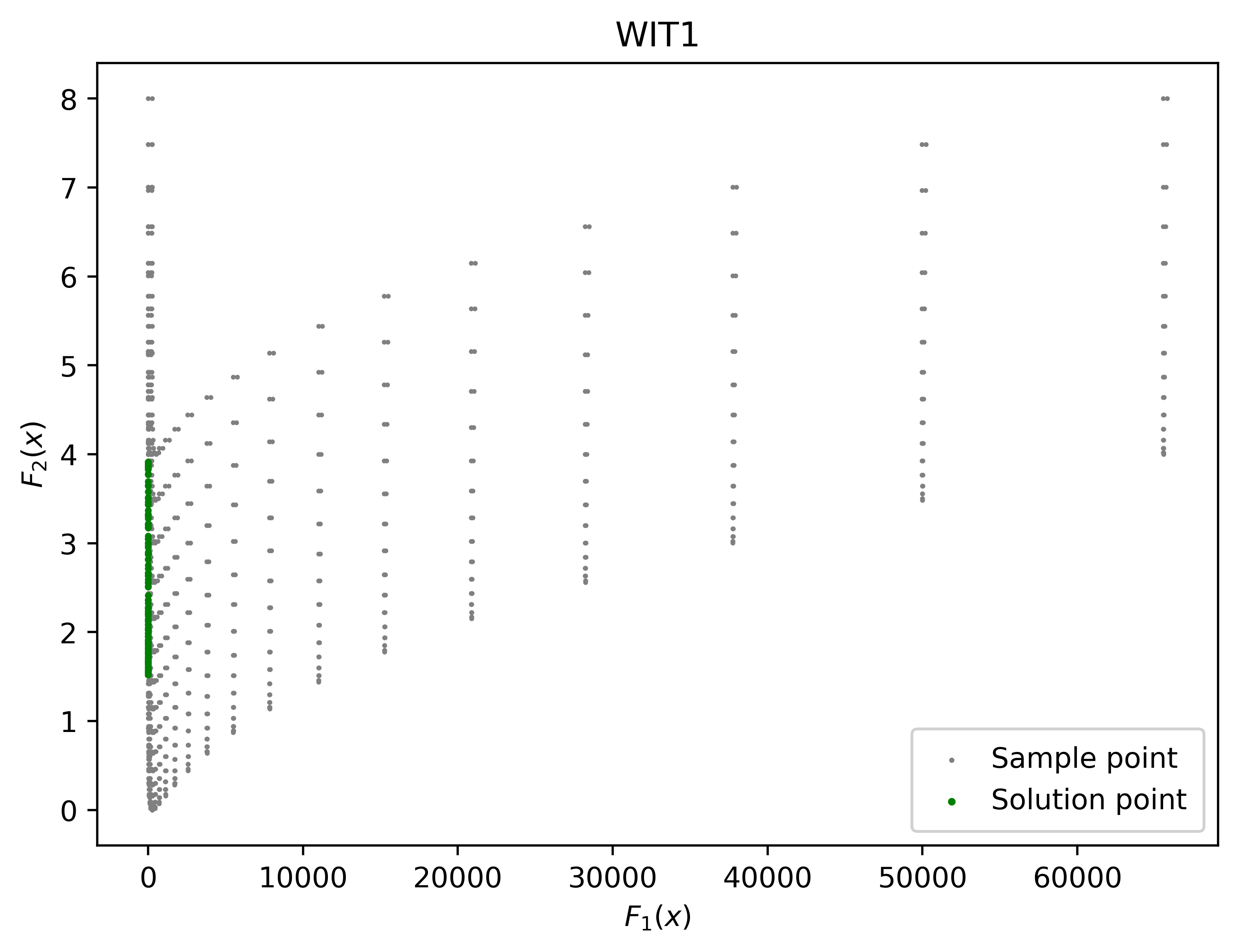}\\
		
		~~~~&	(a) BK1 &(b) FF1 &(c) Hil1 & (d) PNR & (e) WIT1\\
	\end{tabular}
	\caption{Numerical results in value space obtained by BBDVO for problems BK1, FF1, Hil1, PNR and WIT1 with $K=\mathbb{R}^{2}_{+}$, $K=K_{1}$ and $K=K_{2}$, respectively.}
	\label{f1}
\end{figure}

\par For test problems with different partial orders, the number of average iterations (iter), number of average function evaluations (feval), and average CPU time (time(ms)) of the different algorithms are listed in Tables \ref{tab2}, \ref{tab3} and \ref{tab4}, respectively. We conclude that BBDVO outperforms SDVO and EDVO, especially for problems DD1 and Imbalance1. For SDVO, its performance is sensitive to the choice of transform matrix, changing the transform matrix in subproblem cannot improve the performance on all test problems. Naturally, a question arises that how to choose an appropriate transform matrix for a specific test problem in SDVO. It is worth noting that BBDVO can be viewed as SDVO with variable transform matrices ($\Lambda^{k}A$ is a transform matrix of $K$) and thus enjoys promising performance on these test problems. This provides a positive answer to the question. EDVO can also be viewed as SDVO with variable transform matrices, it generates descent directions with norm less than $1$ (the minimizer of subproblem is the minimal norm element of the convex hull of some unit vectors), decelerating the convergence in large-scale problems (the initial point may be far from the Pareto set). Figs. \ref{fk0}, \ref{fk1} and \ref{fk2} present the performance profiles based on iterations, function evaluations, and CPU time. The results confirm that the proposed BBDVO significantly outperforms SDVO and EDVO.
\par Fig. \ref{f1} plots the final points obtained by BBDVO on problems BK1, FF1, Hil1, PNR and WIT1 with $K=\mathbb{R}^{2}_{+}$, $K=K_{1}$ and $K=K_{2}$, respectively. We can observe that enlarging the partial order cone reduces the number of obtained Pareto critical points, especially in the long tail regions, where improving one objective function slightly can sacrifice the others greatly. As a result, we can use an order cone containing the non-negative orthant in real-world MOPs to obtain Pareto points with a better trade-off. 

\section{Conclusions}
In this paper, we develop a unified framework and convergence analysis of descent methods for VOPs from a majorization-minimization perspective. We emphasize that the convergence rate of a descent method can be improved by narrowing the surrogate functions. By changing the base in subproblems, we elucidate that choosing a tighter surrogate function is equivalent to selecting an appropriate base of the dual cone. From the majorization-minimization perspective, we employ Barzilai-Borwein method to narrow the local surrogate functions and propose a Barzilai-Borwein descent method for VOPs polyhedral cone. The proposed method is not sensitive to the choice of transform matrix, which affects the performance of SDVO. Numerical experiments confirm the efficiency of the proposed method.
\par From a majoriztion-minimization perspective, we also rediscover the preconditioned Barzilai-Borwein method for MOPs. This highlights the versatility of the majorization-minimization principle as a powerful framework for designing novel algorithms in vector optimization. In future work, it is worth analyzing proximal gradient methods and high-order methods for VOPs within the majorization-minimization framework. Additionally, exploring solution methods for VOPs with non-polyhedral cones presents an intriguing avenue for further research.

\section*{Acknowledgement}
This work was funded by the National Key Research and Development Program of China [grant number 2023YFA1011504]; the Major Program of the National Natural Science Foundation of China [grant numbers 11991020, 11991024]; the Key Program of the National Natural Science Foundation of China [grant number 12431010]; the General Program of the National Natural Science Foundation of China [grant number 12171060]; NSFC-RGC (Hong Kong) Joint Research Program [grant number 12261160365]; the Team Project of Innovation Leading Talent in Chongqing [grant number CQYC20210309536]; the Natural Science Foundation of Chongqing [grant numbers ncamc2022-msxm01, CSTB2024NSCQ-LZX0140]; Major Project of Science and Technology Research Rrogram of Chongqing Education Commission of China [grant number KJZD-M202300504]; the Chongqing Postdoctoral Research Project Special Grant [grant number 2024CQBSHTB1007];
the Science and Technology Research Program of Chongqing Education commission of China [grant number KJQN202400520] and Foundation of Chongqing Normal University [grant numbers 22XLB005, 22XLB006].

\section*{Data Availability}
The data that support the findings of this study are available from the first author, [Jian Chen], upon reasonable request.

\begin{spacing}{0.95}
	\bibliographystyle{model5-names}
	\biboptions{authoryear}
	\bibliography{references}

@article{MA2004,
  title={Survey of multi-objective optimization methods for engineering},
  author={Marler, R. T. and Arora, J. S.},
  year={2004},
  journal={Structural and Multidisciplinary Optimization},
  volume={26},
  number={6},
  pages={369--395},
  url={https://doi.org/10.1007/s00158-003-0368-6}
}

@inproceedings{TC2007,
  title={Applications of multi-objective evolutionary algorithms in economics and finance: {A} survey},
  author={Tapia, M. G. C. and Coello, C. A. C.},
  booktitle={2007 IEEE Congress on Evolutionary Computation},
  pages={532--539},
  year={2007},
  url={https://doi.org/10.1109/CEC.2007.4424516}
}

@article{FW2014,
  title={Robust multiobjective optimization \& applications in portfolio optimization},
  author={Fliege, J. and Werner, R.},
  year={2014},
  journal={European Journal of Operational Research},
  volume={234},
  number={2},
  pages={422--433},
  url={https://doi.org/10.1016/j.ejor.2013.10.028}
}

@article{E1984,
  title={Overview of techniques for solving multiobjective mathematical programs},
  author={Evans, G.},
  year={1984},
  journal={Management Science},
  volume={30},
  number={11},
  pages={1268--1282},
  url={https://doi.org/10.1287/mnsc.30.11.1268}
}

@article{LW1992,
  title={Interactive multiobjective analysis and assimilative capacity-based ocean disposal decisions},
  author={Leschine, T. M. and Wallenius, H. and Verdini, W. A.},
  year={1992},
  journal={European Journal of Operational Research},
  volume={56},
  number={2},
  pages={278-289},
  url={https://doi.org/10.1016/0377-2217(92)90228-2}
}

@article{LMS2017,
  title={Nonconvex nonsmooth optimization via convex–nonconvex majorization–minimization},
  author={Lanza, A. and Morigi, S. and Selesnick, I. and Sgallari, F. },
  year={2017},
  journal={Numerische Mathematik},
  volume={136},
  pages={343-381},
  url={https://doi.org/10.1007/s00211-016-0842-x}
}

@InProceedings{KWM2022,
  title = {Minimization by Incremental Stochastic Surrogate Optimization for Large Scale Nonconvex Problems},
  author = {Karimi, Belhal and Wai, Hoi-To and Moulines, Eric and Li, Ping},
  booktitle = {Proceedings of The 33rd International Conference on Algorithmic Learning Theory},
  pages = {606--637},
  year = {2022},
  editor = {Dasgupta, Sanjoy and Haghtalab, Nika},
  volume = {167},
  series = {Proceedings of Machine Learning Research},
  publisher = {PMLR},
  url = {https://proceedings.mlr.press/v167/karimi22a.html},
}

@article{Landeros2023MM,
  author = {Landeros, A. and Xu, J. and Lange, K.},
  title = {MM Optimization: Proximal Distance Algorithms, Path Following, and Trust Regions},
  journal = {Proceedings of the National Academy of Sciences},
  year = {2023},
  volume = {120},
  number = {27},
  pages = {e2303168120},
  doi = {10.1073/pnas.2303168120},
}

@inproceedings{SK2018,
 author = {Sener, O. and Koltun, V.},
 booktitle = {Advances in Neural Information Processing Systems},
 editor = {S. Bengio and H. Wallach and H. Larochelle and K. Grauman and N. Cesa-Bianchi and R. Garnett},
 pages = {},
 publisher = {Curran Associates, Inc.},
 title = {Multi-task learning as multi-objective optimization},
 url = {https://proceedings.neurips.cc/paper_files/paper/2018/file/432aca3a1e345e339f35a30c8f65edce-Paper.pdf},
 volume = {31},
 year = {2018}
}

@article{GLP2022,
  title={A study of {L}iu-{S}torey conjugate gradient methods for vector optimization},
  author={Gon{\c{c}}alves, M. L. N. and Lima, F. S. and Prudente, L. F.},
  journal={Applied Mathematics and Computation},
  volume={425},
  pages={127099},
  year={2022},
   url={https://doi.org/10.1016/j.amc.2022.127099}

}

@inproceedings{YL2021,
 author = {Ye, F. Y. and Lin, B. J. and Yue, Z. X. and Guo, P. X. and Xiao, Q. and Zhang, Y.},
 booktitle = {Advances in Neural Information Processing Systems},
 editor = {M. Ranzato and A. Beygelzimer and Y. Dauphin and P.S. Liang and J. Wortman Vaughan},
 pages = {21338--21351},
 publisher = {Curran Associates, Inc.},
 title = {Multi-objective meta learning},
 url = {https://proceedings.neurips.cc/paper_files/paper/2021/file/b23975176653284f1f7356ba5539cfcb-Paper.pdf},
 volume = {34},
 year = {2021}
}

@book{R1970,
author = {Rockafellar, R. T.},
title = {Convex {Analysis}},
publisher = {Princeton University Press},
year = {1970}
}

@article{FS2000,
  title={Steepest descent methods for multicriteria optimization},
  author={Fliege, J. and Svaiter, B. F.},
  journal={Mathematical Methods of Operations Research},
  volume={51},
  number={3},
  pages={479--494},
  year={2000},
  url={https://doi.org/10.1007/s001860000043}
}

@article{BI2005,
  title={Proximal methods in vector optimization},
  author={Bonnel, H. and Iusem, A. N. and Svaiter, B. F.},
  journal={SIAM Journal on Optimization},
  volume={15},
  number={4},
  pages={953--970},
  year={2005},
  url={https://doi.org/10.1137/S1052623403429093}
}

@article{FD2009,
  title={Newton's method for multiobjective optimization},
  author={Fliege, J. and Gra{\~n}a Drummond, L. M. and Svaiter, B. F.},
  journal={SIAM Journal on Optimization},
  volume={20},
  number={2},
  pages={602--626},
  year={2009},
  url={https://doi.org/10.1137/08071692X}
}

@article{QG2011,
  title={Quasi-{N}ewton methods for solving multiobjective optimization},
  author={Qu, S. J. and Goh, M. and Chan, F. T.},
  journal={Operations Research Letters},
  volume={39},
  number={5},
  pages={397--399},
  year={2011},
  url={https://doi.org/10.1016/j.orl.2011.07.008}
}

@article{P2014,
  title={Quasi-{N}ewton's method for multiobjective optimization},
  author={Povalej, {\v{Z}}.},
  journal={Journal of Computational and Applied Mathematics},
  volume={255},
  pages={765--777},
  year={2014},
  url={https://doi.org/10.1016/j.cam.2013.06.045}
}

@article{FV2016,
  title={A method for constrained multiobjective optimization based on {SQP} techniques},
  author={Fliege, J. and Vaz, A. I. F.},
  journal={SIAM Journal on Optimization},
  volume={26},
  number={4},
  pages={2091--2119},
  year={2016},
  url={https://doi.org/10.1137/15M1016424}
}

@article{CL2016,
  title={Trust region globalization strategy for the nonconvex unconstrained multiobjective optimization problem},
  author={Carrizo, G.  A. and Lotito, P. A. and Maciel, M. C.},
  journal={Mathematical Programming},
  volume={159},
  number={1},
  pages={339--369},
  year={2016},
  url={https://doi.org/10.1007/s10107-015-0962-6}
}

@article{MP2018,
  title={A stochastic multiple gradient descent algorithm},
  author={Mercier, Q. and Poirion, F. and D{\'e}sid{\'e}ri, J. A.},
  journal={European Journal of Operational Research},
  volume={271},
  number={3},
  pages={808--817},
  year={2018},
  url={https://doi.org/10.1016/j.ejor.2018.05.064}
}

@article{LP2018,
author = {Lucambio P\'{e}rez, L. R. and Prudente, L. F.},
title = {Nonlinear conjugate gradient methods for vector optimization},
journal = {SIAM Journal on Optimization},
volume = {28},
number = {3},
pages = {2690-2720},
year = {2018},
URL = { https://doi.org/10.1137/17M1126588}
}

@article{MP2019,
  title={Extension of {Z}outendijk method for solving constrained multiobjective optimization problems},
  author={Morovati, V and Pourkarimi, L},
  journal={European Journal of Operational Research},
  volume={273},
  number={1},
  pages={44--57},
  year={2019},
  url={https://doi.org/10.1016/j.ejor.2018.08.018}
}

@article{BB1988,
  title={Two-point step size gradient methods},
  author={Barzilai, J. and Borwein, J. M.},
  journal={IMA Journal of Numerical Analysis},
  volume={8},
  number={1},
  pages={141--148},
  year={1988},
  url={https://doi.org/10.1093/imanum/8.1.141}
}

@inproceedings{JO2001,
  title={Dynamic weighted aggregation for evolutionary multi-objective optimization: {W}hy does it work and how?},
  author={Jin, Y. and Olhofer, M. and Sendhoff, B.},
  booktitle={Proceedings of the Genetic and Evolutionary Computation Conference},
  pages={1042--1049},
  year={2001}
}

@phdthesis{W2012,
  title={Numerical algorithms for the treatment of parametric multiobjective optimization problems and applications},
  author={Witting, K.},
  year={2012},
  school={Paderborn, Universit{\"a}t Paderborn, Diss., 2012}
}

@article{D1999,
  title={Multi-objective genetic algorithms: {P}roblem difficulties and construction of test problems},
  author={Deb, K.},
  journal={Evolutionary Computation},
  volume={7},
  number={3},
  pages={205--230},
  year={1999},
  url={https://doi.org/10.1162/evco.1999.7.3.205}
}

@InProceedings{PN2006,
author={Preuss, M.
and Naujoks, B.
and Rudolph, G.},
editor={Runarsson, Thomas Philip
and Beyer, Hans-Georg
and Burke, Edmund
and Merelo-Guerv{\'o}s, Juan J.
and Whitley, L. Darrell
and Yao, Xin},
title={Pareto set and {EMOA} behavior for simple multimodal multiobjective functions},
booktitle={Parallel Problem Solving from Nature - PPSN IX},
year={2006},
publisher={Springer Berlin Heidelberg},
address={Berlin, Heidelberg},
pages={513--522},
url={https://doi.org/10.1007/11844297_52}
}

@article{DD1998,
  title={Normal-boundary intersection: {A} new method for generating the {P}areto surface in nonlinear multicriteria optimization problems},
  author={Das, I. and Dennis, J. E.},
  journal={SIAM Journal on Optimization},
  volume={8},
  number={3},
  pages={631--657},
  year={1998},
  url={https://doi.org/10.1137/S1052623496307510}
}

@article{TFY2019,
  title={Proximal gradient methods for multiobjective optimization
and their applications},
  author={Tanabe, H. and Fukuda, E. H. and Yamashita, N.},
  journal={Computational Optimization and Applications},
  volume={72},
  pages={339--361},
  year={2019},
  url={https://doi.org/10.1007/s10589-018-0043-x}
}

@article{TFY2023c,
  title={Convergence rates analysis of a multiobjective proximal gradient method},
  author={Tanabe, H. and Fukuda, E. H. and Yamashita, N.},
  journal={Optimization Letters},
  volume={17},
  pages={333--350},
  year={2023},
  url={https://doi.org/10.1007/s11590-022-01877-7}
}

@article{CTY2022b,
  title={Convergence rates analysis of interior {B}regman gradient method for vector optimization problems},
  author={Chen, J. and Tang, L. P. and Yang, X. M.},
  journal={arXiv preprint arXiv:2206.10070},
  year={2022}
}

@article{CTY2023a,
title = {A {B}arzilai-{B}orwein descent method for multiobjective optimization problems},
author={Chen, J. and Tang, L. P. and Yang, X. M.},
journal = {European Journal of Operational Research},
volume = {311},
number = {1},
pages = {196-209},
year = {2023},
url = {https://https://doi.org/10.1016/j.ejor.2023.04.022},
}

@article{ZDH2019,
  title={Convergence rate of gradient descent method for multi-objective optimization},
  author={Zeng, L. Y. and Dai, Y. H. and Huang, Y. K.},
  journal={Journal of Computational Mathematics},
  volume={37},
  number={5},
  pages={689--703},
  year={2019},
  url={https://doi.org/10.4208/jcm.1808-m2017-0214}
}

@article{FVV2019,
  title={Complexity of gradient descent for multiobjective optimization},
  author={Fliege, J. and Vaz, A. I. F. and Vicente, L. N.},
  journal={Optimization Methods and Software},
  volume={34},
  number={5},
  pages={949--959},
  year={2018},
  url={https://doi.org/10.1080/10556788.2018.1510928}
}

@article{L2024,
title = {Convergence and complexity guarantees for a wide class of descent algorithms in nonconvex multi-objective optimization},
journal = {Operations Research Letters},
volume = {54},
pages = {107115},
year = {2024},
 url = {https://doi.org/10.1016/j.orl.2024.107115},
author = {Matteo Lapucci},
}

@article{S1958,
  title={On general minimax theorems},
  author={Sion, M.},
  journal={Pacific Journal of Mathematics},
  volume={8},
  number={1},
  pages={171--176},
  year={1958},
  url={https://doi.org/10.2140/PJM.1958.8.171}
}

@article{BK1,
  title={A review of multiobjective test problems and a
scalable test problem toolkit},
  author={Huband, S. and Hingston, P. and Barone, L. and  While, L.},
  journal={IEEE Transactions on Evolutionary Computation},
  volume={10},
  number={5},
  pages={477--506},
  year={2006},
  url={https://doi.org/10.1109/TEVC.2005.861417}
}

@article{Hil1,
  title={ Generalized homotopy approach to multiobjective optimization},
  author={Hillermeier, C.},
  journal={Journal of Optimization Theory and Applications},
  volume={110},
  number={3},
  pages={557--583},
  year={2001},
  url={https://doi.org/10.1023/A:1017536311488}
}

@article{GI2004,
  title={A projected gradient method for vector optimization problems},
  author={Gra{\~n}a Drummond, L. M.  and  Iusem, A. N. },
  journal={Computational Optimization and Applications},
  volume={28},
  number={1},
  pages={5-29},
  year={2004},
url = {https://doi.org/10.1023/B:COAP.0000018877.86161.8b}
}

@article{CTY2023b,
  title={Barzilai-{B}orwein proximal gradient methods for multiobjective composite optimization problems with improved linear convergence},
  author={Chen, J. and Tang, L. P. and Yang, X. M.},
  journal={arXiv preprint arXiv:2306.09797v2},
  year={2023},
url={https://arxiv.org/pdf/2306.09797v2.pdf}
}

@article{CYZ2023,
  title={Conditional gradient method for vector optimization},
  author={Chen, W. and Yang, X. M. and Zhao, Y.},
  journal={Computational Optimization and Applications},
   volume={85},
  pages={857--896},
  year={2023},
  publisher={Springer}
}

@book{J2011,
  title={Vector optimization: theory, applications and extensions},
  author={Jahn, J.},
  year={2011},
  publisher={Berlin: Springer}
}

@article{GS2005,
  title={A steepest descent method for vector optimization},
  author={Gra\~{n}a Drummond, L. M. and  Svaiter, B. F.},
  journal={Journal of Computational and Applied Mathematics},
  volume={175},
  number={2},
  pages={395--414},
  year={2005},
  publisher={Elsevier}
}

@article{GRS2014,
  title={A quadratically convergent {N}ewton method for vector optimization},
  author={Gra\~{n}a Drummond, L. M. and Raupp, F. M. P. and  Svaiter, B. F.},
  journal={Optimization},
  volume={63},
  number={5},
  pages={661--677},
  year={2014},
  publisher={Taylor \& Francis}
}

@article{VOP1,
title = {Equilibrium analysis in financial markets with countably many securities},
journal = {Journal of Mathematical Economics},
volume = {40},
number = {6},
pages = {683-699},
year = {2004},
url = {https://doi.org/10.1016/j.jmateco.2003.06.003},
author = {C.D Aliprantis and M Florenzano and V.F Martins-da-Rocha and R Tourky},
}

@article{VOP2,
title = {General equilibrium analysis in ordered topological vector spaces},
journal = {Journal of Mathematical Economics},
volume = {40},
number = {3},
pages = {247-269},
year = {2004},
url = {https://doi.org/10.1016/j.jmateco.2003.11.004},
author = {Charalambos D. Aliprantis and Monique Florenzano and Rabee Tourky},
}

@article{DM2002,
  author    = {Dolan, E  D and Mor{\'{e}}, J J},
  title     = {Benchmarking optimization software with performance profiles},
  journal   = {Mathematical Programming},
  volume    = {91},
  number    = {2},
  pages     = {201--213},
  year      = {2002}
}

@article{ADN2020,
  title={Follow the bisector: a simple method for multi-objective optimization},
  author={Katrutsa, Alexandr and Merkulov, Daniil and Tursynbek, Nurislam and Oseledets, Ivan},
  journal={arXiv preprint arXiv:2007.06937},
url = {https://arxiv.org/pdf/2007.06937},
  year={2020}
}

@article{M2015,
  title={Incremental majorization-minimization optimization with application to large-scale machine learning},
  author={Mairal, Julien},
  journal={SIAM Journal on Optimization},
  volume={25},
  number={2},
  pages={829--855},
url = {https://doi.org/10.1137/140957639},
  year={2015}
}
\end{spacing}
\end{document}